\documentclass[11pt,reqno]{article}
 
\usepackage{amsfonts,amsmath,amsxtra,amsthm,amssymb,latexsym}
 \usepackage{mathrsfs} 
\usepackage{graphicx}
\usepackage{color} 
  \usepackage{epsfig}
  \usepackage{cite}
 \usepackage{tikz}
 
   \newcommand\dom{\operatorname{dom}}

\DeclareMathOperator{\rad}{rad}
\DeclareMathOperator{\odd}{odd}

\DeclareMathOperator{\sech}{sech}
\DeclareMathOperator{\Span}{span}
 \DeclareMathOperator{\sign}{sign} 
  \DeclareMathOperator{\ran}{Ran}

   \DeclareMathOperator{\eq}{eq} 
    \DeclareMathOperator{\logg}{Log}
      \DeclareMathOperator{\EE}{\mathcal{E}}
 \DeclareMathOperator{\opl}{\emph{\textbf{l}}}
  \DeclareMathOperator{\opt}{\emph{\textbf{t}}} 
 \newcommand{\bb}[1]{\mathbf{#1}}
 \renewcommand{\log}{\textrm{Log}} 

\newtheorem{theorem}{Theorem}[section]
 
 \newtheorem{lemma}[theorem]{Lemma}
 \newtheorem{proposition}[theorem]{Proposition}
  \newtheorem{definition}[theorem]{Definition}
  \newtheorem{remark}[theorem]{Remark}
 \numberwithin{equation}{section}

\begin{author}
{Jaime Angulo Pava and Nataliia Goloshchapova}
\end{author}

\begin{title}
{Extension theory approach in stability of standing waves for NLS equation with point interactions}
\end{title}
\date{}
\begin{document}

\maketitle
 
\centerline{Department of Mathematics,
IME-USP}
 \centerline{Rua do Mat\~ao 1010, Cidade Universit\'aria, CEP 05508-090,
 S\~ao Paulo, SP, Brazil.}
 \centerline{\it angulo@ime.usp.br, nataliia@ime.usp.br}

\numberwithin{equation}{section}

\begin{abstract} 
The aim of this work is to demonstrate the effectiveness of the
extension
theory of symmetric operators in the investigation of the
stability of
standing waves for the nonlinear Schr\"odinger equations with two
types of
nonlinearities (power and logarithmic) and two types of point
interactions
($\delta$- and $\delta'$-) on  a star graph.  
Our approach allows us to overcome the use of variational
techniques in the
investigation of the Morse index for self-adjoint operators with
non-standard
boundary conditions which appear in the stability study. We also
demonstrate
how our method simplifies the proof of the stability results
known for the
NLS equation with point interactions on the line.

\end{abstract}
\section{Introduction}

In the last two decades the study of nonlinear dispersive models
with point
interactions has attracted a lot of attention of mathematicians
and
physicists. In particular, such models appear in nonlinear
optics,
Bose-Einstein condensates (BEC), and quantum graphs (or networks) (see \cite{AdaNoj13a, BK, BeK,
BurCas01,  CM, K, CozFuk08, Mug15, Noj14} and references therein). The prototype
equation for
description of these models is the nonlinear Schr\"odinger
equation (NLS
henceforth)
\begin{equation}\label{NLS0}
i\partial_{t}u(t,x)+ \partial_{x}^2u(t,x) + \left\vert
u(t,x)\right\vert ^{p
-1}u(t,x)=0,\quad x\neq0,
\end{equation}
$(t,x)\in \mathbb{R}\times \mathbb{R}$, $p>1$, with specific
boundary
conditions at $x=0$ induced by a certain impurity or point
interaction. The
most studied are the models with so-called $\delta$- and
$\delta'$-interaction
(see Section 5 for details). Indeed, the Dirac distribution models an impurity or defect
localized at
the origin. Moreover, the NLS-$\delta$ equation on the line can be
viewed as a
prototype model for the interaction of a wide soliton with a
highly localized
potential. In nonlinear optics it models a soliton propagating in
a medium
with a point defect, or interaction of a wide soliton with a much
narrower
one in a bimodal fiber (see \cite{HMZ1}).
Recently
numerous results on the local well-posedness of initial value
problem and
periodic boundary value problem, the
long time behavior of solutions, the existence of stationary
states, blow up
and scattering results (see \cite{AdaNoj13a, AdaNoj13, Ang, A3,
AP, CM, CMR,
CozFuk08, ghw, HMZ1} and references therein) have been obtained.

In this paper we study the existence and the orbital  
stability of  standing waves  
of  model \eqref{NLS0} being extended to  a star graph 
$\mathcal G$, i.e. $N$ 
half-lines  attached to the common vertex $\nu=0$. Namely,  we consider the following nonlinear Schr\"odinger
equations on the
star graph $\mathcal{G}$
 \begin{equation}\label{NLS_graph_ger}
i\partial_t \mathbf{U}(t,x)+\partial_x^2\mathbf{U}(t,x)
+|\mathbf{U}(t,x)|
^{p-1}\mathbf{U}(t,x)=0,
\end{equation}
where $\mathbf{U}(t,x)=(u_j(t,x))_{j=1}^N:\mathbb{R}\times
 \mathbb{R}_+\rightarrow 
\mathbb{C}^N$, and  $p>1$.
 The nonlinearity acts componentwise, i.e.
  $(|\mathbf{U}|^{p-1}\mathbf{U})_j=|u_j|
^{p-1}u_j$, and  
 the function $\bb U$ is assumed to satisfy 
 specific boundary $\delta$- and $
\delta'$-interaction at the vertex $\nu=0$ to
 be defined below (see Subsection 
\ref{sub_2.1}).

 Equation \eqref{NLS_graph_ger} models 
 propagation through junctions in networks. 
The analysis of the  behavior of NLS  equation on networks is not yet 
 fully developed, but it is currently growing 
(see \cite{AdaNoj15, AdaNoj14, AngGol17a, AngGol18, Ard16,
BanIgn14, Noj14}
 and references therein). In particular, models of 
 BEC on graphs/networks is a topic of active research 
 (see \cite{BurCas01, Fid15}).

We recall that the quantum graphs 
(metric graphs equipped with  a linear  Hamiltonian $\bb H$) 
have been a very developed subject in the last couple of decades. They give  simplified models in mathematics, physics, 
chemistry, and engineering, when one considers propagation of 
waves of various types through a quasi one-dimensional 
(e.g. meso- or nanoscale) system that looks like a thin 
neighborhood of a  graph 
(see \cite{BK, BlaExn08, BurCas01, K, Mug15} for details and
references).

Various recent analytical works   
(see \cite{AdaNoj14, AdaNoj15, AngGol17a, AngGol18, Noj14}  
and references therein) deal with special solutions of  
\eqref{NLS_graph_ger}  called 
{\it standing wave solutions}, i.e. the  solutions of  
 the form
$
\mathbf{U}(t,x)= e^{i\omega t} \mathbf{\Phi}(x),
$
with the profile $\mathbf{\Phi}$ satisfying  
$\delta$-interaction conditions defined by 
\eqref{D_alpha} below. In \cite{AdaNoj14} it was established  a complete 
description of the profiles $\mathbf{\Phi}$ for any 
$\alpha\in\mathbb{R}$,   and  the stability 
investigation for  the $N$-tail profile 
(see \eqref{Phi_vect}) under the restriction  
$\alpha<\alpha^*<0$ which comes from the associated 
variational problem. In \cite{AdaNoj15} the restriction 
$\alpha<\alpha^*$ was removed. 
It is worth noting that  the problems of the 
existence and the stability/instability of 
standing waves are far richer and more complicated 
in the case of the  NLS models with point interactions 
on star graphs  than in the case of the NLS equation   
with point interactions on the line. 
 We propose a novel 
short proof of  the  orbital stability of the
$N$-tail profile for any $\alpha<0$ in the 
framework of the  extension theory approach 
(see Remark \ref{alpha_neg}). Moreover, 
we prove the following new result on the orbital 
instability of $N$-bump profile  $\bb \Phi$ in the case 
$\alpha>0$.

 \begin{theorem}\label{main_delta}
Let $\alpha>0$, $1<p<5$, and $\omega>\tfrac{\alpha^2}{N^2}$. Let
also $
\mathbf{\Phi}_{\alpha,\delta}$ be defined by \eqref{Phi_vect},
and the space $
\mathcal{E}(\mathcal{G})$ be defined in notation section.
 Then the following assertions hold.
\begin{itemize}
    
\item[$(i)$] If $1<p\leq 3$, then $e^{i\omega
t}\mathbf{\Phi}_{\alpha,\delta}$
is orbitally unstable in $\mathcal{E}(\mathcal{G})$.
  
\item[$(ii)$] If $3<p<5$, then there exists $\omega_2>
\tfrac{\alpha^2}{N^2}$
such that $e^{i\omega t}\mathbf{\Phi}_{\alpha,\delta}$ is
orbitally unstable in $
\mathcal{E}(\mathcal{G})$ for $\omega>\omega_2$.
\end{itemize}
\end{theorem}

In the case $p\geq 5$ our method does not provide any information
about orbital
stability of $e^{i\omega t}\mathbf{\Phi}_{\alpha,\delta}$ (see
Remark
\ref{alpha_neg}-$(i)$).
Mention also that in the case $N=2$ the above result coincides
with
\cite[Theorem 4]{CozFuk08}.

In Subsection \ref{NLS_delta'_graph}, we prove the following novel
stability
theorem for the standing waves of  the NLS-$\delta'$ equation 
on the star graph with a specific $N$-tail profile $\bb
\Phi_{\lambda,\delta'}$
satisfying  
$\delta'$-interaction conditions \eqref{H_lambda}.

 \begin{theorem}\label{Main}
Let $\lambda<0$, and $\omega >\tfrac{N^2}{\lambda^2}$. Let also
$\mathbf{\Phi}
_{\lambda,\delta'}$ be defined by \eqref{varphi_lam}, and the
space $H^1_{\eq}
(\mathcal{G})$ be defined by
$$H^1_{\eq}(\mathcal{G})=\{(v_j)_{j=1}^N\in
H^1(\mathcal{G}):
v_1(x)=...=v_N(x),\,x\in\mathbb{R}_+\}.$$ Then the following assertions hold.
\begin{itemize}
  \item[$(i)$] Let  $1<p\leq 5$.
  
$1)$ If $\omega < \tfrac{N^2}{\lambda^2}\tfrac{p+1}{p-1}$, then
$e^{it\omega}
\mathbf{\Phi}_{\lambda,\delta'}$ is orbitally stable in
$H^1(\mathcal{G})$.
  
$2)$ If $\omega > \tfrac{N^2}{\lambda^2}\tfrac{p+1}{p-1}$ and $N$
is even, then
$e^{it\omega}\mathbf{\Phi}_{\lambda,\delta'}$ is orbitally
unstable in
$H^1(\mathcal{G})$.
  
\item[$(ii)$] Let $p>5$ and $\omega\neq
\tfrac{N^2}{\lambda^2}\tfrac{p+1}
{p-1}$. Then there exists $\omega^*>\tfrac{N^2}{\lambda^2}$ such
that
$e^{it\omega}\mathbf{\Phi}_{\lambda,\delta'}$ is orbitally
unstable in
$H^1(\mathcal{G})$ for $\omega>\omega^*$, and
$e^{it\omega}\mathbf{\Phi}
_{\lambda,\delta'}$ is orbitally stable in
$H^1_{\eq}(\mathcal{G})$ for $
\omega<\omega^*$.
    \end{itemize}
\end{theorem}
The relative position of $\omega^*$ and
$\tfrac{N^2}{\lambda^2}\tfrac{p+1}{p-1}$
is discussed in Remark \ref{rel_posit}.
In the case $N=2$ the above result coincides with Proposition
6.9(1) (partially)
and Theorem 6.11 in \cite{AdaNoj13a}.
To our knowledge, the NLS-$\delta'$ equation on the
star graph has
never been studied before, and complete description of the
standing waves for such model is unknown (see Remark \ref{gensolu}).

In Section \ref{log}, we study the following NLS equation with
logarithmic
nonlinearity on the star graph $\mathcal{G}$ (NLS-log equation)
 \begin{equation}\label{NLS_log_graph_ger}
i\partial_t \mathbf{U}(t,x)+\partial_x^2\mathbf{U}(t,x) +\bb
U(t,x)\logg|
\mathbf{U}(t,x)|^2=0,
\end{equation} 
where $\mathbf{U}(t,x)=(u_j(t,x))_{j=1}^N:\mathbb{R}\times
\mathbb{R}_+\rightarrow
\mathbb{C}^N$. The nonlinearity acts componentwise, i.e.
$(\mathbf{U}\logg|
\mathbf{U}|^2)_j=u_j\logg|u_j|^2.$ Note that by
$\logg|\mathbf{U}(t,x)|^2$ we mean
$\logg(|\mathbf{U}(t,x)|^2).$

For the NLS-log equation with $\delta$-interaction, we extend the
result from
\cite{Ard16}  (for any $
\alpha<0$) on the orbital stability of the Gaussian $N$-tail
profile $\bb
\Psi_{\alpha,\delta}=(\psi_{\alpha,\delta})_{j=1}^N$ defined by 
\eqref{Phi_vect_log} below.
 In particular, we prove
\begin{theorem}\label{stabil_log_delta}
Let $\omega\in \mathbb{R}$, and $\mathbf{\Psi}_{\alpha,\delta}$
be defined by
\eqref{Phi_vect_log}. Then the standing wave $e^{i\omega
t}\mathbf{\Psi}_{\alpha,
\delta}$ is orbitally stable in 
$W^1_{\mathcal{E}}(\mathcal{G})$ for any $\alpha<0$, and
$e^{i\omega t}
\mathbf{\Psi}_{\alpha,\delta}$ is spectrally unstable for any
$\alpha>0$.
\end{theorem} 
We also show the result analogous to Theorem \ref{Main} for
the NLS-log equation
with  $\delta'$-interaction on $\mathcal{G}$.
\begin{theorem}\label{Main_log}
Let $\lambda<0$, and $\omega \in \mathbb{R}$. Let also
$\mathbf{\Psi}_{\lambda,
\delta'}$ be defined by \eqref{prof_log'}. Then the following
assertions hold.
\begin{itemize}
\item[$(i)$] If $-N<\lambda<0$, then
$e^{it\omega}\mathbf{\Psi}_{\lambda,
\delta'}$  is orbitally stable in $W^1(\mathcal{G})$.
  
\item[$(ii)$] If $\lambda<-N$, then
$e^{it\omega}\mathbf{\Psi}_{\lambda,
\delta'}$ is spectrally  unstable. 
\end{itemize}
\end{theorem} 
\noindent The spaces  $W^1_{\mathcal{E}}(\mathcal{G})$ and  $W^1(\mathcal{G})$ are defined in notation
section.

In Section \ref{line}, we propose a new approach to prove some known
results on the
orbital stability of standing waves for NLS equation \eqref{NLS0}
with
$\delta$- and $\delta'$-interaction on the line.  
It should be noted that the most of previous results (for NLS on
$\mathcal{G}$ and
on the line) are based on either variational methods or the
abstract
stability theory by Grillakis, Shatah and Strauss
\cite{GrilSha87, GrilSha90}
which requires spectral analysis of certain self-adjoint
Schr\"odinger operators.
In particular, investigation of the spectrum of these operators
is based on the
analytic perturbations theory and the variational methods.

Our approach relies on the theory of extensions of symmetric
operators, the
spectral theory of self-adjoint Schr\"odinger operators and the
analytic
perturbations theory. In particular, the extension theory gives
the advantage
to estimate the number of negative eigenvalues (Morse index) of
the linear
Schr\"odinger operator associated with the NLS equation.
 We emphasize that we do not need to study any
variational problem
associated with the equation, and our method does not use any
minimization
properties of the standing waves studied. We would like to
mention the papers
\cite{KaiPel17a, KaiPel17b} where the non-variational methods
were used for the
investigation of the Morse index in the case of the NLS equation
on the star
graph with classical and generalized Kirchhoff conditions at the
vertex. In
particular, the authors elaborated a kind of extension of the
Sturm theory for the
Schr\"odinger operators on the star graph.

The paper is organized as follows. In the Preliminaries (Section
2), we give some
brief description of all the point interactions on the star graph
and explain the
origin of $\delta$- and $\delta'$-interaction. We also review
previous results on
the orbital stability. In Section 3 we discuss NLS equation
\eqref{NLS_graph_ger}
with $\delta$- and $\delta'$-interaction on the star graph
$\mathcal{G}$. In
Section 4, we study NLS-log equation \eqref{NLS_log_graph_ger}
with $\delta$- and
$\delta'$-interaction on $\mathcal{G}$. In Section 5, we briefly
discuss how the
tools of the extension theory can be applied to the stability study
of the NLS
equations with point interactions on the line.

\smallskip

\noindent \textbf{Notation.} 
By $H^1(\mathbb{R})$, $H^2(\mathbb{R}\setminus
\{0\})=H^2(\mathbb{R}_-)\oplus
H^2(\mathbb{R}_+)$ we denote the Sobolev spaces. Denote by
$\mathcal{G}$ the
star graph constituted by $N$ half-lines attached to a common
vertex $\nu=0$. On
the graph we define the spaces 
  \begin{equation*}
  L^p(\mathcal{G})=\bigoplus\limits_{j=1}^NL^p(\mathbb{R}_+),\,\, 
H^1(\mathcal{G})=\bigoplus\limits_{j=1}^NH^1(\mathbb{R}_+),\,\, H^2(\mathcal{G})=\bigoplus\limits_{j=1}^NH^2(\mathbb{R}_+),
\end{equation*}   
$p>1$. For instance, the norm of  $\mathbf{V}=(v_j)_{j=1}^N$ in $L^p(\mathcal{G})$ is defined by $||\bb V||
^p_{L^p(\mathcal{G})}=\sum\limits_{j=1}^N||v_j||^p_{L^p(\mathbb{R}_+)}$.   Depending on the context we will use the following notations for
different
objects: by $||\cdot||$ we denote  the norm in $L^2(\mathbb{R})$ or in $L^2(\mathcal{G})$ (accordingly  
$(\cdot,\cdot)$ denotes the scalar product in $L^2(\mathbb{R})$
or in
$L^2(\mathcal{G})$);  by $||\cdot||_p$ we denote the norm in $L^p(\mathbb{R})$ or
in
$L^p(\mathcal{G})$. 

\noindent 
Denote 
$ 
\mathcal{E}(\mathcal{G})=\{(v_j)_{j=1}^N\in H^1(\mathcal{G}):
v_1(0)=...=v_N(0)\},
$  
and
$$L^2_{k}(\mathcal{G})=\{(v_j)_{j=1}^N\in L^2(\mathcal{G}):
v_1(x)=...=v_{k}
(x),\, \,v_{k+1}(x)=...=v_N(x)\}.$$
In particular,
$\mathcal{E}_k(\mathcal{G})=\mathcal{E}(\mathcal{G})\cap L^2_{k}
(\mathcal{G}),$ and $H^1_k(\mathcal{G})=H^1(\mathcal{G})\cap
L^2_{k}(\mathcal{G}).$
  On $\mathcal{G}$ we define the
following  weighted Hilbert spaces 
\begin{align*} 
W^j(\mathcal{G})=\bigoplus\limits_{j=1}^NW^j(\mathbb{R}_+),\quad
W^j(\mathbb{R}
_+)=\{f\in H^j(\mathbb{R}_+): x^jf\in L^2(\mathbb{R}_+)\}, 
 \end{align*}
$W^j_k(\mathcal{G})= W^j(\mathcal{G})\cap
L^2_k(\mathcal{G})$, $j\in\{1,2\}$, and the Banach space  
 $$ 
 W(\mathcal{G})=\bigoplus_{j=1}^N W(\mathbb{R}_+),\;\;W(\mathbb{R}_+)=\{f\in H^1(\mathbb{R_+}):  |f|^2
\logg|f|^2\in L^1(\mathbb
R_+)\}.
$$  
In particular, $W_{\mathcal{E}}(\mathcal{G})=W(\mathcal{G})\cap
\mathcal{E}
(\mathcal{G})$,
$W^1_{\mathcal{E}}(\mathcal{G})=W^1(\mathcal{G})\cap \mathcal{E}
(\mathcal{G})$, and
$W^1_{\mathcal{E},k}(\mathcal{G})=W^1_{\mathcal{E}}
(\mathcal{G})\cap L^2_k(\mathcal{G})$.

Let $A$ be a densely defined symmetric operator in the Hilbert
space $
\mathcal{H}$. The domain of $A$ is denoted by $\dom(A)$. The
\textit{ deficiency
subspaces} and \textit{ deficiency numbers} of $A$ are defined by $\mathcal{N}_\pm(A):=\ker(A^*\mp
iI)$ and 
$n_\pm(A):=\dim\ker(A^*\mp
iI)$ respectively. The number of negative eigenvalues counting multiplicities
(or \textit{the
Morse index}) is denoted by $n(A)$. The spectrum and the
resolvent set of $A$
are denoted by $\sigma(A)$ and $\rho(A)$ respectively. In
particular, $
\sigma_p(A)$ and $\sigma_c(A)$ denote the point and the
continuous spectrum of
$A$. Let $X$ be an arbitrary Banach space, then its dual is denoted by
$X'$.

\section{Preliminaries}
In this Section we provide a brief description of point
interactions on the star
graph and also discuss previous results on the orbital stability.\subsection{The  NLS equation with point interactions on a star
graph.}\label{sub_2.1}

The family of self-adjoint conditions naturally arising at the
vertex $\nu=0$ of
the star graph $\mathcal{G}$  has the following description
\begin{equation}\label{s-a_cond}
(U-I)\mathbf{U}(t,0)+i(U+I)\mathbf{U'}(t,0)=0,
\end{equation} 
where $\mathbf{U}(t,0)=(u_j(t,0))_{j=1}^N$,
$\mathbf{U'}(t,0)=(u'_j(t,0))_{j=1}
^N$, $U$ is an arbitrary unitary $N\times N$ matrix, and $I$ is
the $N\times N$
identity matrix. 
The conditions \eqref{s-a_cond} at $\nu=0$ define the
$N^2$-parametric family of
self-adjoint extensions of the closable symmetric operator (see
\cite[Chapter
17]{BlaExn08})
$$\mathbf{H}_0=\bigoplus\limits_{j=1}^N\frac{-d^2}{dx^2},\quad\dom(\mathbf{H}
_0)=\bigoplus\limits_{j=1}^N C_0^\infty(\mathbb{R}_+).$$

We consider two choices of matrix $U$ which correspond to
so-called $\delta$- and
$\delta'$- interactions on the star graph $\mathcal{G}$. More
precisely, the
matrix 
$ 
U=\frac{2}{N+i\alpha}\mathcal{I}-I,$ $
\alpha\in \mathbb{R}\setminus\{0\},
$ 
 where $\mathcal{I}$ is the $N\times N$ matrix 
 whose all entries 
equal  one, induces the following nonlinear 
Schr\"odinger equation with $\delta$-interaction 
on the star graph $\mathcal{G}$ (NLS-$\delta$ equation)
\begin{equation}
\label{NLS_graph}
i\partial_t \mathbf{U}-\mathbf{H}_\alpha^\delta\mathbf{U} 
+|\mathbf{U}|^{p-1}\mathbf{U}=0.
\end{equation}
Here $\mathbf{H}_\alpha^\delta$ is the self-adjoint operator on $L^2(\mathcal{G})$ acting as  $(\mathbf{H}_\alpha^\delta \mathbf{V})(x)
=(-v_j''(x))_{j=1}^N,\,\, x> 0$, on the domain  
$\dom(\mathbf{H}_\alpha^\delta)=\bb D_{\alpha,\delta}$, where  
\begin{equation}\label{D_alpha} 
\bb D_{\alpha,\delta}:= \Big \{\mathbf{V}\in 
H^2(\mathcal{G}): 
v_1(0)=...=v_N(0),\,\,\sum\limits_{j=1}^N v_j'(0)=\alpha
v_1(0)\Big \}.
\end{equation}
Model \eqref{NLS_graph}-\eqref{D_alpha} has been extensively
studied in
\cite{AdaNoj15, AdaNoj14}. In particular, the authors showed
well-posedness of the
corresponding Cauchy problem. Moreover, they investigated the
existence and the
particular form of standing waves, as well as their variational
and stability
properties (see Theorems \ref{1bump} and \ref{st_graph} below).

The second model we are interested in corresponds to  
$ 
U=I-\frac{2}{N-i\lambda}\mathcal{I},$ $ \lambda\in
\mathbb{R}\setminus\{0\},
$ 
which induces the nonlinear Schr\"odinger equation with
$\delta'$-interaction on
the graph $\mathcal{G}$ (NLS-$\delta'$ equation)
  \begin{equation}\label{NLS_graph'}
i\partial_t \mathbf{U}-\mathbf{H}_\lambda^{\delta'} \mathbf{U}
+|\mathbf{U}|
^{p-1}\mathbf{U}=0,
\end{equation}
where $\mathbf{H}_\lambda^{\delta'}$ is the self-adjoint operator
on
$L^2(\mathcal{G})$ acting as $(\mathbf{H}_\lambda^{\delta'} \mathbf{V})(x)
=(-v_j''(x))_{j=1}^N,\,\, x>0,$ on the domain $\dom(\mathbf{H}_\lambda^{\delta'})=\bb D_{\lambda,\delta'}$, 
\begin{equation}
\label{H_lambda} 
\bb D_{\lambda,\delta'}:=
\Big \{\mathbf{V}\in H^2(\mathcal{G}): 
v'_1(0)=...=v'_N(0),\,\,\sum\limits_{j=1}^N 
 v_j(0)=\lambda v'_1(0)\Big \}.
\end{equation}
To our knowledge such type of interaction has never been studied
for the NLS  equation on the
star graph. In this connection one of the principal aims of this
paper is to
establish some results on the existence and the orbital
stability of standing
wave solutions to \eqref{NLS_graph'}.

In Section 4 we consider the following NLS equations with
logarithmic
nonlinearity on the star graph (NLS-log-$\delta$ and
NLS-log-$\delta'$ equation):
 \begin{equation}\label{NLS_graph_log1}
i\partial_t \mathbf{U}-\mathbf{H}_\alpha^\delta\mathbf{U}
+\mathbf{U}\logg|
\mathbf{U}|^{2}=0,
\end{equation}
 \begin{equation}\label{NLS_graph_log2}
i\partial_t \mathbf{U}-\mathbf{H}_\lambda^{\delta'}\mathbf{U}
+\mathbf{U}\logg|
\mathbf{U}|^{2}=0.
\end{equation}
Model \eqref{NLS_graph_log1} has been studied in \cite{Ard16}. In
particular, the
author showed well-posedness of the Cauchy problem in the Banach
space
$W_{\mathcal{E}}(\mathcal{G})$ (see Theorem
\ref{well_log_graph}), and studied
stability properties of the ground state for the corresponding
stationary
equation.

\subsection{Review of the results on the orbital stability for
the NLS equation
with point interactions on a star graph.} 
Crucial role in the orbital stability analysis of standing waves
is played by the
symmetries of NLS equation \eqref{NLS_graph_ger} (and 
\eqref{NLS_log_graph_ger}) The basic symmetry associated to the
mentioned
equation is phase invariance, namely, if $ \mathbf{U}$ is a
solution of
\eqref{NLS_graph_ger} then $ e^{i\theta}\mathbf{U}$ is also a
solution for any $
\theta\in [0,2\pi)$. Thus, it is reasonable to define orbital
stability as
follows (for  models \eqref{NLS_graph_ger} and
\eqref{NLS_log_graph_ger}).

\begin{definition}
The standing wave $\mathbf{U}(t, x) = e^{i\omega
t}\mathbf{\Phi}(x)$ is said to
be \textit{orbitally stable} in a Banach space $X$ if for any
$\varepsilon > 0$
there exists $\eta > 0$ with the following property: if
$\mathbf{U}_0 \in X$
satisfies $||\mathbf{U}_0-\mathbf{\Phi}||_{X} <\eta$,
then the solution $\mathbf{U}(t)$ of (\ref{NLS_graph_ger}) (resp.
(\ref{NLS_log_graph_ger})) with $\mathbf{U}(0) = \mathbf{U}_0$
exists for any
$t\in\mathbb{R}$ and
\[\sup\limits_{t\in \mathbb{R}
}\inf\limits_{\theta\in\mathbb{R}}||\mathbf{U}(t)-
e^{i\theta}\mathbf{\Phi}||_{X} < \varepsilon.\]
Otherwise, the standing wave $\mathbf{U}(t, x) = e^{i\omega
t}\mathbf{\Phi}(x)$ is
said to be \textit{orbitally unstable} in $X$.
\end{definition}

In particular, for the NLS-$\delta$ and NLS-$\delta'$ equations on
the star graph $
\mathcal G$ defined by \eqref{NLS_graph} and \eqref{NLS_graph'},
the space $X$
coincides with $\mathcal{E}(\mathcal{G})$ and $H^1(\mathcal{G})$,
respectively.

In the first part of the paper we study the orbital stability of
the standing
wave solutions 
$ 
\mathbf{U}(t,x)=e^{i\omega t}\mathbf{\Phi}(x)=\left(e^{i\omega
t}\varphi_j(x)
\right)_{j=1}^N
$  
for  NLS-$\delta$ equation \eqref{NLS_graph} on $\mathcal{G}$.
It is easily
seen that amplitude $\mathbf{\Phi}\in \bb D_{\alpha,\delta}$
satisfies the
following  stationary equation 
\begin{equation}\label{H_alpha}
\mathbf{H}_\alpha^\delta\mathbf{\Phi}+\omega\mathbf{\Phi}-|\mathbf{\Phi}|
^{p-1}\mathbf{\Phi}=0.
\end{equation}
In \cite{AdaNoj14} the authors obtained the following description
of all
solutions   to equation \eqref{H_alpha}. 
\begin{theorem}\label{1bump}
Let $[s]$ denote the integer part of $s\in\mathbb{R}$, and
$\alpha\neq 0$. Then
equation \eqref{H_alpha} has $\left[\tfrac{N-1}{2}\right]+1$ (up
to permutations
of the edges of $\mathcal{G}$) vector solutions $\mathbf{\Phi}
_m^\alpha=(\varphi^\alpha_{m,j})_{j=1}^N,
\,\,m=0,...,\left[\tfrac{N-1}{2}\right]
$, which are given by 
\begin{equation*} 
 \varphi_{m,j}^\alpha(x)   = \left\{
                    \begin{array}{ll}
                      \Big[\frac{(p+1)\omega}{2} 
\sech^2\Big(\frac{(p-1)\sqrt{\omega}}{2}x-a_m\Big)\Big]^{\frac{1}{p-1}},
&
\quad\hbox{$j=1,...,m$;} \\
                     \Big[\frac{(p+1)\omega}{2} 
\sech^2\Big(\frac{(p-1)\sqrt{\omega}}{2}x+a_m\Big)\Big]^{\frac{1}{p-1}},
&
\quad\hbox{$j=m+1,...,N,$}
                    \end{array}
                  \right.
\end{equation*} 
 where 
 $
 a_m =\tanh^{-1} (\frac{\alpha}{(2m-N)\sqrt{\omega}} ),\,\,\text{and}\,\,\,\,\omega>\tfrac{\alpha^2}{(N-2m)^2}. 
$
\end{theorem}

\begin{remark}\label{tails_bumps}
\rm
 $(i)$ 
Note that in the case $\alpha<0$ vector  
$\mathbf{\Phi}_m^\alpha=(\varphi^\alpha_{m,j})_{j=1}^N$ 
has $m$ \textit{bumps} and $N-m$ \textit{tails}. 
It is easily seen that  
 $\mathbf{\Phi}^{\alpha}_0$ is the \textit{N-tail profile}. 
 Moreover, the $N$-tail profile is the only symmetric 
 (i.e. invariant under permutations of the edges) solution 
 of equation \eqref{H_alpha}.  
 In the case $N=5$ we have three 
 types of profiles:  \textit{5-tail profile}, 
  \textit{4-tail/1-bump profile} and 
   \textit{3-tail/2-bump profile}. 
They are demonstrated on Figure 1 (from the left to the right).

$(ii)$
  In the case $\alpha>0$ vector  
  $\mathbf{\Phi}_m^\alpha=(\varphi^\alpha_{m,j})_{j=1}^N$ 
  has $m$ \textit{tails} and $N-m$ \textit{bumps} respectively. 
  For $N=5$ we have: \textit{5-bump profile, 
  4-bump/1-tail profile, 3-bump/ 2-tail profile}. 
They are demonstrated on Figure 2.  \end{remark}
 
 \begin{tikzpicture}[scale=0.6]
	\clip (-3,-3) rectangle (17,2);
\draw[-,color=gray] (2,1).. controls +(-0.7,-0.7) ..  (0.3,0.2);
\draw[-,color=gray] (2,1).. controls +(-0.3,-0.9)  ..  (0.5,-1);
\draw[-,color=gray] (2,1).. controls +(0.2,-1)  ..  (2.8,-1.2);
\draw[-,color=gray] (2,1) .. controls +(0.2,-0.5) .. (4.2,-0.2);
\draw[-,color=gray] (2,1).. controls +(0.6,-0.3)  ..  (3.5,1);		
\draw[-latex,thin](2,0)--++(-2,0);
       \draw[-latex, thin](2,0)--++(-1.8,-1.3);
        \draw[-latex, thin](2,0)--++(1,-1.7);
		\draw[-latex, thin](2,0)--++(2.5,-0.5);
        \draw[-latex, thin](2,0)--++(2,1);
		\node[label={[xshift=2cm, yshift=-2.5cm] }]{};

		\begin{scope}[shift={(5.5,0)}]
\draw[-,color=gray] (2,1).. controls +(-0.5,0.5) and +(1.5,-0.1) .. (0.3,0.2);
\draw[-,color=gray] (2,1).. controls +(-0.3,-0.9)  ..  (0.5,-1);
\draw[-,color=gray] (2,1).. controls +(0.2,-1)  ..  (2.8,-1.2);
\draw[-,color=gray] (2,1) .. controls +(0.2,-0.5) .. (4.2,-0.2);
\draw[-,color=gray] (2,1).. controls +(0.6,-0.3)  ..  (3.5,1);		
\draw[-latex,thin](2,0)--++(-2,0);
       \draw[-latex, thin](2,0)--++(-1.8,-1.3);
        \draw[-latex, thin](2,0)--++(1,-1.7);
		\draw[-latex, thin](2,0)--++(2.5,-0.5);
       \draw[-latex, thin](2,0)--++(2,1);
		\node[label={[xshift=1cm, yshift=-1.8cm] Figure 1}]{};	
		\end{scope}

		\begin{scope}[shift={(11,0)}]
\draw[-,color=gray] (2,1).. controls +(-0.5,0.5) and +(1.5,-0.1) .. (0.3,0.2);
\draw[-,color=gray] (2,1).. controls +(-0.3,-0.9)  ..  (0.5,-1);
\draw[-,color=gray] (2,1).. controls +(0.2,-1)  ..  (2.8,-1.2);
\draw[-,color=gray] (2,1) .. controls +(0.2,-0.5) .. (4.2,-0.2);
  \draw[-,color=gray] (2,1).. controls +(0.3,0.5) and +(-0.8,-0.5)  ..  (3.5,1);
		\draw[-latex,thin](2,0)--++(-2,0);
        \draw[-latex, thin](2,0)--++(-1.8,-1.3);
        \draw[-latex, thin](2,0)--++(1,-1.7);
		\draw[-latex, thin](2,0)--++(2.5,-0.5);
        \draw[-latex, thin](2,0)--++(2,1);
	\end{scope}
\end{tikzpicture}

\begin{tikzpicture}[scale=0.6]
	\clip (-3,-3) rectangle (17,2);
\draw[-,color=gray] (2,1).. controls +(-0.5,0.5) and +(1.5,-0.1)
..
(0.3,0.2);
\draw[-,color=gray] (2,1).. controls +(-0.3,0.4) and +(1, 0.5) ..(0.5,-1);  
\draw[-,color=gray] (2,1).. controls +(0.2,0.3) and +(-0.1, -0.1)
..
(2.8,-1.2);
\draw[-,color=gray] (2,1) .. controls +(0.2,0.5) and +(-1.1, 0.3)
..
(4.2,-0.2);
\draw[-,color=gray] (2,1).. controls +(0.3,0.5) and +(-0.8,-0.5)
..
(3.5,1);
		\draw[-latex,thin](2,0)--++(-2,0);
        \draw[-latex, thin](2,0)--++(-1.8,-1.3);
        \draw[-latex, thin](2,0)--++(1,-1.7);
		\draw[-latex, thin](2,0)--++(2.5,-0.5);
      \draw[-latex, thin](2,0)--++(2,1);
		\node[label={[xshift=2cm, yshift=-2.5cm]  }]{};

		\begin{scope}[shift={(5.5,0)}]
\draw[-,color=gray] (2,1).. controls +(-0.7,-0.7) ..  (0.3,0.2);\draw[-,color=gray] (2,1).. controls +(-0.3,0.4) and +(1, 0.5) ..
(0.5,-1);
\draw[-,color=gray] (2,1).. controls +(0.2,0.3) and +(-0.1, -0.1)
..
(2.8,-1.2);
\draw[-,color=gray] (2,1) .. controls +(0.2,0.5) and +(-1.1, 0.3)
..
(4.2,-0.2);
\draw[-,color=gray] (2,1).. controls +(0.3,0.5) and +(-0.8,-0.5)
..
(3.5,1);
		\draw[-latex,thin](2,0)--++(-2,0);
       \draw[-latex, thin](2,0)--++(-1.8,-1.3);
        \draw[-latex, thin](2,0)--++(1,-1.7);
	\draw[-latex, thin](2,0)--++(2.5,-0.5);
        \draw[-latex, thin](2,0)--++(2,1);
\node[label={[xshift=1cm, yshift=-1.8cm] Figure 2 }]{};
\end{scope}

		\begin{scope}[shift={(11,0)}]
\draw[-,color=gray] (2,1).. controls +(-0.7,-0.7) ..  (0.3,0.2);\draw[-,color=gray] (2,1).. controls +(-0.3,0.4) and +(1, 0.5) ..
(0.5,-1);
\draw[-,color=gray] (2,1).. controls +(0.2,0.3) and +(-0.1, -0.1)
..
(2.8,-1.2);
\draw[-,color=gray] (2,1) .. controls +(0.2,0.5) and +(-1.1, 0.3)
..
(4.2,-0.2);
\draw[-,color=gray] (2,1).. controls +(0.6,-0.3)  ..  (3.5,1);		\draw[-latex,thin](2,0)--++(-2,0);
        \draw[-latex, thin](2,0)--++(-1.8,-1.3);
        \draw[-latex, thin](2,0)--++(1,-1.7);
		\draw[-latex, thin](2,0)--++(2.5,-0.5);
       \draw[-latex, thin](2,0)--++(2,1);
	\end{scope}
\end{tikzpicture}

It was shown in \cite{AdaNoj14} that for
$-N\sqrt{\omega}<\alpha<\alpha^*<0$,
the vector solution
$\bb\Phi_{\alpha,\delta}=(\varphi_{\alpha,\delta})_{j=1}
^N:=\bb\Phi^\alpha_0$,
 \begin{equation}\label{Phi_vect} 
\varphi_{\alpha,\delta}:=\varphi_{0,j}^\alpha(x)=
 \Big [\frac{(p+1)\omega}{2}
\sech^2 \Big (\frac{(p-1)\sqrt{\omega}}{2}x+\tanh^{-1} \Big (\frac{-\alpha}
{N\sqrt{\omega}} \Big ) \Big ) \Big ]^{\frac{1}{p-1}}\end{equation}
is the ground
state.
The parameter $\alpha^*$ above originates from the variational
problem associated
with equation \eqref{H_alpha}, and it guarantees constrained
minimality of the
action functional
 \begin{equation}\label{S_graph}
\bb S_\alpha(\bb V)=\tfrac 1{2}||\bb
V'||^2+\tfrac{\omega}{2}||\bb V||^2 -\tfrac
1{p+1}||\bb V||_{p+1}^{p+1}+\tfrac{\alpha}{2}|v_1(0)|^2, \quad
\bb V=(v_j)_{j=1}
^N\in\mathcal E(\mathcal{G}).
 \end{equation}
Namely, the vector solution $\mathbf{\Phi}_{\alpha,\delta}$ is
the ground state
in the sense of the minimality of $\bb S_\alpha(\bb V)$ at
$\mathbf{\Phi}_{\alpha,
\delta}$ with the constraint given by the Nehari manifold 
 $$
\mathcal N=\{\bb V\in \mathcal E(\mathcal{G})\setminus\{0\}:
||\bb V'||^2+\omega
||\bb V||^2 - ||\bb V||_{p+1}^{p+1}+\alpha |v_1(0)|^2=0\}.
 $$ 
For $\alpha>0$ the $N$-bump profile
$\mathbf{\Phi}_{\alpha,\delta}$ does not
have the variational characterization (see \cite[Remark
14]{FukJea08}).
In \cite{AdaNoj14} the following orbital stability result has
been shown.
 \begin{theorem}\cite[Theorem 2]{AdaNoj14}\label{st_graph}
Let $1<p\leq 5$,\,\,$\alpha<\alpha^*<0$, and
$\omega>\tfrac{\alpha^2}{N^2}$. Then
the standing wave $e^{i\omega t}\mathbf{\Phi}_{\alpha,\delta}$ is
orbitally stable
in 
$\mathcal{E}(\mathcal{G})$.
\end{theorem}

The authors in \cite{AdaNoj14} showed also that for $p>5$ there
exists $
\mathbf{\omega}^*>\frac{\alpha^2}{N^2}$ such that $e^{i\omega
t}\mathbf{\Phi}
_{\alpha,\delta}$
is stable in $\mathcal{E}(\mathcal{G})$ for any
$\omega\in (\frac{\alpha^2}
{N^2},\mathbf{\omega}^* )$ and unstable for any $\omega>
\mathbf{\omega}^*$.
Stronger version of the above theorem was proved in \cite[Theorem
1]{AdaNoj15}. In
particular, the authors proved orbital stability of $e^{i\omega
t}\mathbf{\Phi}
_{\alpha,\delta}$ for $\alpha<0$ without restriction
$\alpha<\alpha^*<0$.
For $m\neq 0$, $\alpha<0$ in Theorem \ref{1bump} we have 
$S(\bb\Phi^\alpha_m)>S(\bb\Phi^\alpha_0)$ which means that
$\bb\Phi^\alpha_m$ for
$m\neq 0$ is \textit{an excited state}. Stability properties of
the excited states
as well as of $\bb\Phi^\alpha_m$ for $\alpha>0$ were studied in
\cite{AngGol17a}.

To our knowledge, the problem of orbital stability of standing
waves $\mathbf{U}
(t,x)=e^{i\omega t}\mathbf{\Phi}(x)$ has never been considered
for NLS-$\delta'$
equation \eqref{NLS_graph'} on the star graph.
In the present paper we study the orbital stability of the
standing waves $
\mathbf{U}(t,x)=e^{i\omega t}\mathbf{\Phi}_{\lambda,\delta'}$
with $N$-tail profile
$\mathbf{\Phi}_{\lambda,\delta'}=(\varphi_{\lambda,\delta'})_{j=1}^N$,
where
\begin{equation}\label{varphi_lam}\varphi_{\lambda,\delta'}
(x)= \Big [\frac{(p+1)\omega}{2}
\sech^2 \Big (\frac{(p-1)\sqrt{\omega}}{2}x+
\tanh^{-1} \Big (\frac{-N}{\lambda\sqrt{\omega}} \Big ) \Big ) \Big ]^{\tfrac{1}
{p-1}},
\end{equation}
with  $\omega>\tfrac{N^2}{\lambda^2}$ and $\lambda<0$.
In Section 4 we prove  new result on stability of
$e^{i\omega t}
\mathbf{\Phi}_{\lambda,\delta'}$ (see Theorem \ref{Main}). 

\begin{remark}\label{gensolu}
\rm
The description of the set of all solutions to the stationary
equation
\begin{equation}\label{stat_delta'}
\mathbf{H}_\lambda^{\delta'}\mathbf{\Phi}+\omega\mathbf{\Phi}-|\mathbf{\Phi}|
^{p-1}\mathbf{\Phi}=0,
\end{equation} 
is unknown. We note that any $L^2$-solution to
\eqref{stat_delta'} has the form
$$
\bb
\Phi(x)=(\varphi_j(x))_{j=1}^N=(\sigma_j\varphi_0(x+x_j))_{j=1}^N,
$$
where $\sigma_j\in\mathbb{C},|\sigma_j|=1$, $x_j\in\mathbb{R}$,
and $
\varphi_0(x)= [\frac{(p+1)\omega}{2}
\sech^2 (\frac{(p-1)\sqrt{\omega}}{2}
x ) ]^{\frac{1}{p-1}}$.
Hence, denoting $t_j=\tanh(x_j)$, from \eqref{H_lambda} we get
the relations
\begin{equation*}\label{prof_lambda}
\left\{\begin{array}{c}
\sigma_1(1-t_1)^{\frac{1}{p-1}}t_1=...=\sigma_N(1-t_N)^{\frac{1}{p-1}}t_N,\\
\sum\limits_{j=1}^N\sigma_j(1-t_j)^{\frac{1}{p-1}}=-\lambda\sqrt{\omega}
\sigma_1(1-t_1)^{\frac{1}{p-1}}t_1.
\end{array}\right.
\end{equation*}
In \cite{AdaNoj13a}, for the case of $\mathcal{G}=\mathbb{R}$
($\delta'$-
interaction on the line), the authors established the existence
of two families
(odd and asymmetric) of solutions to \eqref{stat_delta'}. For
$N\geq3$, it seems
to be very nontrivial problem to determine a complete description
of the
solutions to \eqref{stat_delta'}. Observe that in the case of
NLS-$\delta$
equation
  the situation is easier since the continuity condition
$\varphi_1(0)=...=\varphi_N(0)$ implies
$|\varphi'_1(0)|=...=|\varphi'_N(0)|$,
therefore, $\sigma_1=...=\sigma_N$  and $x_j=\pm a,\, a>0.$
\end{remark}
In the case of NLS-log-$\delta$ equation the profile of the
standing wave
$e^{i\omega t}\bb\Psi$  satisfies the equality
 \begin{equation}\label{ellip_log}
\mathbf{H}_\alpha^\delta\mathbf{\Psi}+\omega\mathbf{\Psi}-\mathbf{\Psi}\logg|
\mathbf{\Psi}|^2=0.
\end{equation}
From \cite{Ard16} it follows that solutions to \eqref{ellip_log}
have the
following description.
  \begin{theorem}\label{prof_log}
Let $\alpha\neq 0$. Then equation \eqref{ellip_log} has
$\left[\tfrac{N-1}
{2}\right]+1$ vector solutions
$\mathbf{\Psi}_m^\alpha=(\psi^\alpha_{m,j})_{j=1}
^N, \,\,m=0,...,\left[\tfrac{N-1}{2}\right],$ given by 
\begin{align*}
 \psi_{m,j}^\alpha(x)= \left\{
                    \begin{array}{ll}
e^{\tfrac{\omega+1}{2}}e^{-\tfrac{(x-a_m)^2}{2}}, & \,\hbox{$j=1,...,m$;} \\
e^{\tfrac{\omega+1}{2}}e^{-\tfrac{(x+a_m)^2}{2}}, & \,\hbox{$j=m+1,...,N,$}
                    \end{array}
                  \right.
                \,\,\text{where}\,\, a_m=\frac{\alpha}{2m-N}.
 \end{align*}
 \end{theorem}
 
We should note that the structure of the profiles that solve
\eqref{ellip_log}
is similar to the one in the case of NLS-$\delta$ equation (see
Remark
\ref{tails_bumps}).
It was proved in \cite{Ard16} that for
$\alpha<\alpha^*_{\logg}<0$, the vector
solution $
\mathbf{\Psi}_{\alpha,\delta}=(\psi_{\alpha,\delta})_{j=1}^N$
defined
by 
 \begin{equation}\label{Phi_vect_log} 
\psi_{\alpha,\delta}= \psi_{0,j}^\alpha(x)=
e^{\tfrac{\omega+1}{2}}e^{-\tfrac{(x-
\tfrac{\alpha}{N})^2}{2}}\end{equation} is the ground state.
The condition $\alpha<\alpha^*_{\logg}$ guarantees constrained
minimality of the
following action functional for $\bb V\in
W_{\mathcal{E}}(\mathcal{G})$,
\begin{equation}
\label{act_log} \bb S_{\alpha,\log}(\bb V)=
\tfrac 1{2}||\bb V'||^2+\tfrac{(\omega+1)}{2}
||\bb V||^2 -\tfrac 1{2}\sum\limits_{j=1}^N
\int _0^\infty|v_j|^2\logg|v_j|^2dx+\tfrac{\alpha}{2}|v_1(0)|^2.\end{equation}
Namely, the vector solution $\mathbf{\Psi}_{\alpha,\delta}$ is
the ground state
in the sense of the minimality of $\bb S_{\alpha,\log}(\bb V)$ at
$\mathbf{\Psi}
_{\alpha,\delta}$ with the constraint given by the Nehari
manifold $\mathcal N$,
namely, $\bb V\in \mathcal N$ if and only if $\bb V\in
W_{\mathcal{E}}
(\mathcal{G})\setminus\{0\}$ and 
$$
||\bb V'||^2+\omega||\bb V||^2 -\sum\limits_{j=1}^N \int
_0^\infty|v_j|^2\logg|
v_j|^2dx+\alpha|v_1(0)|^2=0.
 $$ 
In \cite{Ard16} the author proved that the standing wave
$e^{i\omega t}
\mathbf{\Psi}_{\alpha,\delta}$ is orbitally stable in 
$W_{\mathcal{E}}(\mathcal{G})$ for $\alpha<\alpha^*_\log<0$ and
$\omega\in
\mathbb{R}$.
Below we will overcome the restriction $\alpha<\alpha^*_\log$ in
the space
$W^1_{\mathcal{E}}(\mathcal{G})$ (see Theorem
\ref{stabil_log_delta}), moreover,
we will show spectral instability of the standing wave
$e^{i\omega t}\mathbf{\Psi}
_{\alpha,\delta}$ for any $\alpha>0$ ($\mathbf{\Psi}_{\alpha,\delta}$ is the $N$-bump profile in this
case).

Similarly to the previous case, we show that the $N$-tail
standing wave $e^{\omega
it}\bb \Psi_{\lambda,\delta'}$ for the NLS-log-$\delta'$ equation,
where
\begin{equation}\label{prof_log'}
\bb{\Psi}_{\lambda,\delta'}=(\psi_{\lambda,\delta'})_{j=1}^N,\quad
\psi_{\lambda,
\delta'}=e^{\tfrac{\omega+1}{2}}e^{-\tfrac{(x-\tfrac{N}{\lambda})^2}{2}},
 \end{equation}
is orbitally stable in $W^1(\mathcal{G})$ for $-N<\lambda<0$, and
spectrally
unstable for $\lambda<-N$ (see Theorem \ref{Main_log}). Note that
we do not need
to assume that $N$ is even to show the instability (compare with
Theorem
\ref{Main}).

\section{Stability theory of standing wave solutions for the
NLS-$\delta$ and
the NLS-$\delta'$  equations on a star graph} 

 \subsection{ The NLS-$\delta$ equation on a star graph}
 
In this Subsection we study the orbital stability of the standing
wave $
\mathbf{U}(t,x)=e^{i\omega t}\mathbf \Phi_{\alpha, \delta}(x)$ of
NLS-$\delta$
equation \eqref{NLS_graph} with the particular $N$-bump profile
$\mathbf{\Phi}
_{\alpha,\delta}=(\varphi_{\alpha, \delta})_{j=1}^N$  defined by \eqref{Phi_vect}. As we are investigating orbital stability in
$\mathcal
E(\mathcal{G})$ we need to use the well-posedness of the initial
value problem
for equation \eqref{NLS_graph} in this space. In \cite{AdaNoj14}
the authors
established the results on local and global well-posedness of
\eqref{NLS_graph} in
$\mathcal E(\mathcal{G})$. Below we complete and extend these
results, aiming to
use them in the sequel for  our instability analysis.

First, we establish the following property for the unitary group
associated to
\eqref{NLS_graph}. 
\begin{lemma}\label{rela00} Let 
$\{e^{-it\bb H_{\delta}
^\alpha}\}_{t\in \mathbb R}$ be the family of unitary operators
associated to
 NLS-$\delta$ model \eqref{NLS_graph}. Then, for every $ \bb
V=(v_j)_{i=1}^N\in
\EE(\mathcal{G})$ we have 
 \begin{equation}\label{group_comut0}
\partial_x(e^{-it\bb H_{\delta}^\alpha }\bb V)=-e^{-it\bb
H_{\delta}^\alpha}\bb V'
+ 
\mathcal B(\bb V'),
\end{equation}
where $\mathcal B(\bb
V')=(2e^{it\partial^2_x}\tilde{v}_j)_{j=1}^N$, with $
\tilde{v}_j(x)=\left\{\begin{array}{c}
v'_j(x),\,\ x\geq 0,\\
0,\quad\;\;\; x<0
\end{array}\right.$, and $e^{it\partial^2_x}$ is the unitary
group associated with
the free Schr\"odinger operator on $\mathbb{R}$.
\end{lemma}

\begin{proof}[\bf Proof.]
 Without loss of generality we assume that $\alpha>0$.
Using functional calculus for unbounded self-adjoint operators
and the classical
expression for the resolvent of $-\frac{d^2}{dx^2}$ on the
positive half-line we
get the formulas 
\begin{equation}\label{group0}
e^{-it\bb H_\delta^\alpha}\bb V(x)=\tfrac{i}{\pi} \int
_{-\infty}^\infty e^{-
it\tau^2}\tau \bb R_{i\tau}\bb V(x)d\tau,
\end{equation}
where $\bb R_{\mu}\bb V=(\bb H_\delta^\alpha+\mu^2 I)^{-1}\bb V$
has the
components 
 \begin{equation}\label{res0}
(\bb R_{\mu}\bb V)_j(x)=\tilde{c}_je^{-\mu x}+\frac{1}{2\mu} \int
_0^\infty
v_j(y)e^{-|x-y|\mu}dy.  
 \end{equation}
The coefficients $\tilde{c}_j$ are determined by the condition
$\bb R_{\mu}\bb
V\in \bb D_{\alpha, \delta}$. It is easily seen (e.g.
\cite[Appendix-6]
{BanIgn14}) that
$ 
\bb V\in \bb D_{\alpha,\delta}$  iff $ A\bb
V(0)+B\bb V'(0)=\bb 0,
$
 where
 \small
  $$
   A=\left(\begin{array}{ccccc}
  1&-1&0&...&0\\
  0&1&-1&...&0\\
  \vdots &\vdots&\vdots&  &\vdots\\
  0&0&0&...&-1\\
\tfrac{\alpha}{N}&\tfrac{\alpha}{N}&\tfrac{\alpha}{N}&...&\tfrac{\alpha}{N}
  \end{array}\right),\quad B=\left(\begin{array}{ccccc}
  0& & ... & &0\\
  0&  & & & 0\\
  \vdots & & &  &\vdots\\
  & & & & \\
  -1 & &...& &-1
  \end{array}\right) .
  $$
  \normalsize
Let $ t_j(\mu)=\frac12 \int _0^\infty v_j(y)e^{-\mu y}dy$, then
from
\eqref{res0} we get
$(\bb R_{\mu}\bb V)_j(0)= \tilde{c}_j+\frac{1}{\mu} t_j(\mu)$ and
$\partial_x[(\bb
R_{\mu}\bb V)_j](0)= -\mu \tilde{c}_j+ t_j(\mu)$. Therefore,
$(\tilde{c}_j)_{j=1}
^N$ is  the unique solution to the system
\small
\begin{equation}\label{c_tilde_delta}
{ \left(\begin{array}{ccccc}
  1&-1&0&...&0\\
  0&1&-1&...&0\\
    \vdots &\vdots&\vdots&  &\vdots\\
  0&0&0&...&-1\\
  \tfrac{\alpha}{N}+\mu&\tfrac{\alpha}{N}+\mu&\tfrac{\alpha}{N}+
\mu&...&\tfrac{\alpha}{N}+\mu
  \end{array}\right)\left(\begin{array}{ccccc}\tilde{c}_1\\
  \\
 \vdots \\
  \\
  \tilde{c}_N
  \end{array}
\right) 
\! = \!
-\frac{1}{\mu}\left(\begin{array}{ccccc}t_1(\mu)-t_2(\mu)\\
  \vdots\\
  \\
 t_{N-1}(\mu)-t_N(\mu) \\
  (\frac{\alpha}{N}-\mu)\sum\limits_{j=1}^Nt_j(\mu) \end{array}
  \right).}
\end{equation}
\normalsize

Below we find $\bb R_{\mu}\bb V'$. Suppose initially that $v_j\in
C^{\infty}
_0(\mathbb R_+)$, $1\leq j\le N$, then there are coefficients
$\tilde{d}_j$ such
that
\begin{equation}\label{U0}
\begin{aligned}
(\bb R_{\mu}\bb V')_j(x)&=\tilde{d}_je^{-\mu x}+\frac{1}{2\mu}
\int _0^\infty
v'_j(y)e^{-\mu |x-y|}dy\\
&=\tilde{d}_je^{-\mu x}-\frac{1}{2} \int _0^\infty
v_j(y)\sign(x-y)e^{-\mu |x-y|}
dy,
\end{aligned}
\end{equation}
where in the last equality we have used integration by parts.
Thus, we obtain $
(\bb R_{\mu}\bb V')_j(0)=\tilde{d}_j+t_j(\mu)$. Moreover, since
\begin{equation*}\label{U''}
\begin{aligned}
\partial_x (\bb R_{\mu}\bb V')_j(x)&=-\mu \tilde{d}_je^{-\mu
x}-\frac{1}{2} \int
_0^\infty v'_j(y)\sign(x-y)e^{-\mu |x-y|}dy,
\end{aligned}
\end{equation*}
it follows from integration by parts $\partial_x (\bb R_{\mu}\bb
V')_j(0)=-\mu
\tilde{d}_j +\mu t_j(\mu)$. Hence from the uniqueness of solution
to system
(\ref{c_tilde_delta}) it follows that $\bb R_{\mu}\bb V'\in \bb
D_{\alpha,\delta}
$ iff $\tilde{d}_j=\mu \tilde{c}_j$. Therefore, we
obtain from
\eqref{res0} and the  second equality in \eqref{U0} 
\begin{equation*}\label{comut}
\begin{aligned}
\partial_x (\bb R_{\mu}\bb V)_j(x)&= -(\bb R_{\mu}\bb V')_j(x)-
\int _0^\infty
v_j(y)\sign(x-y)e^{-\mu |x-y|}dy\\
&=-(\bb R_{\mu}\bb V')_j(x)+\frac{1}{\mu} \int _0^\infty
v'_j(y)e^{-\mu |x-y|}dy.
\end{aligned}
\end{equation*}
Thus, from representation \eqref{group0} we get
$$
\partial_x(e^{-it\bb H_\lambda^{\delta'}}\bb V)=-e^{-it\bb
H_\lambda^{\delta'} }
\bb V' +\mathcal B(\bb V'),
$$
where
$$
(\mathcal B(\bb V'))_j(x)= \frac{1}{\pi} \int _{-\infty}^\infty
e^{-it\tau^2}
\int _0^\infty v'_j(y)e^{-i\tau |x-y|}dyd\tau.
$$

Below we find $\mathcal B(\bb V')$. It is well-known that
$e^{it\partial_x^2}$
can be represented as $e^{it\partial_x^2}\phi=S_t\ast \phi$,
where $\widehat{S_t}
(\xi)=e^{-it\xi^2}$.  Since for $t\neq 0$ and $x\in \mathbb R$
$$
S_t(x)=\frac{1}{2\pi} \int _{-\infty}^\infty e^{-it\tau^2}
e^{i\tau x} d\tau=
\frac{1}{2\pi}\frac{\sqrt{\pi}}{\sqrt{-t}} e^{i\pi/4} e^{i
\frac{x^2}{4 t}}
=\Big(\frac{1}{4\pi it}\Big)^{1/2} e^{i \frac{x^2}{4 t}},
$$  
it follows for $\phi(x)=\left\{\begin{array}{c}
v'_j(x),\,\ x\geq 0,\\
0,\quad\;\;\; x<0
\end{array}\right.$
\begin{equation}\label{I0}
\begin{aligned}
I=&\frac{1}{\pi} \int _{-\infty}^\infty e^{-it\tau^2} \int
_{-\infty}^\infty
\phi(y)\chi_{[0, x]}(y)e^{i\tau (y-x)}dyd\tau \\
=&2 \int _{-\infty}^\infty \phi(y)\chi_{[0, +\infty)}(x-y)
S_{t}(x-y)dy=2
(\chi_{[0, +\infty)} S_{t})\ast \phi (x).
\end{aligned}
\end{equation}
Similarly, 
\begin{equation}\label{II0}
II=\frac{1}{\pi} \int _{-\infty}^\infty e^{-it\tau^2} \int
_{-\infty}^\infty
\phi(y)\chi_{[x, +\infty)}(y)e^{i\tau (x-y)}dyd\tau= 2
(\chi_{(-\infty, 0])}
S_{t})\ast \phi (x).
\end{equation}
Thus, from \eqref{I0}-\eqref{II0} we have $
(\mathcal B(\bb V'))_j(x)=I+II=2S_{t}\ast \phi
(x)=2e^{it\partial_x^2}\phi (x)$.
Hence relation \eqref{group_comut0} follows provided that each
component of $\bb
V$ has compact support. The general case follows from a density
argument.
\end{proof} 

\begin{remark} 
\rm
Observe that 
$
e^{-it\bb H_{\delta}^\alpha} \bb V= e^{-it\bb
H_{\delta}^\alpha}\bb P_c \bb V+
e^{-it\bb H_{\delta}^\alpha}\bb P_p \bb V,
$
where $\bb P_c$ and $\bb P_p$ are $L^2$-orthogonal projections
onto the subspaces
corresponding to the continuous and the discrete spectral part of
$\bb
H_{\delta}^\alpha$.  
For $\alpha>0$, we have $\sigma_c(\bb
H_{\delta}^\alpha)=[0,\infty)$ and $
\sigma_p(\bb H_{\delta}^\alpha)=\emptyset$, therefore $\bb
P_p=\bb 0$. For $\alpha<0$, $\sigma_c(\bb H_{\delta}^\alpha)=[0,\infty)$ and
$\sigma_p(\bb
H_{\delta}^\alpha)=\{-z_0^2\}=\{-\frac{\alpha^2}{N^2}\}$, where the
corresponding
eigenfunction is $\bb
V_{z_0}(x)=(e^{\frac{\alpha}{N}x})_{j=1}^N$, and therefore
$e^{-it\bb H_{\delta}^\alpha}\bb P_p \bb V=e^{itz_0^2}(\bb V,\bb
V_{z_0})\bb
V_{z_0}$. In this case  formula \eqref{group0} takes the form
\begin{equation*}
e^{-it\bb H_\delta^\alpha}\bb V(x)=\tfrac{i}{\pi} \int
_{-\infty}^\infty e^{-
it\tau^2}\tau \bb R_{i\tau}\bb V(x)d\tau+e^{itz_0^2}(\bb V,\bb
V_{z_0})\bb V_{z_0}
(x),
\end{equation*}
which however does not affect the proof of the well-posedness
result.
The proof of the spectral properties of $\bb H_{\delta}^\alpha$
repeats the one
of \cite[Theorem 3.1.4]{AlbGes05} for the case of the
Schr\"odinger operator with
the $\delta$-interaction on the line. In particular, to describe
the point
spectrum for $\alpha<0$ one needs to consider $\bb
H_{\delta}^\alpha$ as the
self-adjoint extension of the symmetric non-negative operator
$\bb L$ defined by
\eqref{L_symm} with deficiency indices $n_\pm(\bb L)=1$ and then
to apply
Proposition \ref{semibounded}.
\end{remark}

\begin{lemma}\label{persistence} The family of
unitary operators $\{e^{-it\bb H^{\alpha}_\delta}\}_{t\in \mathbb
R}$ on
$L^2(\mathcal G)$ preserves the space $\mathcal E(\mathcal G)$,
i.e. for $
\bb V\in \mathcal E (\mathcal G)$ we have $e^{-it\bb
H_{\delta}^\alpha} \bb
V\in  \mathcal E(\mathcal G)$.
\end{lemma}

\begin{proof}[\bf Proof.]
Assume $\alpha>0$. Let $\bb V
\in \mathcal E
(\mathcal G)$, then it follows from \eqref{group_comut0} that
$e^{-it\bb
H_{\delta}^\alpha} \bb V\in H^1(\mathcal G)$. Further, since
$\bf{R}_\mu \bb V\in
\bb D_{\alpha, \delta}$, we get from \eqref{group0} the equality
$$(e^{-it\bb
H_{\delta}^\alpha}\bb V)_1(0)=...=(e^{-it\bb
H_{\delta}^\alpha}\bb V)_N(0).$$
\end{proof}

\begin{theorem}\label{well0} 
Let $p>1$. Then for any $\bb U_0\in \mathcal
E(\mathcal{G})$
there exists $T > 0$ such that 
equation \eqref{NLS_graph} has a unique solution $\bb U \in C
([-T,T],
\mathcal E(\mathcal{G}))
\cap C^1 ([-T,T], \mathcal E'(\mathcal{G}))$ satisfying  $\bb
U(0)=\bb U_0$. For
each $T_0\in (0, T)$ the mapping 
$
\bb U_0\in \mathcal E (\mathcal{G})\to \bb U \in C ([-T_0,T_0],
\mathcal
E(\mathcal{G})),
$
is continuous. In particular, for $p>2$ this mapping is at least
of class $C^2$.
Moreover, if $\bb U_0\in\mathcal E_k(\mathcal{G})$, then $\bb
U(t)\in \mathcal
E_k(\mathcal{G})$ for all $t\in [-T,T]$.

\end{theorem}

\begin{proof}[\bf Proof.]
The local well-posedness result in $\mathcal E(\mathcal{G})$
follows from
standard arguments of the Banach fixed point theorem applied to
non-linear
Schr\"odinger equations (see \cite{Caz03}). We will give the
sketch of the proof
for the case $\alpha>0$. Consider the mapping $J_{\bb U_0}:
C([-T, T], \mathcal
E(\mathcal{G}))\longrightarrow C([-T,T], \mathcal
E(\mathcal{G}))$ given by
$$
J_{\bb U_0}[\bb U](t)=e^{-it\bb H_{\delta}^\alpha}\bb U_0+i \int
_0^te^{-i(t-s)\bb
H_{\delta}^\alpha}|\bb U(s)|^{p-1}\bb U(s)ds,
$$
where $e^{-it\bb H_{\delta}^\alpha}$ is the unitary group given
by
\eqref{group0}. One needs to
show that the mapping $J_{\bb U_0}$ is well-defined. To do this it is necessary 
 to
estimate initially the nonlinear term $|\bb U(s)|^{p-1}\bb U(s)$. Using
 the one-dimensional Gagliardo-Nirenberg inequality one may
show (see
formula (2.3) in \cite{AdaNoj14})
\begin{equation}\label{G-N_graph0}
\|\bb U\|_q \leq C\|\bb U'\|^{\frac12-\frac1q}\|\bb
U\|^{\frac12+\frac1q},\quad
q>2,\, C>0.
\end{equation}
Using \eqref{G-N_graph0}, the relation $|(|f|^{p-1} f)'|\leq C_0
|f|^{p-1} |f'|$
and H\"older's inequality, we obtain for $\bb U\in
H^1(\mathcal{G})$
\begin{equation}\label{pres_space0}
|||\bb U|^{p-1}\bb U||_{H^1(\mathcal{G})}\leq C_1||\bb
U||^p_{H^1(\mathcal{G})}.
\end{equation}  
Let $\bb U_0, \bb U \in \mathcal E(\mathcal{G})$, then  
from Lemmas \ref{rela00}-\ref{persistence} and
\eqref{pres_space0} it follows
that $J_{\bb U_0}[\bb U](t)\in \mathcal E(\mathcal{G})$.
Moreover, using
\eqref{group_comut0}, \eqref{pres_space0}, $L^2$-unitarity of
$e^{-it\bb
H_{\delta}^\alpha}$ and  $e^{it\partial^2_x}$, we get 
\begin{equation*}\label{neq_well_10}
||J_{\bb U_0}[\bb U](t)||_{H^1(\mathcal{G})}\leq C_2||\bb
U_0||_{H^1(\mathcal{G})}
+C_3T\sup\limits_{s\in[0,T]}||\bb U(s)||^p_{H^1(\mathcal{G})},
\end{equation*}
where the positive constants $C_2, C_3$ do not depend on $\bb
U_0$. The
continuity and contraction property of $J_{\bb U_0}$ are proved
in a standard
way. Therefore, we obtain the existence of a unique solution to
the Cauchy problem
associated to  \eqref{NLS_graph}  on $\mathcal E(\mathcal{G})$.

Next, we recall that the argument based on the contraction
mapping principle above
has the advantage that if the nonlinearity $F(\bb U,
\overline{\bb U})=|\bb U|
^{p-1}\bb U$ has a specific regularity, 
then it is inherited by the mapping data-solution. In particular,
following the
ideas in the proof of \cite[Corollary 5.6]{LinPon09}, we consider
for $(\bb V_0,
\bb V)\in B(\bb U_0;\epsilon)\times C([-T, T], \mathcal
E(\mathcal{G})) $ the
mapping
$$
\Gamma(\bb V_0, \bb V)(t)=\bb V(t)- J_{\bb V_0}[\bb V](t),\qquad
t\in [-T, T].
$$
Then $\Gamma(\bb U_0, \bb U)(t)=0$ for all $t\in [-T, T]$. For
$p-1$ being an even
integer, $F(\bb U, \overline{\bb U})$ is smooth, and therefore
$\Gamma$ is
smooth. Hence, using the arguments applied for obtaining the
local well-posedness
in $\mathcal E(\mathcal{G})$ above, we can show that the operator
$\partial_\bb
V\Gamma(\bb U_0, \bb U)$ is one-to-one and onto. Thus, by the
Implicit Function
Theorem there exists a smooth mapping $\bb \Lambda: B(\bb
U_0;\delta)\to C([-T,
T], \mathcal E(\mathcal{G}))$ such that $\Gamma(\bb V_0, \bb
\Lambda (\bb V_0))=0$
for all $\bb V_0\in B(\bb U_0;\delta)$. This argument establishes
the smoothness
property of the mapping data-solution associated to equation
\eqref{NLS_graph'}
when  $p-1$ is an even integer.  

If $p-1$ is not an even integer and $p>2$, then $F(\bb U,
\overline{\bb U})$ is
$C^{[p]}$-function, and consequently the mapping data-solution is
of class
$C^{[p]}$ (see \cite[Remark 5.7]{LinPon09}). Therefore, for $p>2$
we conclude that
the mapping data-solution is at least of class $C^{2}$.

Next,  we show that the unitary group 
$e^{-it\bb H_{\delta}^\alpha}$ 
preserves the subspace $\mathcal E_k(\mathcal G)$. 
Indeed,  let  $\bb V=(v_j) \in \mathcal E_k(\mathcal G)$, 
then  we obtain  
 $t_1(\mu)=...=t_k(\mu)$ and $t_{k+ 1}(\mu)=...=t_N(\mu)$, 
 where $ t_j(\mu)=\frac12  \int _0^\infty v_j(y)e^{-\mu y}dy$. 
 Hence, from \eqref{c_tilde_delta} it follows 
 $\tilde{c}_1=...=\tilde{c}_k$ and 
 $\tilde{c}_{k+ 1}=...=\tilde{c}_N$.  
  Thus, by \eqref{group0} we get 
$e^{-it\bb H_{\delta}^\alpha}\bb V\in \mathcal E_k(\mathcal G)$.
Lastly, the well-posedness in $\EE_k(\mathcal{G})$ follows 
from the uniqueness of the solution to the Cauchy problem in 
$\EE(\mathcal{G})$ and the invariance  of the space 
$\EE_k(\mathcal{G})$ for the  unitary group 
$e^{-it\bb H_{\delta}^\alpha}$ shown above.
\end{proof}

\begin{remark} 
\rm 
 $(i)$ 
In \cite[Proposition 2.2]{AdaNoj14} the authors proved that for
any solution to
Cauchy problem associated with \eqref{NLS_graph}, the
conservation of charge and
energy hold, i.e.
  $$
Q(\bb U(t))=||\bb U(t)||^2=||\bb
U_0||^2,\,\,\,\text{and}\,\,\,E_\alpha(\bb U(t))
= E_\alpha(\bb U_0),\,t\in[-T,T], 
$$
where $E_\alpha$ is defined for $\bb V=(v_j)_{j=1}^N\in \mathcal
E(\mathcal{G})$
by 
 \begin{equation*}\label{energy_delta}
E_\alpha(\bb V)=\tfrac 1{2}||\bb V'||^2 -\tfrac 1{p+1}||\bb
V||_{p+1}^{p+1}+
\tfrac{\alpha}{2}\left|v_1(0)\right|^2.
\end{equation*}
Using the Sobolev embedding theorem, Gagliardo-Nirenberg
inequality \eqref{G-N_graph0}, the above conservation laws, one can induce global
well-posedness of
\eqref{NLS_graph} for  $1<p<5$ (i.e. we can choose $T=+\infty$).

 $(ii)$
Observe that $E_\alpha\in C^2(\mathcal{E}(\mathcal{G}),
\mathbb{R})$ since
$p>1$. This fact allows us to apply the results by Ohta \cite{oh}
in our
instability analysis.
  
\end{remark}
 
Next we introduce the basic objects of the classical theory by
Grillakis, Shatah
and Strauss.
Consider the following two self-adjoint matrix operators
associated with $\bb
\Phi_{\alpha,\delta}=(\varphi_{\alpha,\delta})_{j=1}^N$
 \begin{align*} 
\mathbf{L}_{1,\alpha}&= \Big ( \Big (-\frac{d^2}{dx^2}+\omega-
p(\varphi_{\alpha,\delta})^{p-1}\Big)\delta_{k,j} \Big ),\\
\mathbf{L}_{2,\alpha}&= \Big (\Big(-\frac{d^2}{dx^2}+\omega-(\varphi_{\alpha,
\delta})^{p-1}\Big)\delta_{k,j} \Big ),\\ \dom&(\mathbf{L}
_{1,\alpha})=\dom(\mathbf{L}_{2,\alpha})=\bb D_{\alpha,\delta},
\end{align*}
where $\delta_{k,j}$ is the Kronecker symbol, $\bb
D_{\alpha,\delta}$ and $\varphi_{\alpha,\delta}$ are  defined
by \eqref{D_alpha} and 
\eqref{Phi_vect}.
The operators $\mathbf{L}_{1,\alpha}$ and $\mathbf{L}_{2,\alpha}$
are associated
with the functional $\bb S_\alpha$ defined by \eqref{S_graph} via
the following
equality 
\begin{equation*}\label{L_1+L_2}
(\bb S_\alpha)''(\mathbf{\Phi}_{\alpha,\delta})(\mathbf{U}, 
\mathbf{V})=(\mathbf{L}_{1,\alpha}\mathbf{U}_1,
\mathbf{V}_1)+(\mathbf{L}
_{2,\alpha}\mathbf{U}_2, \mathbf{V}_2),\end{equation*}
where $\mathbf{U}=\mathbf{U}_1+i\mathbf{U}_2$ and
$\mathbf{V}=\mathbf{V}
_1+i\mathbf{V}_2$. The vector functions
$\mathbf{U}_j,\mathbf{V}_j,\,j\in\{1,2\},
$ are assumed to be real-valued.

Formally $(\bb S_\alpha)''(\mathbf{\Phi}_{\alpha,\delta})$ can be
considered as a
self-adjoint $2N\times 2N$ matrix operator (see \cite{GrilSha87,
GrilSha90} for
the details)
$
\bb H_\alpha=\left(\begin{array}{cc} \bb L_{1,\alpha}& \bb 0 \\
\bb 0 & \bb
L_{2,\alpha} \end{array}\right).
$ 
Define 
 \begin{equation*}\label{p(omega)}
 p(\omega_0)=\left\{\begin{array}{ll}
1 &\, \text{if}\
\partial_\omega||\mathbf{\Phi}_{\alpha,\delta}||^2>0\ \text{at}\
\omega=\omega_0, \\
0 & \, \text{if}\
\partial_\omega||\mathbf{\Phi}_{\alpha,\delta}||^2<0\ \text{at}
\ \omega=\omega_0.
  \end{array}\right. 
\end{equation*}  
Having established \emph{Assumptions 1, 2} in \cite{GrilSha87}, 
i.e. well-posedness of the associated Cauchy problem 
(see Theorem \ref{well0}) and the existence of 
$C^1$ in $\omega$  standing wave,  
the  next stability/instability result follows  
from \cite[Theorem 3]{GrilSha87} and 
\cite[Corollary 3 and 4]{oh}.

\begin{theorem}\label{stabil_graph}
 Let $\alpha\neq 0$,\, $\omega>\tfrac{\alpha^2}{N^2}$, 
and $n(\bb H_\alpha)$ be the number of negative eigenvalues of
$\bb H_\alpha$.
Suppose also  that 
 
$1)$\,\,$\ker(\bb
L_{2,\alpha})=\Span\{\mathbf{\Phi}_{\alpha,\delta}\}$,
 
 $2)$\,\,$\ker(\bb  L_{1,\alpha})=\{\bb 0\}$, 
  
  $3)$\,\, the negative spectrum of $\bb  L_{1,\alpha}$ and  
$\bb L_{2,\alpha}$ consists of a finite number of negative
eigenvalues
(counting multiplicities),
  
$4)$\,\, the rest of  the  spectrum of $\bb  L_{1,\alpha}$ and $\bb  L_{2,\alpha}$ is positive and bounded away from zero. 
  \
  Then the following assertions hold.
\begin{itemize}
  \item[$(i)$] If $n(\bb H_\alpha)=p(\omega)=1$, 
  then   the standing wave $e^{i\omega t}
\mathbf{\Phi}_{\alpha,\delta}$ is orbitally stable in
$\mathcal{E}(\mathcal{G})$.
   \item[$(ii)$]  If $n(\bb H_\alpha)-p(\omega)=1$
    in $L^2_k(\mathcal{G})$, then   the standing 
wave $e^{i\omega t}\mathbf{\Phi}_{\alpha,\delta}$  is orbitally unstable in $\mathcal{E}_k(\mathcal{G})$ and, consequently, in     $\mathcal{E}(\mathcal{G})$.
\end{itemize}
\end{theorem}

\begin{remark}\label{nonstab} 
\rm
The instability part of the above 
theorem needs some additional comments.

$(i)$
 It is  known from \cite{GrilSha90}  that when 
 $n(\bb H_\alpha)-p(\omega)$ is odd, we obtain only 
 spectral instability of 
 $e^{i\omega t}\mathbf{\Phi}_{\alpha,\delta}$. 
 To obtain orbital instability due to 
 \cite[Theorem 6.1]{GrilSha90}, it is 
 sufficient to show estimate (6.2) in   
 \cite{GrilSha90} for the semigroup 
  $e^{t\bb A_\alpha} $ generated by 
$ 
\bb A_\alpha=\left(\begin{array}{cc}
 \bb 0& \bb L_{2,\alpha} \\ -\bb L_{1,\alpha}
  & \bb 0 \end{array}\right).
$ 
In our particular case it is not clear how to prove estimate
(6.2).

$(ii)$
When $n(\bb H_\alpha)=2$ 
 (which happens for  $\alpha>0$),
we can apply the results by  Ohta 
\cite[Corollary 3 and 4]{oh} to get the instability 
part of the above Theorem.    
We note that in this case the orbital instability 
follows without using spectral instability.

$(iii)$
 Generally,  to imply the orbital instability 
 from the spectral one, the approach by    
 \cite{HenPer82} can be used (see Theorem 2).  
 The key point of this method is to use the fact that  
 the mapping data-solution associated to the  
 model is  of class $C^2$. In particular, for 
 the NLS-$\delta$ and NLS-$\delta'$ models the 
 mapping data-solution is  of class $C^2$ as 
 $p> 2$ (see Theorem \ref{well0}  and \ref{well1}). 
 The approach by  \cite{HenPer82} have 
 been applied successfully in \cite{AngLop08} 
 and  \cite{AngNat16} for the models of KdV-type.
\end{remark}

Below we  describe  the spectrum of the operators 
$\mathbf{L}_{1,\alpha}$ and $\mathbf{L}_{2,\alpha}$  which will 
help us to verify the conditions of Theorem \ref{stabil_graph}. 
Our ideas are based on the extension theory of symmetric
operators
 and the perturbation theory.  For convenience of the reader 
and for the future references we formulate the following 
extension theory results (see \cite[Chapter IV, \S 14]{Nai67}).
\begin{proposition}\label{d5} (von Neumann decomposition)
Let $A$ be a closed densely defined  symmetric operator. 
Then the following decomposition holds
\begin{equation}\label{d6}
\dom(A^*)=\dom(A)\oplus\mathcal N_{+}(A)\oplus\mathcal N_{-}(A).
\end{equation}
Therefore, for $u\in \dom(A^*)$ such that  $u=f+f_i+f_{-i}$, 
with $f\in \dom(A)$, $f_{\pm i}\in \mathcal{N}_{\pm}(A)$,
we get
$ 
A^*u=Af+if_i-if_{-i}.
$
\end{proposition}

\begin{proposition}\label{semibounded}
Let $A$ be a densely defined lower semi-bounded symmetric
operator (that is, $A\geq mI$) with finite deficiency indices
$n_{\pm}(A)=k<\infty$ in the Hilbert space $\mathcal{H}$, and let
$\widetilde{A}$ be a self-adjoint extension of $A$. Then the
spectrum of $\widetilde{A}$ in $(-\infty, m)$ is discrete and
consists of at most $k$ eigenvalues counting multiplicities.
\end{proposition}

\begin{remark}
\rm
When $m=0$, Proposition \ref{semibounded} provides an estimate
for $n(\widetilde{A})$.
\end{remark}

 Below, using the perturbation theory we show the equality
$n(\mathbf{L}_{1,\alpha})=2$ in the space $L^2_k(\mathcal{G})$
for any $k\in\{1,..., N-1\}$, i.e.
$n(\mathbf{L}_{1,\alpha}|_{L^2_k(\mathcal{G})})=2$.
For this purpose let us define the following self-adjoint matrix
Schr\"odinger operator on $L^2(\mathcal{G})$ with Kirchhoff
condition at $\nu=0$
\begin{align}
\label{L^0_1} 
&\bb
L_{1,0}=
 \Big (\Big(-\frac{d^2}{dx^2}+\omega-p
 \varphi_0^{p-1}\Big)\delta_{i,j} \Big ),
\\
&\dom(\bb L_{1,0})= \Big \{\mathbf{V}\in H^2(\mathcal{G}):
v_1(0)=...=v_N(0),\,\,\sum\limits_{j=1}^N v_j'(0)=0 \Big \},
\notag
 \end{align}
where $\varphi_0= \Big [\frac{(p+1)\omega}{2}
\sech^2 \Big (\frac{(p-1)\sqrt{\omega}}{2}x
 \Big ) \Big ]^{\frac{1}{p-1}}, \, x>0,$ is the half-soliton for the classical NLS
model \eqref{NLS0}. 

Let $\bb \Phi_0=(\varphi_0)_{j=1}^N$, then it is not
difficult to see that
$ 
\bb \Phi_{\alpha,\delta}\to\bb \Phi_0,
$
as
$
\alpha\to 0,
$ in $ H^1(\mathcal G).$  
The following lemma states the analyticity of the family of
operators $(\mathbf{L}_{1,\alpha})$.
\begin{lemma}\label{analici} As a function of $\alpha$,
$(\mathbf{L}_{1,\alpha})$ is real-analytic family of self-adjoint
operators of type (B) in the sense of Kato.
\end{lemma}

\begin{proof}[\bf Proof.]
By \cite[Theorem VII-4.2]{kato}, it suffices to note that the
family of bilinear forms $(B_{1,\alpha})$ defined for $\bb
U=(u_j)_{j=1}^N,\bb V=(v_j)_{j=1}^N\in \mathcal{E}(\mathcal{G})$
by
\begin{equation*}
B_{1,\alpha}(\bb U,\bb V)=\sum\limits_{j=1}^N \int
_{0}^\infty(u_j'v_j'+\omega
u_jv_j-p(\varphi_{\alpha,\delta})^{p-1}u_jv_j)dx+\alpha
u_1(0)v_1(0), \\
\end{equation*}
  is   real-analytic  of type (B). 
\end{proof}
As we intend to study the negative spectrum of
$\mathbf{L}_{1,\alpha}$ using perturbation theory, we need to
describe spectral properties of $\bb L_{1,0}$ (which is a "limit
value" of $\mathbf{L}_{1,\alpha}$ as $\alpha\to 0$).
  \begin{theorem}\label{spect_L^0_1}
Let $\bb L_{1,0}$ be defined by \eqref{L^0_1} and
$k\in\left\{1,...,N-1\right\}$. Then the following assertions hold.
\begin{itemize}
\item[$(i)$] $\ker(\bb
L_{1,0})=\Span\{\hat{\bb\Phi}_{0,1},...,\hat{\bb\Phi}_{0,
N-1}\}$, where
  \begin{equation*} 
\hat{\bb{\Phi}}_{0,j}=(0,...,0,\underset{\bf
j}{\varphi'_{0}},\underset{\bf j+1}{-\varphi'_{0}},0,...,0).
 \end{equation*}
\item[$(ii)$] In the space $L^2_k(\mathcal{G})$ we have $\ker(\bb
L_{1,0})=\Span\{\mathbf{\widetilde{\Phi}}_{0,k}\}$, where
  \begin{equation}\label{Psi_0}
\mathbf{\widetilde{\Phi}}_{0,k}= \Big (\underset{\bf
1}{\tfrac{N-k}{k}\varphi'_{0}},..., \underset{\bf
k}{\tfrac{N-k}{k}\varphi'_{0}},\underset{\bf
k+1}{-\varphi'_{0}},...,\underset{\bf N}{-\varphi'_{0}} \Big ),
 \end{equation}
i.e.  $\ker(\bb
L_{1,0}|_{L^2_k(\mathcal{G})})=\Span\{\mathbf{\widetilde{\Phi}}_{0,k}\}$.
\item[$(iii)$] The operator $\bb L_{1,0}$ has one simple negative
eigenvalue in $L^2(\mathcal{G})$, i.e. $n(\bb L_{1,0})=1$.
Moreover, $\bb L_{1,0}$ has one simple negative eigenvalue in
$L^2_k(\mathcal{G})$ for any $k$, i.e.  $n(\bb
L_{1,0}|_{L^2_k(\mathcal{G})})=1$.
\item[$(iv)$] The rest of the spectrum of $\bb L_{1,0}$ is
positive and bounded away from zero, and $\sigma_{\emph{ess}}(\bb L_{1,0})=[\omega, \infty).$
  \end{itemize}
\end{theorem}

\begin{proof}[\bf Proof.]
$(i)$ Recall that the only $L^2(\mathbb{R}_+)$-solution to the
equation
$$
 -v''_j+\omega v_j-p\varphi_0^{p-1}v_j=0
 $$ 
 is $v_j=\varphi'_0$ (up to a factor). 
 Thus, any element of $\ker(\mathbf{L}_{1,0})$ 
has the form $\mathbf{V}=(v_j)_{j=1}^N=(c_j\varphi'_0)_{j=1}^N,\,
c_j\in\mathbb{R}$. The continuity
condition is satisfied since $\varphi'_0(0)=0$. Condition
$\sum _{j=1}^Nv_j'(0)=0$ gives rise to $(N-1)$-dimensional
kernel of $\mathbf{L}_{1,0}$. It is easily seen that the
functions $\hat{\bb{\Phi}}_{0,j},\, j=1,..., N-1,$ form basis
there.

$(ii)$
Arguing as in the previous item, we can see that $\ker(\bb
L_{1,0})$ is one-dimensional in $L^2_k(\mathcal{G})$, and it is
spanned on $\mathbf{\widetilde{\Phi}}_{0,k}$.

$(iii)$
The main idea of the proof is to apply Proposition
\ref{semibounded}. In what follows, we use the notation
$\opl_0= \Big (\Big(-\frac{d^2}{dx^2}+\omega-p\varphi_0^{p-1}
\Big)\delta_{k,j} \Big )$.
First, note that $\mathbf{L}_{1,0}$ is the self-adjoint extension
of the following symmetric operator (see Remark \ref{s-a_ext})
\begin{equation}\label{L_0_alpha}
\begin{array}{c}
\mathbf{L}_0=\opl_0,\\
\dom(\mathbf{L}_0)= \Big \{\mathbf{V}\in H^2(\mathcal{G}):
v_1(0)=...=v_N(0)=0, \sum\limits_{j=1}^N v_j'(0)=0  \Big \}.
\end{array}
\end{equation}

Below we show that the operator $\mathbf{L}_0$ is non-negative,
and $n_\pm(\mathbf{L}_0)=1$. Let us show that the adjoint
operator of $\mathbf{L}_0$ is given by
\begin{equation}
\label{adjoint_graph} 
\mathbf{L}_0^*=\opl_0,\quad
\dom(\mathbf{L}_0^*)=\left\{\mathbf{V}\in H^2(\mathcal{G}):
v_1(0)=...=v_N(0)\right\}. 
\end{equation}
Using standard arguments one can prove that
$\dom(\mathbf{L}_0^*)\subset H^2(\mathcal{G})$ and
$\mathbf{L}_0^*=\opl_0$ (see \cite[Chapter V,\S 17]{Nai67}).
Denoting
$$ 
D_0^*:= \{\mathbf{V}\in H^2(\mathcal{G}):
v_1(0)=\cdots=v_N(0) \},
$$
 we easily arrive at $D_0^*\subseteq \dom(\mathbf{L}_0^*)$.
Indeed, for any $\mathbf{U}=(u_j)_{j=1}^N\in D_0^*$ and $\bb
V=(v_j)_{j=1}^N\in \dom(\mathbf{L}_0)$ denoting
$\mathbf{U}^*=\opl_0(\mathbf{U})\in L^2(\mathcal{G})$, we get
\begin{align*} 
 (\mathbf{L}_0\mathbf{V}, \mathbf{U})
 =(\bb V,\opl_0(\mathbf{U}))+\sum\limits_{j=1}^N\left[-v'_j
u_j+v_ju'_j\right]_0^\infty
=(\bb V,\opl_0(\mathbf{U}))=(\bb V,
\bb U^*), 
\end{align*}
which, by definition of the adjoint operator, means that
$\mathbf{U}\in \dom(\mathbf{L}_0^*)$ or $D_0^*\subseteq
\dom(\mathbf{L}_0^*)$.

Let us show the inverse inclusion $D_0^*\supseteq
\dom(\mathbf{L}_0^*)$. Take $\bb U\in\dom(\mathbf{L}_0^*)$, then
for any $\bb V\in \dom(\mathbf{L}_0)$ we have
\begin{align*} 
(\mathbf{L}_0\bb V, \bb U)=(\bb V,
\opl_0(\bb
U))+\sum\limits_{j=1}^N\left[-v'_ju_j+v_ju'_j\right]_0^\infty=(\bb
V, \mathbf{L}_0^*\bb U)=(\bb V, \opl_0(\bb U)).
\end{align*}
Thus, we arrive at the equality 
\begin{equation}\label{adjoint}
\sum\limits_{j=1}^N\left[-v'_ju_j+v_ju'_j\right]_0^\infty=\sum\limits_{j=1}^N
v'_j(0)u_j(0)=0
\end{equation} 
for any $\bb V\in \dom(\mathbf{L}_0)$. 
Let ${\bb W}=(w_j)_{j=1}^N\in \dom(\mathbf{L}_0)$ be such that
$w'_3(0)= w'_4(0)=...= w'_N(0)= 0.$ Then for $\bb U\in
\dom(\mathbf{L}_0^*)$ from \eqref{adjoint} it follows that
\begin{equation}
\label{adjoint1} 
\sum\limits_{j=1}^N w_{j}
'(0)u_j(0)=w'_1(0)u_1(0)+w'_2(0)u_2(0)=0.
\end{equation}
Recalling that $\sum _{j=1}^N w_j'(0)= w_1'(0)+w_2'(0)=0$
and assuming $w_2'(0)\neq 0$, we obtain from \eqref{adjoint1} the
equality $u_1(0)=u_2(0)$.
Repeating the similar arguments for ${\bb W}=({w}_j)_{j=1}^N\in
\dom(\mathbf{L}_0)$ such that ${w}'_4(0)= {w}'_5(0)=...=
{w}'_N(0)=0$, we get $u_1(0)=u_2(0)=u_3(0)$ and so on. Finally
taking ${\bb W}=({w}_j)_{j=1}^N\in \dom(\mathbf{L}_0)$ such that
${w}'_N(0)= 0$, we arrive at $u_1(0)=u_2(0)=...=u_{N-1}(0)$, and
consequently $u_1(0)=u_2(0)=...=u_{N}(0)$.
Thus, $\bb U\in D_0^*$ or $D_0^*\supseteq \dom(\mathbf{L}_0^*)$,
and \eqref{adjoint_graph} holds.

Let us show that the operator $\mathbf{L}_0$ is non-negative.
First, note that every component of the vector
$\mathbf{V}=(v_j)_{j=1}^N\in H^2(\mathcal{G})$ satisfies the
following identity
\begin{equation*} 
-v_j''+\omega v_j-p\varphi_0^{p-1}v_j=
\frac{-1}{\varphi'_0}\frac{d}{dx}
 \Big [(\varphi'_0)^2\frac{d}{dx}
  \Big (\frac{v_j}{\varphi'_0} \Big ) \Big ],\quad
x>0.
\end{equation*}
Using the above equality and integrating by parts, we get for
$\mathbf{V}\in \dom(\mathbf{L}_0)$
 \begin{align*}
(\mathbf{L}_0\mathbf{V},\mathbf{V})&=
\sum\limits_{j=1}^N \int
^{\infty}_{0}(\varphi'_{0})^2
 \Big[\frac{d}{dx} \Big (\frac{v_j}{\varphi'_{0}} \Big ) \Big]^2dx+
\sum\limits_{j=1}^N
 \Big [-v_j'v_j+v_j^2\frac{\varphi''_{0}}{\varphi'_{0}} \Big ]^{\infty}_{0}\\&=\sum\limits_{j=1}^N
\int
^{\infty}_{0}(\varphi'_{0})^2 \Big [\frac{d}{dx}
 \Big (\frac{v_j}{\varphi'_{0}} \Big ) \Big ]^2dx\geq
0.
\end{align*}
Note that the equality 
$
\sum\limits_{j=1}^N \Big [-v_j'v_j+v_j^2\frac{\varphi''_{0}}{\varphi'_{0}} \Big ]^{\infty}_{0}=0
$ follows from the condition $v_j(0)=0$ and the fact
that $x=0$ is the first-order zero for $\varphi'_0(x)$ (i.e. 
$\varphi''_0(0)\neq 0$).

Due to the von  Neumann decomposition \eqref{d6},
\begin{equation*} 
\begin{array}{c}
\dom(\mathbf{L}_0^*)=\left\{\mathbf{V}\in H^2(\mathcal{G}):
v_1(0)=...=v_N(0)\right\}\\
\\ \;\;\;\;\;\;\;\;\;\quad
\quad=\dom(\mathbf{L}_0)\oplus
\Span\{\mathbf{V}_i\}\oplus\Span\{\mathbf{V}_{-i}\},
\end{array}
\end{equation*}
where $\mathbf{V}_{\pm i}= (e^{i\sqrt{\pm
i}x} )_{j=1}^N,\,\,\Im(\sqrt{\pm i})>0$. Indeed, since $\varphi_0\in L^\infty(\mathbb{R_+})$, it follows $
\dom(\mathbf{L}^*_0)=\dom(\mathbf{L}^*)=
\dom(\mathbf{L})\oplus\Span\{\mathbf{V}_i\}\oplus\Span\{\mathbf{V}_{-i}\}$, 
where
\begin{equation}\label{L_symm}\mathbf{L}= \Big (\Big(-\frac{d^2}{dx^2}\Big)\delta_{k,j} \Big ),
\quad \dom(\mathbf{L})=\dom(\mathbf{L}_0),\quad
\mathcal{N}_\pm(\bb L)=\Span\{\mathbf{V}_{\pm i}\}.\end{equation}
Since $n_\pm (\bb L)=1$, by \cite[Chapter IV, Theorem 6]{Nai67},
it follows that $n_{\pm}(\mathbf{L}_0)=1.$

Due to Proposition \ref{semibounded}, $n(\mathbf{L}_{1,0})\leq
1$.
For $\mathbf{\Phi}_0=(\varphi_{0})_{j=1}^N$ we obviously have
\newline
$(\mathbf{L}_{1,0}\mathbf{\Phi}_0,\mathbf{\Phi}_0)=-(p-1)||\mathbf{\Phi}_0||_{p+1}^{p+1}<0$.
By minimax principle, we arrive at $n(\mathbf{L}_{1,0})=1$.
Noting that $\bb \Phi_0\in L^2_k(\mathcal{G})$ for any $k$, we
get $n_{\pm}(\mathbf{L}_0|_{L^2_k(\mathcal{G})})=1.$

$(iv)$
By Weyl's theorem (see \cite[Theorem XIII.14]{RS}), the essential
spectrum of $\mathbf{L}_{1,0}$ coincides with $[\omega,\infty).$
Since $\bb\Phi_0\in L^\infty(\mathcal{G})$, there can be only
finitely many isolated eigenvalues in $(-\infty, \omega')$ for
any $\omega'<\omega$. Then $(iv)$ follows easily.
\end{proof}

\begin{remark}
\rm
Observe that, when we deal with deficiency indices, the operator
$\bb L_0$ is assumed to act on complex-valued functions which
however does not affect the analysis of negative spectrum of $\bb
L_{1,0}$ acting on real-valued functions.
\end{remark}

\begin{remark}
\label{s-a_ext} 
\rm
Let us show that the domain of any self-adjoint extension
$\widehat{\bb L}$ of the operator $\bb L_0$ defined by
\eqref{L_0_alpha}(and acting on complex-valued functions) is given by
$$
\dom(\widehat{\bb L})=
 \Big \{\mathbf{V}\in H^2(\mathcal{G}):
v_1(0)=...=v_N(0),\quad \sum\limits_{j=1}^N
v_j'(0)=zv_1(0),\,z\in\mathbb{R} \Big \}.$$
Indeed, due to \cite[Theorem A.1]{AlbGes05},
$$\dom(\widehat{\bb L})=\left\{\bb F=\bb F_0+c\bb
F_i+ce^{i\theta}\bb F_{-i}:\,\bb F_0\in \dom(\bb L_0),
c\in\mathbb{C},\theta\in[0,2\pi)\right\},$$
where $\bb F_{\pm i}= (\tfrac{i}{\sqrt{\pm i}}e^{i\sqrt{\pm
i}x} )_{j=1}^N,\,\,\Im(\sqrt{\pm i})>0$.
It is easily seen that for $\bb F\in \dom(\widehat{\bb L})$, we
have
$$\sum\limits_{j=1}^N(\bb F)'_j(0)=-Nc(1+e^{i\theta}),\quad (\bb
F)_j(0)=c\left(e^{i\pi/4}+e^{i(\theta-\pi/4)}\right).$$
From the last equalities it follows that  
$$\sum\limits_{j=1}^N(\bb F)'_j(0)=z(\bb F)_1(0),\,\,
\text{where}\,\,
z=\frac{-N(1+e^{i\theta})}{\left(e^{i\pi/4}+e^{i(\theta-\pi/4)}\right)}\in\mathbb{R}.$$
\end{remark}

Combining Lemma \ref{analici} and Theorem \ref{spect_L^0_1}, in
the framework of the perturbation theory we obtain the following
proposition.

\begin{proposition}\label{perteigen} Let
$k\in\left\{1,...,N-1\right\}$. Then there exist $\alpha_0>0$ and
two analytic functions  $\mu : (-\alpha_0,\alpha_0)\to \mathbb
R$ and $\bb F: (-\alpha_0,\alpha_0)\to L^2_k(\mathcal{G})$ such
that
\begin{enumerate}
\item[$(i)$] $\mu(0)=0$ and $\bb
F(0)=\mathbf{\widetilde{\Phi}}_{0,k}$, where
$\mathbf{\widetilde{\Phi}}_{0,k}$ is defined by \eqref{Psi_0}.

\item[$(ii)$] For all $\alpha\in (-\alpha_0,\alpha_0)$,
$\mu(\alpha)$ is the simple isolated second eigenvalue of $\bb
L_{1,\alpha}$ in $L^2_k(\mathcal{G})$, and $\bb F(\alpha)$ is the
associated eigenvector for $\mu(\alpha)$.

\item[$(iii)$] $\alpha_0$ can be chosen small enough to ensure
that for $\alpha\in (-\alpha_0,\alpha_0)$ the spectrum of $\bb
L_{1,\alpha}$ in $L^2_k(\mathcal{G})$ is positive, except at most
the first two eigenvalues.
\end{enumerate}
\end{proposition}

\begin{proof}[\bf Proof.]
Using the spectral structure of the operator $\bb L_{1,0}$ (see
Theorem \ref{spect_L^0_1}), we can separate the spectrum
$\sigma(\bb L_{1,0})$ into two parts $\sigma_0=\{\mu^0_{1,0},
0\}$ and $\sigma_1$ by a closed curve $\Gamma$ (for example, a
circle), such that $\sigma_0$ belongs to the inner domain of
$\Gamma$ and $\sigma_1$ to the outer domain of $\Gamma$ (note
that $\sigma_1\subset (\epsilon, +\infty)$ for $\epsilon>0$).
Next, Lemma \ref{analici} and the analytic perturbations theory
imply that $\Gamma\subset \rho(\bb L_{1,\alpha})$ for
sufficiently small $|\alpha |$, and $\sigma (\bb L_{1,\alpha})$
is likewise separated by $\Gamma$ into two parts, such that the
part of $\sigma (\bb L_{1,\alpha})$ inside $\Gamma$ consists of a
finite number of eigenvalues with total multiplicity (algebraic)
two. Therefore, we obtain from the Kato-Rellich Theorem (see
\cite[Theorem XII.8]{RS}) the existence of two analytic functions
$\mu, \bb F$ defined in a neighborhood of zero such that items
$(i)$, $(ii)$, and $(iii)$ hold.
\end{proof}
Below we investigate how the perturbed second eigenvalue moves
depending on the sign of $\alpha$.
\begin{proposition}\label{signeigen} There exists
$0<\alpha_1<\alpha_0$ such that $\mu(\alpha)>0$ for any
$\alpha\in (-\alpha_1,0)$, and $\mu(\alpha)<0$ for any $\alpha\in
(0, \alpha_1)$. Thus, in the space $ L^2_k(\mathcal G)$ for
$\alpha$ small, we have $n(\bb L_{1,\alpha})=1$ as $\alpha<0$,
and $n(\bb L_{1,\alpha})=2$ as $\alpha>0$.
\end{proposition}

\begin{proof}[\bf Proof.]
 From Taylor's theorem we have the following expansions
\begin{equation}\label{decomp1}
\mu(\alpha)=\mu_0 \alpha+ O(\alpha^2)\quad\text{and}\quad \bb
F(\alpha)=\bb{\widetilde{\Phi}}_{0,k}+ \alpha \bb F_0 + \bb
O(\alpha^2),
\end{equation}
where $\mu_0=\mu'(0)\in \mathbb R$, $\bb F_0=\partial_{\alpha}\bb
F(\alpha)|_{\alpha=0}\in L^2_k(\mathcal{G})$, and $\widetilde{\bb
\Phi}_{0,k}$ is defined by \eqref{Psi_0}. The desired result will
follow if we show that $\mu_0<0$. We compute $(\bb L_{1,\alpha}
\bb F(\alpha), \bb{\widetilde{\Phi}}_{0,k})$ in two different
ways.

In what follows, we will use the following decomposition for $\bb
\Phi_{\alpha,\delta}$ defined by \eqref{Phi_vect}
\begin{equation}\label{1a}
\bb \Phi_{\alpha,\delta}(\alpha)=\bb\Phi_0+\alpha\bb G_0+\bb
O(\alpha^2),\    \bb G_0=\partial_{\alpha}(\bb
\Phi_{\alpha,\delta})|_{\alpha=0}=\tfrac{-2}{(p-1)N\omega}\left(\varphi'_{0}\right)_{j=1}^N.
\end{equation} 
 From \eqref{decomp1} we obtain
\begin{equation}\label{1}
(\bb L_{1,\alpha} \bb F(\alpha),
\bb{\widetilde{\Phi}}_{0,k})=\mu_0
\alpha||\bb{\widetilde{\Phi}}_{0,k}||^2+ O(\alpha^2).
\end{equation}
By $\bb L_{1,0} \bb{\widetilde{\Phi}}_{0,k}=\bb 0$ and
\eqref{decomp1}, we get
\begin{align}
\label{2} 
\bb L_{1,\alpha}\bb{\widetilde{\Phi}}_{0,k}&=p\left((\bb
\Phi_0)^{p-1}-(\bb
\Phi_{\alpha,\delta})^{p-1}\right)\bb{\widetilde{\Phi}}_{0,k}\\
&=-\alpha p(p-1)(\bb \Phi_0)^{p-2}\bb
G_0\bb{\widetilde{\Phi}}_{0,k}+\bb O(\alpha^2).
\notag
\end{align}
 The operations in the last equality are
componentwise.
Equations \eqref{2} and \eqref{1a} induce
\begin{align}
\label{3} 
(\bb L_{1,\alpha} \bb F(\alpha),
\bb{\widetilde{\Phi}}_{0,k})&=-\left(\bb{\widetilde{\Phi}}_{0,k},
\alpha p(p-1)(\bb \Phi_0)^{p-2}\bb
G_0\bb{\widetilde{\Phi}}_{0,k}\right)+O(\alpha^2)\\
&=\tfrac{2\alpha p (N-k)}{k\omega} \int
_0^\infty(\varphi'_0)^3\varphi_0^{p-2}dx+O(\alpha^2).
\notag
\end{align}
Finally, combining \eqref{3} and \eqref{1}, we obtain for
$k\in\{1,..., N-1\}$
$$\mu_0=\frac{2p (N-k)}{k\omega ||\bb{\widetilde{\Phi}}_{0,k}||^2} \int
_0^\infty(\varphi'_0)^3\varphi_0^{p-2}dx+O(\alpha).
$$
It follows that $\mu_0$ is negative for sufficiently small
$|\alpha|$ (due to the negativity of $\varphi'_0$ on
$\mathbb{R}_+$), which in view of \eqref{decomp1} ends the proof.
\end{proof}

Now we can count the number of negative eigenvalues of $\bb
L_{1,\alpha}$ for any $\alpha$
using the classical continuation argument based on the
Riesz projection (see \cite{CozFuk08}) and the extension theory.
\begin{proposition}\label{n(L_1)} Let
$k\in\left\{1,...,N-1\right\}$ and $\alpha\neq 0$. Then
\begin{itemize}
\item[$(i)$] $\ker(\bb
L_{2,\alpha})=\Span\{\bb\Phi_{\alpha,\delta}\}$ and $\bb
L_{2,\alpha}\geq 0$,
\item[$(ii)$] $\ker(\bb L_{1,\alpha})=\{\bb 0\}$,
\item[$(iii)$] for $\alpha>0$, $n(\bb L_{1,\alpha})=2$ in
$L^2_k(\mathcal{G})$, i.e. $n(\bb
L_{1,\alpha}|_{L^2_k(\mathcal{G})})=2$,
\item[$(iv)$] for $\alpha<0$, $n(\bb L_{1,\alpha})=1$ in
$L^2_k(\mathcal{G})$, i.e. $n(\bb
L_{1,\alpha}|_{L^2_k(\mathcal{G})})=1$, moreover, $n(\bb
L_{1,\alpha})=1$ in $L^2(\mathcal{G})$.
\end{itemize} 
\end{proposition}

\begin{proof}[\bf Proof.] 
Assertions 
$(i)$-$(ii)$ were proved in \cite[Proposition 6.1]{AdaNoj14}.

$(iii)$
Recall that $\ker(\bb L_{1,\alpha})=\{\bb 0\}$ for $\alpha\neq
0$. Define $\alpha_\infty$ by

$$ 
\alpha_\infty=\inf \{\tilde \alpha>0: \bb
L_{1,\alpha}\;{\text{has  two negative eigenvalues}} 
{\text{ for all}}\; \alpha \in (0,\tilde \alpha)\}.
$$
Proposition \ref{signeigen} implies that $\alpha_\infty$ is well-defined and $\alpha_\infty\in (0,\infty]$. We claim that
$\alpha_\infty=\infty$. Suppose that $\alpha_\infty< \infty$. Let
$M=n(\bb L_{1,\alpha_{\infty}})$ and $\Gamma$ be a closed curve
(for example, a circle or a rectangle) such that $0\in
\Gamma\subset \rho(\bb L_{1,\alpha_{\infty}})$, and all the
negative eigenvalues of $\bb L_{1,\alpha_{\infty}}$ belong to the
inner domain of $\Gamma$. The existence of such $\Gamma$ can be
deduced from the lower semi-boundedness of the quadratic form
associated to $\bb L_{1,\alpha_{\infty}}$.
 
Next, from Lemma \ref{analici} it follows that there is
$\epsilon>0$ such that for $\alpha\in [\alpha_{\infty}-\epsilon,
\alpha_{\infty}+\epsilon]$ we have $\Gamma\subset \rho(\bb
L_{1,\alpha})$ and for $\xi \in \Gamma$,
$\alpha\to (\bb L_{1,\alpha}-\xi)^{-1}$ is analytic. Therefore,
the existence of an analytic family of Riesz projections
$\alpha\to P(\alpha)$ given by
$$
P(\alpha)=-\frac{1}{2\pi i} \int _{\Gamma} (\bb
L_{1,\alpha}-\xi)^{-1}d\xi
$$
implies that $\dim(\ran P(\alpha))=\dim(\ran P(\alpha_\infty))=M$
for all $\alpha\in [\alpha_\infty-\epsilon,
\alpha_{\infty}+\epsilon]$. Next, by definition of
$\alpha_\infty$, $\bb L_{1,\alpha_{\infty}-\epsilon} $ has two
negative eigenvalues, and $M=2$, hence $\bb L_{1,\alpha}$ has two
negative eigenvalues for $\alpha\in
(0,\alpha_{\infty}+\epsilon]$, which contradicts with the
definition of $\alpha_{\infty}$. Therefore,
$\alpha_{\infty}=\infty$.

$(iv)$
Analogously we can prove that $n(\bb L_{1,\alpha})=1$ in
$L^2_k(\mathcal{G})$ in the case $\alpha<0$.
To show the equality in the whole space $L^2(\mathcal{G})$, we
need to repeat the arguments of the proof of Theorem
\ref{spect_L^0_1}-$(iii)$ (i.e. $\bb L_{1,0}$ has to be replaced
by $\bb L_{1,\alpha}$, and $\bb \Phi_0$ by $\bb
\Phi_{\alpha,\delta}$). Namely, $\bb L_{1,\alpha}$ has to be
considered as the self-adjoint extension of the non-negative
symmetric operator
\begin{equation*}
\begin{array}{c}
\bb
L_{\alpha}= \Big (\Big(-\frac{d^2}{dx^2}+\omega-p(\varphi_{\alpha,\delta})^{p-1}\Big)\delta_{k,j} \Big ),\qquad\qquad\qquad\quad\\\dom(\mathbf{L}_\alpha)= \Big \{\mathbf{V}\in
H^2(\mathcal{G}): v_1(0)=...=v_N(0)=0,\quad \sum\limits_{j=1}^N
v_j'(0)=0  \Big \},
\end{array}
\end{equation*}
with deficiency indices $n_\pm(\bb L_{\alpha})=1$. Note that
since $\alpha<0$, we have $\varphi'_{\alpha,\delta}(x)<0,\, x\geq
0$.
\end{proof}

\begin{remark}\label{Rem_alpha_pos}
\rm

$(i)$
Using instruments of the
extension theory, it can be shown that \\ $n(\bb L_{1,\alpha})\leq
N$ in $L^2(\mathcal{G})$.

$(ii)$ Note that by Weyl's theorem (see \cite[Theorem
XIII.14]{RS}) the rest of the spectrum of $\bb L_{1,\alpha}$ and
$\bb L_{2,\alpha}$ in $L^2(\mathcal{G})$ is positive and bounded
away from zero, moreover, $\sigma_{\mathrm{ess}}(\bb L_{1,\alpha})=\sigma_{\mathrm{ess}}(\bb L_{2,\alpha})=[\omega,\infty).$
\end{remark}  

To apply  Theorem \ref{stabil_graph}, we need
to study the sign of
$\partial_\omega||\mathbf{\Phi}_{\alpha,\delta}||^2$.
\begin{proposition}\label{slope_graph_alpha}
Let $\omega> \tfrac{\alpha^2}{N^2}$ and
$J(\omega)=\partial_\omega||\mathbf{\Phi}_{\alpha,\delta}||^2$.
Then the following assertions hold.
\begin{itemize}
\item[$(i)$] Let $\alpha<0$, then

$1)$  for   $1<p\leq 5$,  we have $J(\omega)>0$;
 
$2)$ for $p>5$, there exists $\omega_1$ such that
$J(\omega_1)=0$, and $J(\omega)>0$ for
$\omega\in\left(\frac{\alpha^2}{N^2},\omega_1\right)$, while
$J(\omega)<0$ for $\omega\in (\omega_1,\infty)$.
\item[$(ii)$] Let $\alpha>0$, then

$1)$  for   $1<p\leq 3$,  we have $J(\omega)>0$;
 
$2)$ for $3<p<5$, there exists $\omega_2$ such that
$J(\omega_2)=0$, and $J(\omega)<0$ for
$\omega\in\left(\frac{\alpha^2}{N^2},\omega_2\right)$, while
$J(\omega)>0$ for $\omega\in (\omega_2,\infty)$;

 $3)$  for $p\geq 5$, we have $J(\omega)<0$.
\end{itemize} 
\end{proposition}

\begin{proof}[\bf Proof.]

To prove all the assertions, we will use the equality (see
\cite{AdaNoj14})
 \begin{equation*} 
 J(\omega)=C\omega^{\tfrac {7-3p}{2(p-1)}}J_1(\omega),
 \end{equation*}
where $C=\tfrac{N}{p-1} (\frac{p+1}{2} )^{\tfrac
2{p-1}}>0$ and
$$J_1(\omega)= \tfrac{5-p}{p-1} \int _{\tfrac{-\alpha}
{N\sqrt{\omega}}}^1(1-t^2)^{\tfrac
{3-p}{p-1}}dt+\tfrac{-\alpha}{N\sqrt{\omega}}(1-\tfrac{\alpha^2}{N^2\omega})^{\tfrac{3-p}{p-1}}.
$$
Thus,
\begin{equation*} 
J'_1(\omega)=\tfrac{-\alpha}{N\omega^{3/2}}\tfrac{3-p}{p-1} \Big [ \Big (1-\tfrac{\alpha^2}{N^2\omega} \Big )^{\tfrac{3-p}{p-1}}+\tfrac{\alpha^2}{N^2\omega} \Big (1-\tfrac{\alpha^2}{N^2\omega} \Big )^{-\tfrac{2(p-2)}{p-1}} \Big ].
\end{equation*}
Item $(i)$ was proved in \cite{AdaNoj14}. 

Let us prove the assertion $(ii)$. Item $3)$ is immediate.
Consider $p\in (1,5)$.
It is easily seen that
 \begin{equation}\label{slo4}
a_0= \lim_{\omega\to
+\infty}J_1(\omega)=\frac{5-p}{p-1}\int_0^{1}(1-t^2)^{\tfrac
{3-p}{p-1}}dt>0,
 \end{equation}
and 
\begin{equation}\label{slo5} 
\lim_{\omega\to \frac{\alpha^2}{N^2}}J_1(\omega)=\left\{
                     \begin{array}{ll}
                       2a_0, & \hbox{$p\in (1, 3]$,} \\
                       -\infty, & \hbox{$p\in (3,5)$.}
                     \end{array} \right.
 \end{equation}
Observing that $J'_1(\omega)\leq 0$ for $p\in (1, 3]$
($J'_1(\omega)\equiv0$ as $p=3$) and using
\eqref{slo4}-\eqref{slo5}, we get $J(\omega)>0$.
Let $p\in(3,5)$, then $J'_1(\omega)>0$. Thus, from
\eqref{slo4}-\eqref{slo5} it follows that there exists unique
$\omega_2>\tfrac{\alpha^2}{N^2}$ such that
$J_1(\omega_2)=J(\omega_2)=0$, and $J(\omega)<0$ for
$\omega\in (\frac{\alpha^2}{N^2},\omega_2 )$, while
$J(\omega)>0$ for $\omega\in (\omega_2,\infty)$.
 \end{proof}
 
\noindent\textbf{Proof of Theorem \ref{main_delta}.}
From Theorem \ref{well0}, we obtain well-posedness of
\eqref{NLS_graph} in $\mathcal{E}_k(\mathcal{G})$ for any
$k\in\{1,..., N-1\}$. For $\alpha>0$, from Proposition
\ref{n(L_1)}-$(iii)$ and Proposition \ref{slope_graph_alpha}
-$(ii)$, we obtain $
n(\bb H_\alpha |_{L^2_k(\mathcal{G})})-p(\omega)=1$ as $p\in (1,
3]$, $\omega>\frac{\alpha^2}{N^2}$, and $p\in (3,5)$,
$\omega>\omega_2$. Thus, from Theorem \ref{stabil_graph} we get
orbital instability of $e^{i\omega t}\bb \Phi_{\alpha, \delta}$
in $\mathcal{E}_k(\mathcal{G})$ and consequently  in   $\mathcal{E}(\mathcal{G})$.
 \qed

\begin{remark}\label{alpha_neg}
\rm
$(i)$
Let $p\geq 5$ and $\alpha>0$, then by Proposition
\ref{slope_graph_alpha}-$(ii)$ and Proposition
\ref{n(L_1)}-$(iii)$ we get $
n(\bb H_\alpha |_{L^2_k(\mathcal{G})})-p(\omega)=2$. This means
that Theorem \ref{stabil_graph} does not provide any information
about stability properties of $e^{i\omega t}\bb \Phi_{\alpha,
\delta}$ in $\mathcal{E}_k(\mathcal{G})$.

$(ii)$
Since the mapping data-solution is of class $C^2$ for $p> 2$, we
can apply the approach by \cite{HenPer82}, to imply the orbital
instability from the spectral one for $p\in (2,5)$.

$(iii)$
Theorem \ref{st_graph} above initially established in
\cite{AdaNoj14} easily follows for any $\alpha<0$ from our
approach. Indeed, combining Theorem \ref{well0}, Proposition
\ref{n(L_1)}-$(i)$-$(ii)$-$(iv)$, Proposition
\ref{slope_graph_alpha}-$(i)$ and Theorem \ref{stabil_graph} we
get the orbital stability of $e^{i\omega t}\bb \Phi_{\alpha,
\delta}$ in $\mathcal{E}(\mathcal{G})$ for $1<p\leq 5$. Moreover,
applying the approach by \cite{HenPer82} we may deduce the
orbital instability of $e^{i\omega t}\bb \Phi_{\alpha, \delta}$
from the spectral one for $p>5$ and $\omega>\omega_1$ (see
\cite[Remark 6.1]{AdaNoj14} where $\omega_1$ is denoted by $\omega^*$).

 \end{remark}

\subsection{The NLS-$\delta'$ equation on a star
graph}\label{NLS_delta'_graph}
 
As it was announced in the Introduction, in this Subsection we
discuss a new problem. In particular, we study the orbital
stability of the standing wave $\mathbf{U}(t,x)=e^{i\omega
t}\mathbf{\Phi}(x)$ of NLS-$\delta'$ equation \eqref{NLS_graph'}
with the particular $N$-tail profile
$\mathbf{\Phi}_{\lambda,\delta'}=(\varphi_{\lambda,j})_{j=1}^N$
satisfying the stationary equation
\begin{equation*}\label{H_alpha'}
\mathbf{H}_\lambda^{\delta'}\mathbf{\Phi}+\omega\mathbf{\Phi}-|\mathbf{\Phi}|^{p-1}\mathbf{\Phi}=0
\end{equation*}
under the conditions
$\varphi_{\lambda,1}=...=\varphi_{\lambda,N}=:\varphi_{\lambda,\delta'}$
and $N\varphi_{\lambda, j}(0)=\lambda \varphi'_{\lambda, j}(0)$.
It is easily seen that $\mathbf{\Phi}_{\lambda,\delta'}$ is
defined by \eqref{varphi_lam} for $\lambda<0$ and
$\omega>\tfrac{N^2}{\lambda^2}$.

As we are investigating orbital stability in $H^1(\mathcal{G})$
we need to show the well-posedness of the initial value problem
for equation \eqref{NLS_graph'} in this space (\textit{Assumption
2} in \cite{GrilSha90}). First, we establish the following
property for the unitary group associated to \eqref{NLS_graph'}.
\begin{lemma}\label{rela0} Let 
$\{e^{-it\bb H_\lambda^{\delta'}}\}_{t\in \mathbb R}$ be the
family of unitary operators associated to NLS-$\delta'$ model
\eqref{NLS_graph'}. Then for every $\bb V\in H^1(\mathcal{G})$
we have
$$ 
\partial_x(e^{-it\bb H_\lambda^{\delta'} }\bb V)=-e^{-it\bb
H_\lambda^{\delta'} }\bb V' +
\mathcal B(\bb V'),
$$
where $\mathcal B(\bb
V')=(2e^{it\partial^2_x}\tilde{v}_j)_{j=1}^N$, with $
\tilde{v}_j(x)=\left\{\begin{array}{c}
v'_j(x),\,\ x\geq 0,\\
0,\quad\;\;\;x<0
\end{array}\right.$, and $e^{it\partial^2_x}$ is the unitary
group associated with the free Schr\"odinger operator on
$\mathbb{R}$.
\end{lemma}
\begin{proof}[\bf Proof.]
The proof repeats the one of Lemma \ref{rela00}. The only
difference is that $\delta'$-interaction on $\mathcal{G}$ is
induced by the following condition

$$
\bb V\in \bb D_{\lambda,\delta'}\quad \text{iff}\quad
\tilde{A}\bb V(0)+\tilde{B}\bb V'(0)=\bb 0,\,\,\,\text{where}$$
\small
 $$
 \tilde{A}=\left(\begin{array}{ccccc}
  0& & ... & &0\\
  0&  & & & 0\\
  \vdots & & &  &\vdots\\
  & & & & \\
  -1 & &...& &-1
  \end{array}\right), \quad \tilde{B}=\left(\begin{array}{ccccc}
  1&-1&0&...&0\\
  0&1&-1&...&0\\
  \vdots &\vdots&\vdots&  &\vdots\\
  0&0&0&...&-1\\
\tfrac{\lambda}{N}&\tfrac{\lambda}{N}&\tfrac{\lambda}{N}&...&\tfrac{\lambda}{N}
  \end{array}\right).
  $$
  \normalsize
\end{proof}  
  
 \begin{theorem}\label{well1} 
Let $p>1$. Then for any $\bb U_0\in
H^1(\mathcal{G})$ there exists $T > 0$ such that 
equation \eqref{NLS_graph} has a unique solution $\bb U \in C
([-T,T],
H^1(\mathcal{G}))
\cap C^1 ([-T,T], [H^{1}(\mathcal{G})]')$ satisfying $\bb U(0)=\bb
U_0$. For each $T_0\in (0, T)$ the mapping
$ 
\bb U_0\in H^1(\mathcal{G})\to \bb U \in C ([-T_0,T_0],
H^1(\mathcal{G}))
$
is continuous. In particular, for $p>2$ this mapping is at least
of class $C^2$.

Moreover, the conservation of energy and charge holds:

$ 
E_\lambda(\bb U(t)) = E_\lambda(\bb U_0),\,\,\, \text{and}\,\,\,
Q(\bb U(t))=||\bb U(t)||^2=||\bb U_0||^2,\,\,t\in [-T, T],
$ 
where the energy $E_\lambda$ is defined for $\bb
V=(v_j)_{j=1}^N\in H^1(\mathcal{G})$ by
 \begin{equation*}\label{energy_delta'}
E_\lambda(\bb V)=\tfrac 1{2}||\bb V'||^2 -\tfrac 1{p+1}||\bb
V||_{p+1}^{p+1}+\tfrac{1}{2\lambda} \Big |\sum\limits_{j=1}^Nv_j(0) \Big |^2.
\end{equation*}
Consequently, for $1<p<5$, we can choose $T=+\infty$.
\end{theorem}

\begin{proof}[\bf Proof.]
The prove repeats the one of Theorem \ref{well0}. In particular,
it essentially uses \\ Lemma~\ref{rela0} and the Banach contraction
theorem.
\end{proof}

\begin{remark}
\rm
Analogously to the case of NLS-$\delta$ equation the following
equality holds
$
e^{-it\bb H_{\delta'}^\lambda} \bb V= e^{-it\bb
H_{\delta'}^\lambda}\bb P_c \bb V+ e^{-it\bb
H_{\delta'}^\lambda}\bb P_p \bb V.
$
Similarly, for $\lambda>0$, we have $\sigma_c(\bb
H_{\delta'}^\lambda)=[0,\infty)$ and $\sigma_p(\bb
H_{\delta'}^\lambda)=\emptyset$, therefore $\bb P_p=\bb 0$. For $\lambda<0$, $\sigma_c(\bb H_{\delta'}^\lambda)=[0,\infty)$
and $\sigma_p(\bb
H_{\delta'}^\lambda)=\{-z_0^2\}=\{-\frac{N^2}{\lambda^2}\}$, where
the corresponding eigenfunction is $\bb
V_{z_0}(x)=(e^{\frac{N}{\lambda}x})_{j=1}^N$, and therefore
$e^{-it\bb H_{\delta'}^\lambda}\bb P_p \bb V=e^{itz_0^2}(\bb
V,\bb V_{z_0})\bb V_{z_0}$.

The proof of the spectral properties of $\bb H_{\delta'}^\lambda$
repeats the one of \cite[Theorem 4.3]{AlbGes05} for the case of
the Schr\"odinger operator with $\delta'$-interaction on the
line. In particular, to describe the point spectrum for
$\lambda<0$ one needs to consider $\bb H_{\delta'}^\lambda$ as
the self-adjoint extension of the symmetric non-negative operator
$\bb L'$ defined by \eqref{L'_symm} with deficiency indices
$n_\pm(\bb L')=1$ and then to apply Proposition
\ref{semibounded}.
\end{remark}

Consider the following two self-adjoint matrix operators 
\begin{align*}
\label{linear_lambda'} 
\mathbf{L}_{1,\lambda}&= \Big (\Big(-\frac{d^2}{dx^2}+\omega-p(\varphi_{\lambda,\delta'})^{p-1}\Big)\delta_{k,j} \Big ),\\
\mathbf{L}_{2,\lambda}&= \Big (\Big(-\frac{d^2}{dx^2}+\omega-(\varphi_{\lambda,\delta'})^{p-1}\Big)\delta_{k,j} \Big ),
\end{align*}
with $
\dom(\mathbf{L}_{1,\lambda})=\dom(\mathbf{L}_{2,\lambda})=\bb
D_{\lambda,\delta'}$. Here $\delta_{k,j}$ is the 
Kronecker symbol.
These operators are associated in a standard way with the second
derivative of the following action functional
\begin{equation*}
\bb S_\lambda(\bb V)=\tfrac 1{2}||\bb V'||^2 -\tfrac 1{p+1}||\bb
V||_{p+1}^{p+1}+\tfrac{1}{2\lambda} \Big |\sum\limits_{j=1}^Nv_j(0) \Big |^2
+\tfrac{\omega}{2}||\bb V||^2,
\end{equation*}
where $\bb V=(v_j)_{j=1}^N\in H^1(\mathcal{G})$. 
Namely, $(\bb S_\lambda)''(\mathbf{\Phi}_{\lambda,\delta'})(\mathbf{U},
\mathbf{V})=(\mathbf{L}_{1,\lambda}\mathbf{U}_1,
\mathbf{V}_1)+(\mathbf{L}_{2,\lambda}\mathbf{U}_2, \mathbf{V}_2)$ with
$\mathbf{U}=\mathbf{U}_1+i\mathbf{U}_2$ and
$\mathbf{V}=\mathbf{V}_1+i\mathbf{V}_2$.
As in the previous paragraph, we consider the form $(\bb
S_\lambda)''(\mathbf{\Phi}_{\lambda,\delta'})$ as a linear
operator \begin{equation}\label{form_L1+L2'}\bb
H_\lambda=\left(\begin{array}{cc} \bb L_{1,\lambda}& \bb 0 \\ \bb
0 & \bb L_{2,\lambda} \end{array}\right).\end{equation}

The energy functional $E_\lambda$ defined by \eqref{energy_delta'} belongs to $C^2(H^1(\mathcal{G}),\mathbb{R})$  and \textit{Assumptions 1,2} in \cite{GrilSha87}  are satisfied. Thus, the analog  of  stability/instability  Theorem \ref{stabil_graph} is true for $e^{i\omega t}\mathbf{\Phi}_{\lambda,\delta'}$.

Below we give the description of the spectrum of the operators $\mathbf{L}_{1,\lambda}$ and $\mathbf{L}_{2,\lambda}$, which due to formula \eqref{form_L1+L2'}, will help us to verify the conditions of mentioned stability/instability result.
\begin{proposition}\label{spec_graph'}
Let $\lambda<0$ and $\omega >\tfrac{N^2}{\lambda^2}$, then the following results hold.
\begin{itemize}
\item[$(i)$]
$\ker(\mathbf{L}_{2,\lambda})=\Span\{\mathbf{\Phi}_{\lambda,\delta'}\}$,
and $\mathbf{L}_{2,\lambda}\geq 0$.
\item[$(ii)$] If $\omega <
\tfrac{N^2}{\lambda^2}\tfrac{p+1}{p-1}$, then
$\ker(\mathbf{L}_{1,\lambda})=\{\mathbf{0}\}$, and
$n(\mathbf{L}_{1,\lambda})=1$.
\item[$(iii)$] If $\omega =
\tfrac{N^2}{\lambda^2}\tfrac{p+1}{p-1}$, then
$n(\mathbf{L}_{1,\lambda})=1$, and the kernel of
$\mathbf{L}_{1,\lambda}$ is given by
$\ker(\mathbf{L}_{1,\lambda})=\Span\{\hat{\mathbf{\Phi}}_{\lambda,1},..,
\hat{\mathbf{\Phi}}_{\lambda,N-1}\}$, where
\begin{equation}\label{Psi}
\hat{\mathbf{\Phi}}_{\lambda,j}=(0,..,0,\underset{\bf
j}{\varphi'_{\lambda,\delta'}},\underset{\bf
j+1}{-\varphi'_{\lambda,\delta'}},0,..,0).
 \end{equation}
\item[$(iv)$] If $\omega >
\tfrac{N^2}{\lambda^2}\tfrac{p+1}{p-1}$, then
$\ker(\mathbf{L}_{1,\lambda})=\{\mathbf{0}\}$, and
$n(\mathbf{L}_{1,\lambda})\leq N$. Moreover, for $N$ even in the
space $L^2_{\frac{N}{2}}(\mathcal{G})$ we have
$n(\mathbf{L}_{1,\lambda}|_{L^2_{\frac{N}{2}}(\mathcal{G})})=2$.
\item[$(v)$] The rest of the spectrum of $\mathbf{L}_{1,\lambda}$
and $\mathbf{L}_{2,\lambda}$ is positive and bounded away from
zero, moreover, $\sigma_{\mathrm{ess}}(\bb L_{1,\lambda})=\sigma_{\mathrm{ess}}(\bb L_{2,\lambda})=[\omega,\infty).$
\end{itemize}
\end{proposition}

\begin{proof}[\bf Proof.]
 
 $(i)$
It is clear that
$\mathbf{\Phi}_{\lambda,\delta'}\in\ker(\mathbf{L}_{2,\lambda})$.
To show the equality
$\ker(\mathbf{L}_{2,\lambda})=\Span\{\mathbf{\Phi}_{\lambda,\delta'}\}$
let us note that any $\mathbf{V}=(v_j)_{j=1}^N\in
H^2(\mathcal{G})$ satisfies the following identity
\begin{equation}\label{identity_graph'}
-v_j''+\omega v_j-(\varphi_{\lambda,\delta'})^{p-1}v_j=
\frac{-1}{\varphi_{\lambda,\delta'}}\frac{d}{dx} \Big [\varphi_{\lambda,\delta'}^2\frac{d}{dx} \Big (\frac{v_j}{\varphi_{\lambda,\delta'}} \Big ) \Big ],\quad
x>0.
\end{equation}
Thus, for $\mathbf{V}\in \bb D_{\lambda,\delta'}$,
 we obtain from \eqref{identity_graph'}, \eqref{H_lambda}, and 
\eqref{varphi_lam}
  \begin{align*}
 &(\mathbf{L}_{2,\lambda}\mathbf{V},\mathbf{V})=
\sum\limits_{j=1}^N
\int
^{\infty}_{0}\varphi_{\lambda,\delta'}^2 \Big [\frac{d}{dx} \Big (\frac{v_j}{\varphi_{\lambda,\delta'}} \Big ) \Big ]^2dx+ R_{\lambda, N},\,\,\text{where}\\
&R_{\lambda, N}=\sum\limits_{j=1}^N \Big [v_j'(0)v_j(0)-v_j^2(0)\frac{\varphi'_{\lambda,\delta'}(0)}{\varphi_{\lambda,\delta'}(0)} \Big ]=\frac{1}{\lambda}\Big [\sum\limits_{j=1}^Nv_j(0) \Big ]^2-\frac{N}{\lambda}\sum\limits_{j=1}^Nv_j^2(0).
\end{align*}
The term $R_{\lambda, N}$ if positive for  $\lambda<0$ by Jensen's inequality applied to  $f(x)=x^2$. Thus, $(\mathbf{L}_{2,\lambda}\mathbf{V},\mathbf{V})>0$ for
$\mathbf{V}\in \bb
D_{\lambda,\delta'}\setminus\Span\{\mathbf{\Phi}_{\lambda,\delta'}\}$
which proves $(i)$.

 $(ii)$ 
Concerning the kernel of $\mathbf{L}_{1,\lambda}$, we recall that
the only $L^2(\mathbb{R}_+)$-solution of the equation $
 -v''_j+\omega v_j-p(\varphi_{\lambda,\delta'})^{p-1}v_j=0$  
is given by $v_j=\varphi'_{\lambda,\delta'}$ (up to a factor).
Thus, any element of $\ker(\mathbf{L}_{1,\lambda})$ has the form
$\mathbf{V}=(v_j)_{j=1}^N=(c_j\varphi'_{\lambda,\delta'})_{j=1}^N,\,
c_j\in\mathbb{R}$.
If $v'_1(0)=...=v'_N(0)\neq 0$, then by \eqref{H_lambda} we get
$c_1=...=c_N\neq 0$, and consequently $N
\varphi'_{\lambda,\delta'}(0)=\lambda\varphi''_{\lambda,\delta'}(0)$.
Therefore, $\omega=\frac{N^2}{\lambda^2}$, which is impossible.
Otherwise, the condition $v'_j(0)=0$ implies that
$\varphi''_{\lambda,\delta'}(0)=0$, which is equivalent to the
identity $\omega=\frac{N^2}{\lambda^2}\frac{p+1}{p-1}$. Thus, we
get that $c_1=...=c_N=0$ and $\mathbf{V}\equiv\mathbf{0}$ for
$\omega\neq \frac{N^2}{\lambda^2}\frac{p+1}{p-1}$.

The proof of the equality $n(\mathbf{L}_{1,\lambda})=1$ for
$\omega < \tfrac{N^2}{\lambda^2}\tfrac{p+1}{p-1}$ is similar to
the one in the case of the operator $\bb L_{1,0}$ defined by
\eqref{L^0_1}. Namely, denoting
\begin{equation}\label{l}
\opl_\lambda= \Big (\Big(-\frac{d^2}{dx^2}+\omega-p(\varphi_{\lambda,\delta'})^{p-1}\Big)\delta_{k,j} \Big ),
\end{equation}
we define the following symmetric operator
$\mathbf{L}'_0=\opl_\lambda$ with
\begin{equation*}
\dom(\mathbf{L}'_0)= \Big \{\mathbf{V}\in H^2(\mathcal{G}):
v'_1(0)=...=v'_N(0)=0, \;\sum\limits_{j=1}^N v_j(0)=0  \Big \}.
\end{equation*}
It is easily seen that $\mathbf{L}_{1,\lambda}$ is the
self-adjoint extension of $\mathbf{L}'_0$.
Let us show that the operator $\mathbf{L}'_0$ is non-negative.
First, note that any $\mathbf{V}=(v_j)_{j=1}^N\in
H^2(\mathcal{G})$ satisfies the following identity
\begin{equation*}
-v_j''+\omega v_j-p(\varphi_{\lambda,\delta'})^{p-1}v_j=
\frac{-1}{\varphi'_{\lambda,\delta'}}\frac{d}{dx} \Big [(\varphi'_{\lambda,\delta'})^2\frac{d}{dx} \Big (\frac{v_j}{\varphi'_{\lambda,\delta'}} \Big ) \Big ],\quad
x>0.
\end{equation*}
Using the above equality and integrating by parts, we get for
$\mathbf{V}\in \dom(\mathbf{L}'_0)$
 \begin{align}
  \label{nonneg_graph'}
\notag
(\mathbf{L}'_0\mathbf{V},\mathbf{V})=
\sum\limits_{j=1}^N \int
^{\infty}_{0}(\varphi'_{\lambda,\delta'})^2 \Big [\frac{d}{dx} \Big (\frac{v_j}{\varphi'_{\lambda,\delta'}} \Big ) \Big ]^2dx-\sum\limits_{j=1}^Nv_j^2(0)\frac{\varphi''_{\lambda,\delta'}(0)}{\varphi'_{\lambda,\delta'}(0)}.
 \end{align}
Taking into account that 
\begin{equation}\label{3.31a}
-v_j^2(0)\frac{\varphi''_{\lambda,\delta'}(0)}{\varphi'_{\lambda,\delta'}(0)}=v_j^2(0)\frac{\lambda\omega}{2N} \Big (p-1-(p+1)\frac{N^2}{\lambda^2\omega} \Big ),
\end{equation}
we get non-negativity of $\mathbf{L}'_0$ for
$\omega\leq\frac{N^2}{\lambda^2}\frac{p+1}{p-1}$.

Next, the adjoint operator of $\bb L'_0$ is given by 
\begin{equation*}
\label{adjoint_graph'} 
(\mathbf{L}'_0)^*=\opl_\lambda,\quad
\dom((\mathbf{L}'_0)^*)=\left\{\mathbf{V}\in H^2(\mathcal{G}):
v'_1(0)=...=v'_N(0)\right\}. 
\end{equation*}
The last formula can be shown analogously to
\eqref{adjoint_graph}. Due to the von Neumann decomposition
\eqref{d6}, we get (assuming that $\mathbf{L}'_0$  acts on complex-valued functions)
\begin{align*}
\dom((\mathbf{L}'_0)^*)=\dom(\mathbf{L}'_0)\oplus\Span\{\mathbf{V}_i\}\oplus\Span\{\mathbf{V}_{-i}\},
\end{align*}
where $\mathbf{V}_{\pm i}= (e^{i\sqrt{\pm
i}x} )_{j=1}^N,\, \Im(\sqrt{\pm i})>0$.
Indeed, since $\varphi_{\lambda,\delta'}\in
L^\infty(\mathbb{R_+})$, we get
$\dom((\mathbf{L}'_0)^*)=\dom((\mathbf{L'})^*)$, where
\begin{equation}\label{L'_symm}\mathbf{L'}= \Big (\Big(-\frac{d^2}{dx^2}\Big)\delta_{k,j} \Big ),
\quad \dom(\mathbf{L'})=\dom(\mathbf{L}'_0).\end{equation}
Finally, by \cite[Chapter IV, Theorem 6]{Nai67},
$n_{\pm}(\mathbf{L}'_0)=n_\pm(\mathbf{L'})=1.$
By Proposition \ref{semibounded}, $n(\mathbf{L}_{1,\lambda})\leq
1$.
Due to
$(\mathbf{L}_{1,\lambda}\mathbf{\Phi}_{\lambda,\delta'},\mathbf{\Phi}_{\lambda,\delta'})=-(p-1)||\mathbf{\Phi}_{\lambda,\delta'}||_{p+1}^{p+1}<0$,
we finally arrive at $n(\mathbf{L}_{1,\lambda})=1$, and $(ii)$ is
proved.

$(iii)$
From the proof of item $(ii)$ we induce that
$n(\mathbf{L}_{1,\lambda})=1$, and the kernel of
$\mathbf{L}_{1,\lambda}$ is nonempty as
$\omega=\frac{N^2}{\lambda^2}\frac{p+1}{p-1}$. Moreover, we know
that any element of the kernel has the form
$\mathbf{V}=(v_j)_{j=1}^N=(c_j\varphi'_{\lambda,\delta'})_{j=1}^N,\,
c_j\in\mathbb{R}$, and it is necessary that
$v_1'(0)=...=v_N'(0)=0$. Hence, the condition
\begin{equation}\label{condi}
\lambda v'_1(0)=\sum\limits_{j=1}^N v_j(0)=0
\end{equation} 
gives rise to $(N-1)$-dimensional kernel of
$\mathbf{L}_{1,\lambda}$. Since the functions
$\hat{\mathbf{\Phi}}_{\lambda,j}$, $1\leq j\leq N-1$, defined in
\eqref{Psi} are linearly independent and satisfy  condition
\eqref{condi}, they form the basis in
$\ker(\mathbf{L}_{1,\lambda})$, and $(iii)$ is proved.

$(iv)$
The identity $\ker(\mathbf{L}_{1,\lambda})=\{\mathbf{0}\}$ was
shown in $(ii)$.
To show the inequality $n(\mathbf{L}_{1,\lambda})\leq N$ we
introduce the following minimal symmetric operator
$\mathbf{L}_{\min}=\opl_\lambda$ with
\begin{equation}\label{L_min}
\dom(\mathbf{L}_{\min})=\left\{\mathbf{V}\in
H^2(\mathcal{G}):\begin{array}{c}
v'_1(0)=...=v'_N(0)=0,\\
v_1(0)=...=v_N(0)=0
\end{array}\right\},
\end{equation}
where $\opl_\lambda$ is defined in \eqref{l}. The operator
$\mathbf{L}_{1,\lambda}$ is  self-adjoint extension of
$\mathbf{L}_{\min}$.
From the formula \eqref{3.31a}  it follows that
$\mathbf{L}_{\min}$ is a  non-negative operator.
It is obvious that  $\mathbf{L}_{\min}^*=\opl_\lambda$,
$
 \dom(\mathbf{L}_{\min}^*)=
H^2(\mathcal{G}).
$
Then, due to the von Neumann formula (for $\bb L_{\min}$ acting on complex-valued functions)
$$\dom(\mathbf{L}_{\min}^*)=\dom(\mathbf{L}_{\min})\oplus\Span\{\mathbf{V}^1_i,..,\mathbf{V}^N_i\}\oplus\Span\{\mathbf{V}^1_{-i},..,\mathbf{V}^N_{-i}\},$$
where
$\mathbf{V}^j_{\pm i}=(0,...,
\underset{\bf j}{e^{i\sqrt{\pm i}x}},0,...,0),$
$\Im(\sqrt{\pm i})>0$, and consequently 
$n_{\pm}(\mathbf{L}_{\min})$ $=N$.   
By Proposition \ref{semibounded}, $n(\mathbf{L}_{1,\lambda})\leq
N$.

Let $N$ be even. It is easily seen that
$n_{\pm}(\mathbf{L}_{\min})=2$ in
$L^2_{\frac{N}{2}}(\mathcal{G})$. Indeed,
$  \dom(\mathbf{L}_{\min}^*)=\dom(\mathbf{L}_{\min})\oplus
\Span\{\widetilde{\bf V}^1_{i}, \widetilde{\bf
V}^2_{i}\}\oplus\Span\{\widetilde{\bf V}^1_{-i}, \widetilde{\bf
V}^2_{-i}\}, $
 where 
$ 
\widetilde{\bf V}^1_{\pm i}=(\underset{\bf 1}{e^{i\sqrt{\pm
i}x}},...,\underset{\bf N/2}{e^{i\sqrt{\pm i}x}},\underset{\bf
N/2+1}{0},...,\underset{\bf N}{0}),
$  
and
$ 
\widetilde{\bf V}^2_{\pm i}=(\underset{\bf
1}{0},...,\underset{\bf N/2}{0},\underset{\bf
N/2+1}{e^{i\sqrt{\pm i}x}},..., \underset{\bf N}{e^{i\sqrt{\pm
i}x}}).
$ 
By Proposition \ref{semibounded}, we get
$n(\mathbf{L}_{1,\lambda})\leq 2$ in
$L^2_{\frac{N}{2}}(\mathcal{G})$.

Let us introduce the following quadratic form $\bb F_{1,\lambda}$
associated with the operator $\mathbf{L}_{1,\lambda}$
\begin{equation*}
\bb
F_{1,\lambda}(\mathbf{V})=||\mathbf{V}'||^2+\omega||\mathbf{V}||^2-p\sum\limits_{j=1}^N
\int
_{0}^{\infty}(\varphi_{\lambda,\delta'})^{p-1}|v_j|^2dx+\tfrac
1{\lambda} \Big |\sum\limits_{j=1}^Nv_j(0) \Big |^2,
\end{equation*}
with $\dom(\bb F_{1,\lambda})=H^1(\mathcal{G})$. Let $
\bf \Phi^-_\lambda=(\underset{\bf
1}{\varphi'_{\lambda,\delta'}},..., \underset{\bf
N/2}{\varphi'_{\lambda,\delta'}},\underset{\bf
N/2+1}{-\varphi'_{\lambda,\delta'}},...,\underset{\bf
N}{-\varphi'_{\lambda,\delta'}}),
$
then integrating by parts we obtain
\begin{align*}
&\bb F_{1,\lambda}(\bb\Phi^-_\lambda)=N \int
_{0}^\infty\varphi'_{\lambda,\delta'}\left(-\varphi'''_{\lambda,\delta'}+\omega\varphi'_{\lambda,\delta'}-p(\varphi_{\lambda,\delta'})^{p-1}\varphi'_{\lambda,\delta'}\right)dx\\
&-N\varphi'_{\lambda,\delta'}(0)\varphi''_{\lambda,\delta'}(0)
=\tfrac{N^2}{2\lambda}\omega\left[\left(\tfrac{(p+1)\omega}{2}\right)\left(1-\tfrac{N^2}{\lambda^2\omega}\right)\right]^{\tfrac
2{p-1}}\left(p-1-(p+1)\tfrac{N^2}{\lambda^2\omega}\right),
\end{align*} 
which is negative for
$\omega>\frac{N^2}{\lambda^2}\frac{p+1}{p-1}$. Since
$(\mathbf{L}_{1,\lambda}\mathbf{\Phi}_{\lambda,\delta'},\mathbf{\Phi}_{\lambda,\delta'})<0$,
we get by orthogonality of $\bb\Phi^-_\lambda$ and
$\mathbf{\Phi}_{\lambda,\delta'}$
$$
\bb F_{1,\lambda}(s\mathbf{\Phi}_{\lambda,\delta'}+r\bb
\Phi^-_\lambda)=|s|^2\bb F_{1,\lambda}(\mathbf{\Phi}_{\lambda,\delta'})+|r|^2\bb F_{1,\lambda}(\bb
\Phi^-_\lambda)<0.
$$
Thus, we obtain that $\bb F_{1,\lambda}$ is negative on
two-dimensional subspace
$\mathcal{M}=\Span\{\mathbf{\Phi}_{\lambda,\delta'},\bf
\Phi^-_\lambda\}$. Therefore, by minimax principle, we get
$n(\mathbf{L}_{1,\lambda})\geq 2$.
 The assertion $(iv)$ is proved.
The proof of item $(v)$ is standard and relies on Weyl's theorem.
This finishes the proof of the Proposition.
\end{proof} 

Finally, we have to study the sign of
$\partial_\omega||\mathbf{\Phi}_{\lambda,\delta'}||^2$.
\begin{proposition}\label{slope_graph'}
Let $\omega> \tfrac{N^2}{\lambda^2}$, $\lambda<0$, and
$J(\omega)=\partial_\omega||\mathbf{\Phi}_{\lambda,\delta'}||^2$.
\begin{itemize}
\item[$(i)$] If $1<p\leq 5$, then $J(\omega)>0$.
\item[$(ii)$] If $p>5$, then there exists $\omega^*$ such that
$J(\omega^*)=0$, and $J(\omega)>0$ for
$\omega\in (\frac{N^2}{\lambda^2},\omega^* )$, while
$J(\omega)<0$ for $\omega\in (\omega^*,\infty)$.
\end{itemize} 
\end{proposition}

\begin{proof}[\bf Proof.]
Recall that
$\mathbf{\Phi}_{\lambda,\delta'}=(\varphi_{\lambda,\delta'})_{j=1}^N$,
where $\varphi_{\lambda,\delta'}$ is defined by
\eqref{varphi_lam}. We have via
change of variables
\begin{align*}
& \int
_{0}^\infty(\varphi_{\lambda,\delta'}(x))^2dx=\left(\frac{p+1}{2}\right)^{\tfrac 2{p-1}}\frac{2\omega^{\tfrac
2{p-1}-\tfrac 1{2}}}{p-1} \int
_{\tfrac{N}{|\lambda|\sqrt{\omega}}}^1(1-t^2)^{\tfrac
2{p-1}-1}dt.
\end{align*}
From the last  equality, we get 
\begin{equation}\label{J} 
 J(\omega)=C\omega^{\tfrac {7-3p}{2(p-1)}}J_1(\omega),\qquad C=\tfrac{N}{p-1}\left(\frac{p+1}{2}\right)^{\tfrac
2{p-1}}, 
 \end{equation}
where
$$
J_1(\omega)= \tfrac{5-p}{p-1} \int _{\tfrac
N{|\lambda|\sqrt{\omega}}}^1(1-t^2)^{\tfrac
{3-p}{p-1}}dt+\tfrac{N}{|\lambda|\sqrt{\omega}}(1-\tfrac{N^2}{\lambda^2\omega})^{\tfrac{3-p}{p-1}}.
$$ 
Thus,
\begin{equation}\label{J'}
J'_1(\omega)=\tfrac{N}{|\lambda|\omega^{3/2}}\tfrac{3-p}{p-1}
 \Big [ \Big (1-\tfrac{N^2}{\lambda^2\omega}
 \Big )^{\tfrac{3-p}{p-1}}+\tfrac{N^2}{\lambda^2\omega}
 \Big (1-\tfrac{N^2}{\lambda^2\omega} \Big )^{-\tfrac{2(p-2)}{p-1}} \Big ].
\end{equation}
It is immediate that $J(\omega)>0$ for 
$1< p\leq 5$. Consider the  case $p>5$.   
It is easily seen 
\begin{equation*}
\lim_{\omega\to +\infty}J_1(\omega)
=\frac{5-p}{p-1}\int_0^{1}(1-t^2)^{\tfrac {3-p}{p-1}}dt<0,
\quad \lim_{\omega\to \frac{N^2}{\lambda^2}}J_1(\omega)=\infty.
 \end{equation*}
 Moreover, from \eqref{J'} 
 it follows that $J'_1(\omega)<0$ for 
 $\omega> \tfrac{N^2}{\lambda^2}$, and 
 consequently $J_1(\omega)$ is strictly decreasing. 
Therefore, there exists a unique
$\omega^*>\tfrac{N^2}{\lambda^2}$
 such that $J_1(\omega^*)=J(\omega^*)=0$, 
 consequently  $J(\omega)>0$ for $
 \omega\in(\tfrac{N^2}{\lambda^2},\omega^*)$,
 and $J(\omega)<0$ for $\omega\in(\omega^*, \infty)$. 
 \end{proof}
\noindent{\bf Proof of Theorem \ref{Main}.}
$(i)$
 $1)$ 
Combining Theorem \ref{well1}, Theorem \ref{stabil_graph}
(adapted to the case of the NLS-$\delta'$ equation), Proposition
\ref{spec_graph'} (items $(i)$, $(ii)$ and $(v)$), and
Proposition \ref{slope_graph'}-$(i)$, we get stability of
$e^{i\omega t}\mathbf{\Phi}_{\lambda,\delta'}$ in
$H^1(\mathcal{G})$.
 
 $2)$
Combining Theorem \ref{stabil_graph}, Proposition
\ref{spec_graph'} (items $(i)$, $(iv)$ and $(v)$), and
Proposition \ref{slope_graph'}-$(i)$, we get orbital instability
of $e^{i\omega t}\mathbf{\Phi}_{\lambda,\delta'}$ in
$H^1_{\frac{N}{2}}(\mathcal{G})$ (compare with Remark
\ref{nonstab}-$(ii)$). We note that well-posedness of the Cauchy
problem associated with equation \eqref{NLS_graph'} in
$H^1_{\frac{N}{2}}(\mathcal{G})$ follows from the uniqueness of
the solution to the Cauchy problem in $H^1(\mathcal{G})$ and the
fact that the group $e^{-it\bb H_\lambda^{\delta'}}$ preserves
the space $H^1_{\frac{N}{2}}(\mathcal{G})$. Finally, instability
in the smaller space $H^1_{\frac{N}{2}}(\mathcal{G})$ induces
instability in all $H^1(\mathcal{G})$.

$(ii)$
Relative position of $\omega^*$ and
$\omega=\tfrac{N^2}{\lambda^2}\tfrac{p+1}{p-1}$ is not clear (see
Remark \ref{rel_posit}), which complicates the analysis in the
framework of Theorem \ref{stabil_graph}. But we can overcome this
difficulty restricting the operator $\bb L_{1,\lambda}$ onto the
space $L^2_{\eq}(\mathcal{G})$ defined by
$$L^2_{\eq}(\mathcal{G})=\{\mathbf{V}=(v_j)_{j=1}^N\in
L^2(\mathcal{G}): v_1(x)=...=v_{N}(x),\, x>0\}.$$
Moreover, we introduce
$H^1_{\eq}(\mathcal{G})=H^1(\mathcal{G})\cap
L^2_{\eq}(\mathcal{G})$. We note that $H^1_{\eq}(\mathcal{G})$ is
also preserved by the group $e^{-it\bb H_\lambda^{\delta'}}$.
 
Recall that $\bb L_{1,\lambda}$ is the self-adjoint extension of
the minimal symmetric operator $\bb L_{\min}$ defined by
\eqref{L_min}. It is easily seen that the operator $\bb
L_{\min}|_{L^2_{\eq}(\mathcal{G})}$ satisfies  $\mathcal{N}_\pm(\bb
L_{\min}|_{L^2_{\eq}(\mathcal{G})})=\Span\{(e^{i\sqrt{\pm
i}x})_{j=1}^N\}$. The last equality, by Proposition
\ref{semibounded}, implies $n(\bb
L_{1,\lambda}|_{L^2_{\eq}(\mathcal{G})})=1$ since $(\bb
L_{1,\lambda} \mathbf{\Phi}_{\lambda,\delta'},
\mathbf{\Phi}_{\lambda,\delta'}) <0$ and
$\mathbf{\Phi}_{\lambda,\delta'}\in L^2_{\eq}(\mathcal{G})$.
 
Without loss of generality we can assume that $\omega^*\neq
\tfrac{N^2}{\lambda^2}\tfrac{p+1}{p-1}.$
All our forthcoming conclusions about orbital stability are based
on Theorem \ref{stabil_graph} for the spaces $H^1(\mathcal{G})$
and $H^1_{\eq}(\mathcal{G})$, Remark \ref{nonstab}, Theorem
\ref{well1}, Proposition \ref{spec_graph'}, and Proposition
\ref{slope_graph'}. Consider 2 cases.

\textbf{1.} Suppose that
$\omega^*<\tfrac{N^2}{\lambda^2}\tfrac{p+1}{p-1}$.

  Let
$\omega<\omega^*<\tfrac{N^2}{\lambda^2}\tfrac{p+1}{p-1}$, then
$n(\bb L_{1,\lambda})=1$ in $L^2(\mathcal{G})$ and we have
$\partial_\omega||\mathbf{\Phi}_{\lambda,\delta'}||^2>0$.
Therefore, $e^{i\omega t}\mathbf{\Phi}_{\lambda,\delta'}$ is
orbitally stable in $H^1(\mathcal{G})$, and hence in
$H^1_{\eq}(\mathcal{G})$.
  
  If
$\omega^*<\omega<\tfrac{N^2}{\lambda^2}\tfrac{p+1}{p-1}$, then
$n(\bb L_{1,\lambda})=1$ in $L^2(\mathcal{G})$ and
$\partial_\omega||\mathbf{\Phi}_{\lambda,\delta'}||^2<0$, which
induces orbital instability of $e^{i\omega
t}\mathbf{\Phi}_{\lambda,\delta'}$ in $H^1(\mathcal{G})$.
    
  Let
$\omega>\tfrac{N^2}{\lambda^2}\tfrac{p+1}{p-1}>\omega^*$. Then
$n(\bb L_{1,\lambda}|_{L^2_{\eq}(\mathcal{G})})=1$ and also
$\partial_\omega||\mathbf{\Phi}_{\lambda,\delta'}||^2<0$, which
induces orbital instability of $e^{i\omega
t}\mathbf{\Phi}_{\lambda,\delta'}$ in $H^1_{\eq}(\mathcal{G})$
and consequently in $H^1(\mathcal{G})$.
  
\textbf{ 2.}
Suppose that  $\omega^*>\tfrac{N^2}{\lambda^2}\tfrac{p+1}{p-1}$.

 If
$\omega<\tfrac{N^2}{\lambda^2}\tfrac{p+1}{p-1}<\omega^*$, then
$n(\bb L_{1,\lambda})=1$ in $L^2(\mathcal{G})$ and
$\partial_\omega||\mathbf{\Phi}_{\lambda,\delta'}||^2>0$,
consequently $e^{i\omega t}\mathbf{\Phi}_{\lambda,\delta'}$ is
orbitally stable in $H^1(\mathcal{G})$, and therefore in
$H^1_{\eq}(\mathcal{G})$.
 
  If
$\tfrac{N^2}{\lambda^2}\tfrac{p+1}{p-1}<\omega<\omega^*$, then
$n(\bb L_{1,\lambda}|_{L^2_{\eq}(\mathcal{G})})=1$ and
$\partial_\omega||\mathbf{\Phi}_{\lambda,\delta'}||^2>0$, which
induces stability of $e^{i\omega
t}\mathbf{\Phi}_{\lambda,\delta'}$ in $H^1_{\eq}(\mathcal{G})$ .

  Let
$\omega>\omega^*>\tfrac{N^2}{\lambda^2}\tfrac{p+1}{p-1}$, then
$n(\bb L_{1,\lambda}|_{L^2_{\eq}(\mathcal{G})})=1$ and
$\partial_\omega||\mathbf{\Phi}_{\lambda,\delta'}||^2<0$, which
induces orbital instability of $e^{i\omega
t}\mathbf{\Phi}_{\lambda,\delta'}$ in $H^1_{\eq}(\mathcal{G})$
and consequently in $H^1(\mathcal{G})$.
  
Summarizing all the cases, we get for $\omega>\omega^*$ nonlinear
instability of $e^{i\omega t}\mathbf{\Phi}_{\lambda,\delta'}$ in
$H^1(\mathcal{G})$, and for $\omega<\omega^*$ stability of
$e^{i\omega t}\mathbf{\Phi}_{\lambda,\delta'}$ at least in
$H^1_{\eq}(\mathcal{G})$. This finishes the proof.
\qed

\begin{remark}
\rm 
$(i)$
It is worth mentioning that the orbital instability result
follows easily for $2<p<5$ from the spectral instability using
the fact that the mapping data-solution for \eqref{NLS_graph'} is
of class $C^2$ (see Theorem \ref{well1} and Remark
\ref{nonstab}-$(iii)$).

$(ii)$ Observe that for $p>5$ the orbital instability results are
obtained via classical approach by \cite{GrilSha87} without using spectral instability. 
Otherwise, the orbital instability can be deduced from the
spectral one since for $p>5$ the mapping data-solution for
\eqref{NLS_graph'} is of class $C^2$.
 
\end{remark}

\begin{remark}\label{rel_posit}
\rm
Note that the integral appearing in \eqref{J} (via change of
variables) is related to the incomplete Beta function
$$
B\Big(y; \frac12, b\Big)=\int_0^{y}x^{-\tfrac1 {2}}(1-x)^{b-1}dx,$$
with $b= \frac{2}{p-1}$. Using basic numerical simulations, one can
show that for $p=6,7,...$, relation
$\omega^*>\tfrac{N^2}{\lambda^2}\tfrac{p+1}{p-1}$ holds. By the
continuity of the function $J$ as a function of $p$, we get the
relation $\omega^*>\tfrac{N^2}{\lambda^2}\tfrac{p+1}{p-1}$ in the
neighborhood of every integer $p>5$.

We conjecture that
$\omega^*>\tfrac{N^2}{\lambda^2}\tfrac{p+1}{p-1}$ holds for any
$p>5$.
This conjecture by Theorem \ref{stabil_graph} implies the
following stability properties of $e^{i\omega
t}\mathbf{\Phi}_{\lambda,\delta'}$ in the case $p>5$:
\begin{itemize}
\item[$(i)$]
if
$\omega\in (\tfrac{N^2}{\lambda^2},\tfrac{N^2}{\lambda^2}\tfrac{p+1}{p-1} )$,
then $e^{i\omega t}\mathbf{\Phi}_{\lambda,\delta'}$ is stable in
$H^1(\mathcal{G})$;

\item[$(ii)$] if
$\omega\in (\tfrac{N^2}{\lambda^2}\tfrac{p+1}{p-1},\omega^* )$
and $N$ is even, then $e^{i\omega
t}\mathbf{\Phi}_{\lambda,\delta'}$ is unstable in
$H^1(\mathcal{G})$.
\end{itemize}
 \end{remark} 
 
\section{Stability theory of standing wave solutions for the
NLS-log-$\delta$ and the NLS-log-$\delta'$ equation on a star
graph}\label{log}
\subsection{ The NLS-log-$\delta$   equation on a star graph}
In this Subsection we prove spectral instability of the $N$-bump
stationary state solution
$\mathbf{\Psi}_{\alpha,\delta}=(\psi_{\alpha,\delta})_{j=1}^N$ of
Gaussian type, where
$
\psi_{\alpha,\delta}(x)=
e^{\tfrac{\omega+1}{2}}e^{-\tfrac{(x-\tfrac{\alpha}{N})^2}{2}},$
$\alpha>0,\, \omega\in \mathbb{R}.
$ 
We also extend the stability result in \cite{Ard16} for any
$\alpha<0$ (see Theorem \ref{stabil_log_delta}).

Since well-posedness is a crucial assumption for stability
theory, it is worth proving that  equation
\eqref{NLS_graph_log1} is well-posed in the space $W^1_{\mathcal
E}(\mathcal{G})$. In \cite{Ard16} the following well-posedness
result in $W_{\mathcal E}(\mathcal{G})$ was proved.

\begin{proposition}\label{well_log_graph} For any $\bb U_0\in
W_{\mathcal{E}}(\mathcal{G})$ there is a unique solution\newline
$\bb U \in C(\mathbb{R}, W_{\mathcal{E}}(\mathcal{G}))\cap
C^1(\mathbb{R}, W'_{\mathcal E}(\mathcal{G}))$ of
\eqref{NLS_graph_log1} such that $\bb U(0) = \bb U_0$ and\,\,
$\sup\limits_{t\in\mathbb{R}}||\bb
U(t)||_{W_{\mathcal{E}}(\mathcal{G})} <\infty$. Furthermore, the
conservation of energy and charge holds, that is,
$$E_{\alpha,\log}(\bb U(t)) = E_{\alpha,\log}(\bb U_0),\,\,\,
\text{and}\,\,\, Q(\bb U(t))=||\bb U(t)||^2=||\bb U_0||^2,$$
where the energy $E_{\alpha,\log}$ is defined for
$\mathbf{V}=(v_j)_{j=1}^N\in W_{\mathcal E}(\mathcal{G})$ by
 \begin{equation*} 
E_{\alpha,\log}(\bb V)=\tfrac 1{2}||\bb V'||^2-\tfrac
1{2}\sum\limits_{j=1}^N \int _0^\infty|v_j|^2\logg|v_j|^2dx
+\tfrac{\alpha}{2}|v_1(0)|^2.
\end{equation*}
\end{proposition}

Using the above result, we obtain well-posedness in
$W^1_{\mathcal{E}}(\mathcal{G})$.
\begin{theorem}\label{well_pos_log}
If $\bb U_0\in W^1_{\mathcal{E}}(\mathcal{G})$, there is a unique
solution $\bb U(t)$ of \eqref{NLS_graph_log1} such that $\bb
U(t)\in C(\mathbb R, W^1_{\mathcal{E}}(\mathcal{G}))$ and $\bb
U(0)=\bb U_0$.
\end{theorem}

\begin{proof}[\bf Proof.]
The proof can be found in \cite{AngGol18}. Basically it follows
from Proposition \ref{well_log_graph} and two additional facts.
The first one is that $W^1_{\mathcal{E}}(\mathcal{G})\subset
W_{\mathcal{E}}(\mathcal{G})$ (see \cite[Lemma 3.1]{A3}). And the
second one is the continuity of the mapping $t\mapsto ||x\bb
U(t)||^2$ on $\mathbb{R}$.
\end{proof}
 
The strategy of the proof of Theorem \ref{stabil_log_delta} is
analogous to the one in the previous case of the NLS equation with
power nonlinearity. In particular, we will use the adapted
(weaker) version of the stability/instability Theorem
\ref{stabil_graph} (to the specific Gaussian profile $\bb
\Psi_{\alpha,\delta}$ and the space
$W^1_{\mathcal{E}}(\mathcal{G})$).

Consider the following two harmonic oscillator self-adjoint
matrix operators with domain $\dom(\bb T_{1,\alpha})=\dom(\bb
T_{2,\alpha})=\bb D_{\alpha,\delta}^{\log}$ defined by
 \begin{align*} 
  \bb
T_{1,\alpha}&= \Big (\Big(-\frac{d^2}{dx^2}+(x-\tfrac{\alpha}{N})^2-3\Big)\delta_{k,j} \Big ),
\\
\bb T_{2,\alpha}&= \Big (\Big(-\frac{d^2}{dx^2}+
(x-\tfrac{\alpha}{N})^2-1\Big)\delta_{k,j} \Big ),\\
\bb D_{\alpha,\delta}^{\log}:&= \Big \{\mathbf{V}\in W^2(\mathcal{G}):
v_1(0)=...=v_N(0),\quad \sum\limits_{j=1}^N v_j'(0)=\alpha v_1(0)
 \Big \},
\end{align*}
where $\delta_{k,j}$ is the Kronecker symbol. These operators are
associated with \newline $\bb H_{\alpha,\log}:= (\bb
S_{\alpha,\log})''(\bb \Psi_{\alpha,\delta})$ (where $\bb
S_{\alpha,\log}$ is defined by \eqref{act_log}) in a standard
way, i.e. 
\begin{equation*}\label{L1+L2_log}
\bb H_{\alpha,\log}=\left(\begin{array}{cc} \bb T_{1,\alpha}&\bb
0\\
 \bb 0&\bb T_{2,\alpha}\end{array}\right).
\end{equation*}
Noting that $\partial_\omega||\bb \Psi_{\alpha,\delta}||^2>0$,
$E_{\alpha,\log}\in C(W^1_\mathcal{E}(\mathcal{G}),\mathbb{R})$
(see \cite[Proposition 2.3]{AngGol18}), and combining
\cite[Theorem 3.5]{GrilSha87} with \cite[Theorem 5.1]{GrilSha90},
we can formulate the stability/instability theorem for the
NLS-log-$\delta$ equation.
\begin{theorem}\label{stabil_graph_log}
Let $\alpha\neq 0$, and $n(\bb H_{\alpha,\log})$ be the number of
negative eigenvalues of $\bb H_{\alpha,\log}$. Suppose also that
 
$1)$\,\,$\ker(\bb T_{2,\alpha})=\Span\{\bb
\Psi_{\alpha,\delta}\}$,
 
 $2)$\,\,$\ker(\bb T_{1,\alpha})=\{\bb 0\}$, 
 
$3)$\,\, the negative spectrum of $\bb T_{1,\alpha}$ and $\bb
T_{2,\alpha}$ consists of a finite number of negative eigenvalues
(counting multiplicities),
  
$4)$\,\, the rest of the spectrum of $\bb T_{2,\alpha}$ and $\bb
T_{1,\alpha}$ is positive and bounded away from zero.
  \
  Then the following assertions hold.
\begin{itemize}
\item[$(i)$] If $n(\bb H_{\alpha,\log})=1$, then the standing
wave $e^{i\omega t}\bb \Psi_{\alpha,\delta}$ is orbitally stable
in $W^1_{\mathcal{E}}(\mathcal{G})$.
\item[$(ii)$] If $n(\bb H_{\alpha,\log})=2$ in
$L^2_k(\mathcal{G})$, then the standing wave $e^{i\omega t}\bb
\Psi_{\alpha,\delta}$ is spectrally unstable.
\end{itemize}
\end{theorem}

\begin{remark}
\rm
$(i)$
By saying $e^{i\omega t}\bb \Psi_{\alpha,\delta}$ ``{\it is
spectrally unstable}'' we mean that the spectrum of the linear
part $\bb A_{\alpha,\log}=\left(\begin{array}{cc} \bb 0& \bb
T_{2,\alpha} \\ -\bb T_{1,\alpha} & \bb 0 \end{array}\right)$ of
the linearization of the NLS-log-$\delta$ equation around $\bb
\Psi_{\alpha,\delta}$ contains an eigenvalue with positive real
part.

$(ii)$
In item $(ii)$ we affirm only spectral instability since we can't
apply neither \cite[Corollary 3 and 4]{oh} (since we don't know
if $E_{\alpha,\log}\in
C^2(W^1_\mathcal{E}(\mathcal{G}),\mathbb{R})$), nor \cite[Theorem
2 Remark, Section 2]{HenPer82} (since we don't know if the
mapping data-solution associated to the NLS-log-$\delta$ equation is
of class $C^2$ around $\bb \Psi_{\alpha,\delta}$) to prove
orbital instability (see Remark \ref{nonstab} above).

\end{remark}

Below we study the spectral properties of $\bb T_{1,\alpha}$ and
$\bb T_{2,\alpha}$.
To investigate the spectrum of the operator $\bb T_{1,\alpha}$ we
will use the perturbation theory analogously to the previous case
of the NLS-$\delta$ equation with power nonlinearity.
In particular, define the following self-adjoint Schr\"odinger
operator on $L^2(\mathcal{G})$ with Kirchhoff condition at
$\nu=0$
\begin{align}
\label{T^0_1} 
&\bb
T_{1,0}= \Big (\Big(-\frac{d^2}{dx^2}+x^2-3\Big)\delta_{i,j} \Big ),
\\
&\dom(\bb T_{1,0})= \Big \{\mathbf{V}\in W^2(\mathcal{G}):
v_1(0)=...=v_N(0),\,\,\sum\limits_{j=1}^N v_j'(0)=0 \Big \}.
\notag
 \end{align}
As above $\bb T_{1,\alpha}$ ``tends" to $\bb T_{1,0}$ for
$\alpha\to 0$. In the next Theorem we describe the spectral
properties of $\bb T_{1,0}$.
 \begin{theorem}\label{spect_T^0_1}
Let $\bb T_{1,0}$ be defined by \eqref{T^0_1} and
$k\in\left\{1,...,N-1\right\}$. Then the following assertions hold
\begin{itemize}
  \item[$(i)$] 
$\ker(\bb
T_{1,0})=\Span\{\hat{\bb\Psi}_{0,1},...,\hat{\bb\Psi}_{0,
N-1}\}$,
  where \begin{equation*} 
\hat{\mathbf{\Psi}}_{0,j}=(0,...,0,\underset{\bf
j}{\psi'_{0}},\underset{\bf j+1}{-\psi'_{0}},0,...,0),
\quad \psi_{0}(x)=e^{-\tfrac{x^2}{2}}.
 \end{equation*}
\item[$(ii)$] In the space $L^2_k(\mathcal{G})$ we have $\ker(\bb
T_{1,0})=\Span\{\mathbf{\widetilde{\Psi}}_{0,k}\}$, where
  \begin{equation}\label{Psi_0_log}
\mathbf{\widetilde{\Psi}}_{0,k}= \Big (\underset{\bf
1}{\tfrac{N-k}{k}\psi'_{0}},..., \underset{\bf
k}{\tfrac{N-k}{k}\psi'_{0}},\underset{\bf
k+1}{-\psi'_{0}},...,\underset{\bf N}{-\psi'_{0}} \Big ),
 \end{equation}
i.e. $\ker(\bb
T_{1,0}|_{L^2_k(\mathcal{G})})=\Span\{\mathbf{\widetilde{\Psi}}_{0,k}\}$.
\item[$(iii)$] The operator $\bb T_{1,0}$ has one simple negative
eigenvalue, i.e. $n(\bb T_{1,0})=1$. Moreover, the operator $\bb
T_{1,0}$ has one simple negative eigenvalue in
$L^2_k(\mathcal{G})$, i.e. $n(\bb
T_{1,0}|_{L^2_k(\mathcal{G})})=1$.
   \item[$(iv)$] The spectrum of $\bb T_{1,0}$ is discrete. 
  \end{itemize}
\end{theorem}

\begin{proof}[\bf Proof.]
The proof of items $(i)$-$(ii)$ repeats the one of Theorem
\ref{spect_L^0_1} $(i)$-$(ii)$.

$(iii)$
We will follow the ideas of the proof of item $(iii)$ of Theorem
\ref{spect_L^0_1} and Lemma 4.11 in \cite{A3}.
Denote
$\opt_0= ( (-\frac{d^2}{dx^2}+x^2-3 )\delta_{k,j} )$.
First, one needs to show that the operator $\bb T_0$ acting as $\bb
T_0=\opt_0$ on
\begin{equation*}
\dom(\bb T_0)= \Big \{\mathbf{V}\in W^2(\mathcal{G}):
v_1(0)=...=v_N(0)=0,\quad \sum\limits_{j=1}^N v_j'(0)=0  \Big \}.
\end{equation*}
is non-negative. The proof follows from the  identity   
\begin{equation*}
-v_j''+(x^2-3)v_j=
\frac{-1}{\psi'_{0}}\frac{d}{dx} \Big [(\psi'_{0})^2\frac{d}{dx}
 \Big (\frac{v_j}{\psi'_{0}} \Big ) \Big ],\quad
x> 0,
\end{equation*}
for any  $\mathbf{V}=(v_j)_{j=1}^N\in W^2(\mathcal{G})$.

Next we need to prove that $n_\pm(\bb T_0)=1$. We use the
ideas of the proof of \cite[Lemma 4.11]{A3}.
First, we establish the scale of Hilbert spaces associated with
the self-adjoint non-negative operator (see \cite[Section I,\S
1.2.2]{ak}) $\bb T=\opt_0+3 I$ defined on
$$
\dom(\bb T)= \Big \{\mathbf{V}\in W^2(\mathcal{G}):
v_1(0)=...=v_N(0),\quad \sum\limits_{j=1}^N v_j'(0)=0  \Big \}.
 $$
Define for $s\geq 0$ the space 
$$
\mathfrak H_s(\bb T)=\left\{\bb V\in L^2(\mathcal G): \|\bb
V\|_{s,2}=\Big\|(\bb T+ I)^{s/2}\bb V\Big\|<\infty\right\}.
$$
The space $\mathfrak H_s(\bb T)$ with norm $\|\cdot\|_{s,2}$ is
complete. The dual space of $\mathfrak H_s(\bb T)$ is denoted by
$\mathfrak H_{-s}(\bb T)=\mathfrak H_s(\bb T)'$. The norm in the
space $\mathfrak H_{-s}(\bb T)$ is defined by the formula
$ 
\|\bb V\|_{-s,2}=  \|(\bb T+ I)^{-s/2} \bb V \|.
$ 
The spaces $\mathfrak H_s(\bb T)$ form the following chain 
$$ 
...\subset \mathfrak H_2(\bb T)\subset \mathfrak H_1(\bb
T)\subset L^2(\mathcal G)=\mathfrak H_0(\bb T)\subset\mathfrak
H_{-1}(\bb T)\subset \mathfrak H_{-2}(\bb T)\subset...
$$

 The norm in the space 
$\mathfrak H_1(\bb T)$ can be calculated as follows
\begin{align*}
&\|\bb V\|^2_{1,2}=((\bb T+ I)^{1/2}\bb V, (\bb T+I)^{1/2}\bb
V)\\&=\sum\limits_{j=1}^N \int _0^\infty\left( |v'_j(x)|^2
+|v_j(x)|^2 +x^2|v_j(x)|^2\right)dx.
\end{align*}
Therefore, we have the embedding $\mathfrak H_1(\bb T)
\hookrightarrow H^1(\mathcal G)$ and, by the Sobolev embedding,
 $\mathfrak H_1(\bb T) \hookrightarrow L^\infty (\mathcal
G)$. From the former remark we obtain that the functional
$\delta_1: \mathfrak H_1(\bb T)\to \mathbb C$ acting as
$\delta_1(\bb V)=v_1(0)$ belongs to $\mathfrak H_1(\bb
T)'=\mathfrak H_{-1}(\bb T)$, and consequently $\delta_1\in
\mathfrak H_{-2}(\bb T)$. Therefore, using \cite[Lemma
1.2.3]{ak}, it follows that the restriction $\hat{\bb T}_0$ of
the operator $\bb T$ onto the domain
$$ 
\dom(\hat{\bb T}_0)=\{\bb V\in \dom(\bb T): \delta_1(\bb
V)=v_1(0)=0\}=\dom(\bb T_0)
$$ 
is a densely defined symmetric operator with equal deficiency
indices $n_{\pm}(\hat{\bb T}_0)=1$. By \cite[Chapter IV, Theorem
6]{Nai67}, the operators $\hat{\bb T}_0$ and $\bb T_0$ have equal
deficiency indices. Therefore, $n(\bb T_{1,0})\leq 1$. Since $\bb
T_{1,0}\bb \Psi_0=-2\bb \Psi_0,$ where $\bb
\Psi_0=(\psi_0)_{j=1}^N$, we get $n(\bb T_{1,0})=1$. As
$\bb \Psi_0\in L^2_k(\mathcal{G})$ for any $k$, we get $n(\bb
T_{1,0}|_{L^2_k(\mathcal{G})})=1$.

$(iv)$
With slight modifications we can repeat the proof of
\cite[Theorem 3.1, Chapter II]{BerShu91} to show that the
spectrum of $\bb T_{1,0}$ is discrete by $\lim\limits_{x\to
+\infty}(x^2-3)=+\infty,$
i.e. $\sigma(\bb T_{1,0})=\sigma_p(\bb
T_{1,0})=\{\mu_{0,j}\}_{j\in\mathbb{N}}$. In particular, we have
the following distribution of the eigenvalues
$ 
\mu_{0,1}<\mu_{0,2}<\cdot\cdot\cdot<\mu_{0,j}<\cdot\cdot\cdot,
$ 
with $\mu_{0,j}\to +\infty$ as $j\to +\infty$.
 \end{proof} 
 
\begin{proposition}\label{spect_log_delta}
Let $k\in\left\{1,...,N-1\right\}$, $\alpha\neq 0$, and $\bb
\Psi_{\alpha,\delta}$ be defined by \eqref{Phi_vect_log}. Then
\begin{itemize}
\item[$(i)$] $\ker(\bb
T_{2,\alpha})=\Span\{\mathbf{\Psi}_{\alpha,\delta}\}$ and $\bb
T_{2,\alpha}\geq 0$,
\item[$(ii)$] $\ker(\bb T_{1,\alpha})=\{\mathbf{0}\}$,
\item[$(iii)$] for $\alpha>0$, $n(\bb T_{1,\alpha})=2$ in
$L^2_k(\mathcal{G})$, i.e. $n(\bb
T_{1,\alpha}|_{L^2_k(\mathcal{G})})=2$,
\item[$(iv)$] for $\alpha<0$, $n(\bb T_{1,\alpha})=1$ in
$L^2(\mathcal{G})$,
\item[$(v)$] the spectrum of the operators $\bb T_{1,\alpha}$ and
$\bb T_{2,\alpha}$ in $L^2(\mathcal{G})$ is discrete.
\end{itemize}
\end{proposition}

\begin{proof}[\bf Proof.]
 
$(i)$ The proof repeats the one of \cite[Proposition
6.1]{AdaNoj14}. We only need to note that any
$\mathbf{V}=(v_j)_{j=1}^N\in W^2(\mathcal{G})$ satisfies the
following identity
\begin{equation*}
-v_j''+((x-\tfrac{\alpha}{N})^2-1)v_j=
\frac{-1}{\psi_{\alpha,\delta}}\frac{d}{dx}
 \Big [\psi_{\alpha,\delta}^2\frac{d}{dx}
  \Big (\frac{v_j}{\psi_{\alpha,\delta}} \Big ) \Big ],\quad
x>0.
\end{equation*}

$(ii)$
The proof is standard. It is sufficient to note that any vector
from the kernel of $\bb T_{1,\alpha}$ has the form $\bb
V=(v_j)_{j=1}^N,$ where $v_j=c_j\psi'_{\alpha,\delta}\,\,
c_j\in\mathbb{R}$.

$(iii)$
The proof of this item is analogous to the one of the item
$(iii)$ of Proposition \ref{n(L_1)}. It suffices to note that for
the operator $\bb T_{1,\alpha}$ the coefficient $\mu_0$ in
decomposition \eqref{decomp1} is negative. Indeed, (see the proof
of Proposition 4.17 in \cite{A3})
$$
\mu_0=-\frac{2(N-k)}{k||\widetilde{\bb \Psi}_{0,k}||^2} \int _0^\infty x(\psi'_0)^2dx+O(\alpha),
 $$
where $\widetilde{\bb \Psi}_{0,k}$ is defined by
\eqref{Psi_0_log}.

$(iv)$
To show the equality in the whole space $L^2(\mathcal{G})$, we
need to repeat the arguments of the proof of Theorem
\ref{spect_T^0_1}-$(iii)$ (i.e. $\bb T_{1,0}$ has to be replaced
by $\bb T_{1,\alpha}$, and $\bb \Psi_0$ by $\bb
\Psi_{\alpha,\delta}$).

$(v)$
The proof follows from \cite[Chapter II, Theorem 3.1]{BerShu91}.\end{proof}

\noindent\textbf{Proof of Theorem \ref{stabil_log_delta}.}
Combining Theorem \ref{well_pos_log}, Theorem
\ref{stabil_graph_log}, Proposition \ref{spect_log_delta}, we get
orbital stability of $e^{i\omega t}\mathbf{\Psi}_{\alpha,\delta}$
in $W^1_{\mathcal{E}}(\mathcal{G})$ for $\alpha<0$ and spectral
instability of $e^{i\omega t}\mathbf{\Psi}_{\alpha,\delta}$ for
$\alpha>0$.
\qed

\subsection{The NLS-log-$\delta'$ equation on a star graph}

In this Subsection we study the stability properties for the
N-tail profile
$\bb{\Psi}_{\lambda,\delta'}=(\psi_{\lambda,\delta'})_{j=1}^N$,
where
$
\psi_{\lambda,\delta'}=
e^{\tfrac{\omega+1}{2}}e^{-\tfrac{(x-\tfrac{N}{\lambda})^2}{2}},
$
$\lambda<0,\,\, \omega\in \mathbb R.
$
Similarly to \cite[Proposition 1.1]{Ard16}, we get the
well-posedness result in $W(\mathcal{G})$.
 
\begin{proposition}\label{well_log_graph'} For any $\bb U_0\in
W(\mathcal{G})$ there is a unique solution $\bb U \in
C(\mathbb{R}, W(\mathcal{G}))\cap C^1(\mathbb{R},
W'(\mathcal{G}))$ of \eqref{NLS_graph_log2} such that $\bb U(0) =
\bb U_0$ and $\sup\limits_{t\in\mathbb{R}}||\bb
U(t)||_{W(\mathcal{G})} <\infty$. Furthermore, the conservation
of energy and charge holds, that is,
$$E_{\lambda,\log}(\bb U(t)) = E_{\lambda,\log}(\bb U_0),\,\,\,
\text{and}\,\,\, Q(\bb U(t))=||\bb U(t)||^2=||\bb U_0||^2,$$
where the energy $E_{\lambda,\log}$ is defined for
$\mathbf{V}=(v_j)_{j=1}^N\in W(\mathcal{G})$ by
 \begin{equation*} 
E_{\lambda,\log}(\bb V)=\tfrac 1{2}||\bb V'||^2-\tfrac
1{2}\sum\limits_{j=1}^N \int _0^\infty|v_j|^2\logg|v_j|^2dx
+\tfrac{1}{2\lambda} \Big |\sum\limits_{j=1}^Nv_j(0) \Big |^2.
\end{equation*}
\end{proposition}

\begin{proof}[\bf Proof.]
The proof repeats the one of \cite[Proposition 1.1]{Ard16}. One
just needs to replace $\mathfrak{F}_\gamma[u]=\sum\limits_{j=1}^N
\int _{\mathbb{R}_+}|u'_j|^2dx -\gamma |u_1(0)|^2$ by
$\sum\limits_{j=1}^N \int
_{\mathbb{R}_+}|u'_j|^2dx+\tfrac{1}
{\lambda} \Big |\sum\limits_{j=1}^Nu_j(0) \Big |^2$.
We also refer the reader to \cite[Section 9.2]{Caz03}.
\end{proof}
Using the above result, one may show the well-posedness in
$W^1(\mathcal{G})$.
\begin{theorem}\label{well_pos_log'}
If $\bb U_0\in W^1(\mathcal{G})$, there is a unique solution $\bb
U(t)$ of \eqref{NLS_graph_log2} such that $\bb U(t)\in C(\mathbb
R, W^1(\mathcal{G}))$ and $\bb U(0)=\bb U_0$.
\end{theorem}

\begin{proof}[\bf Proof.]
One should repeat the proof of  
\cite[Theorem 2.2]{AngGol18} 
substituting $W^1_{\mathcal{E}}(\mathcal{G})$ by 
$W^1(\mathcal{G})$. 
\end{proof} 
 
Consider the action functional associated with 
equation \eqref{NLS_graph_log2} for 
$ \bb V\in W^1(\mathcal{G})$,
\begin{equation*}
\label{act_log'} 
\bb S_{\lambda,\log}(\bb V) \! = \! 
\tfrac 1{2}||\bb V'||^2+\tfrac{(\omega+1)}{2}||\bb V||^2 -\tfrac
1{2}
\sum\limits_{j=1}^N \int _0^\infty 
\!\!
|v_j|^2\logg|v_j|^2dx
+\tfrac{1}{2\lambda} \Big |\sum\limits_{j=1}^Nv_j(0) \Big |^2.
\end{equation*}
As above, our idea is to study the spectral properties of the
self-adjoint operators associated with $(\bb
S_{\lambda,\log})''(\bb \Phi_{\lambda,\delta'})$
 \begin{align*} 
\mathbf{T}_{1,\lambda}&= \Big (\Big(-\frac{d^2}{dx^2}+(x-\tfrac{N}{\lambda})^2-3\Big)\delta_{k,j} \Big ),
\\
\mathbf{T}_{2,\lambda}&= \Big (\Big(-\frac{d^2}{dx^2}+
(x-\tfrac{N}{\lambda})^2-1\Big)\delta_{k,j} \Big ),
\end{align*} 
acting on $\dom(\mathbf{T}_{1,\lambda})=\dom(\mathbf{T}_{2,\lambda})=\bb D_{\lambda,\delta'}^{\log}$, where
 \begin{equation*} 
\bb D_{\lambda,\delta'}^{\log}:=
  \Big \{\mathbf{V}\in W^2(\mathcal{G}): 
   v'_1(0)=...=v'_N(0),
    \quad \sum\limits_{j=1}^N  v_j(0)=\lambda v'_1(0)  \Big \}.
\end{equation*}
Using arguments from the proof of Proposition
\ref{spect_log_delta} and Proposition \ref{spec_graph'}, we can
show the following result.

\begin{proposition}\label{spec_T_lambda} Let $k\in\{1,...,
N-1\}$, $\lambda<0$, and $\bb{\Psi}_{\lambda,\delta'}$ be defined
by \eqref{prof_log'}. Then the following assertions hold.
\end{proposition}
\begin{itemize}
\item[$(i)$]
$\ker(\mathbf{T}_{2,\lambda})=\Span\{\mathbf{\Psi}_{\lambda,\delta'}\}$,
and $\mathbf{T}_{2,\lambda}\geq 0$.
\item[$(ii)$] If $-N<\lambda<0$, then
$\ker(\mathbf{T}_{1,\lambda})=\{\mathbf{0}\}$, and
$n(\mathbf{T}_{1,\lambda})=1$ in $L^2(\mathcal G)$.
  
\item[$(iii)$] If $\lambda = -N$, then
$n(\mathbf{T}_{1,\lambda})=1$, and the kernel of
$\mathbf{T}_{1,\lambda}$ is given by
$\ker(\mathbf{T}_{1,\lambda})=\Span\{\hat{\mathbf{\Psi}}_{\lambda,1},..,
\hat{\mathbf{\Psi}}_{\lambda,N-1}\}$, where
\begin{equation*}
\hat{\mathbf{\Psi}}_{\lambda,j}=
(0,..,0,\underset{\bf j}{\psi'_{-N,\delta'}},
\underset{\bf j+1}{-\psi'_{-N,\delta'}},0,..,0). 
 \end{equation*}
 In particular, in this case 
 $n(\mathbf{T}_{1,\lambda}|_{L^2_k(\mathcal{G})})=1$, 
 and 
 $\ker(\mathbf{T}_{1,\lambda}|_{L^2_k(\mathcal{G})})
 =\Span\{\widetilde{\bb \Psi}_{-N,k}\}$, 
 where
  $$
  \widetilde{\bb \Psi}_{-N,k}
  =
\ \Big (\underset{\bf 1}{\tfrac{N-k}{k}\psi'_{-N,\delta'}},...,   \underset{\bf k}{\tfrac{N-k}{k}\psi'_{-N,\delta'}},
   \underset{\bf k+1}{-\psi'_{-N,\delta'}},...,
   \underset{\bf N}{-\psi'_{-N,\delta'}} \Big ).
   $$
  
\item[$(iv)$] If $\lambda < -N$, then
$\ker(\mathbf{T}_{1,\lambda})=\{\mathbf{0}\}$, and
$n(\mathbf{T}_{1,\lambda}|_{L^2_k(\mathcal{G})})=2$.
\item[$(v)$] The spectrum of $\mathbf{T}_{1,\lambda}$ and
$\mathbf{T}_{2,\lambda}$ is discrete.
\end{itemize}

\begin{proof}[\bf Proof.]
 
 $(i)$
The proof is analogous to the one of item $(i)$ of Proposition
\ref{spec_graph'}.

$(ii)$
The proof repeats the one of item $(ii)$ of Proposition
\ref{spec_graph'}. We only need to note that
 the non-negative (for $-N<\lambda<0$) symmetric operator 
 \begin{align*} 
&\mathbf{T}'_0= \Big (\Big(-\frac{d^2}{dx^2}+(x-\tfrac{N}{\lambda})^2-3\Big)\delta_{k,j} \Big ),\\
&\dom(\mathbf{T}'_0)= \Big \{\mathbf{V}\in W^2(\mathcal{G}):
v'_1(0)=...=v'_N(0)=0, \quad\sum\limits_{j=1}^N v_j(0)=0
 \Big \}.
\end{align*}
has deficiency indices equal one. It can be shown repeating the
arguments of the proof of item $(iii)$ of Theorem
\ref{spect_T^0_1}.

$(iii)$
It suffices to repeat the arguments of the proof of item $(iii)$
of Proposition \ref{spec_graph'}.

$(iv)$ 
By the analyticity of the family $(\bb T_{1,\lambda})$ as 
a function of $\lambda<0$ and  the spectral properties of 
$\bb T_{1,\lambda}$, for  $\lambda=-N$,  
we obtain (via the Kato-Rellich Theorem):
 \begin{enumerate}
 \item[1)] There  exist $\delta>0$ small and 
 two analytic functions $\mu(\lambda) : 
 (-N-\delta, -N+\delta )\to \mathbb R$ and 
 $\bb F(\lambda): (-N-\delta, -N+\delta)\to 
 L^2_k(\mathcal{G})$ such that $\mu(-N)=0$ and 
 $\bb F(-N)= \widetilde{\bb \Psi}_{-N,k}$.
 
\item[2)] $\mu(\lambda)$ is a simple isolated  
eigenvalue of $\bb T_{1,\lambda}$, and 
$\bb F(\lambda)$ is an associated eigenvector for $\mu(\lambda)$.
\item[3)] Except  at most the first two 
eigenvalues, the spectrum of 
$\bb T_{1,\lambda}|_{L^2_k(\mathcal{G})}$ is positive.
 \end{enumerate}

Below we show  that $\mu(\lambda)<0$ for 
$\lambda < -N$, and $\mu(\lambda)>0$ for 
$\gamma> -N$. From Taylor's theorem we have the following
expansions
\begin{align}
\label{2decomp1} 
\mu(\lambda)&=\mu_{-N}(\lambda+N)+ O((\lambda+N)
^2),\,\,\text{and}\\
\bb F(\lambda)&=\widetilde{\bb \Psi}_{-N,k}+ (\lambda+N) \bb
G_{-N} + \bb O((\lambda+N)^2),
\notag
\end{align}
where  $\mu_{-N}=\mu'(-N)\in \mathbb R$  and 
$\bb G_{-N}=\partial_{\lambda}\bb F(\lambda)|_{\lambda=-N}\in
L^2_k(\mathcal{G})$.

Let us show that $\mu_{-N}>0$. To show the positivity of
$\mu_{-N}$, we compute $(\bb T_{1,\lambda} \bb F(\lambda),
\widetilde{\bb \Psi}_{-N,k})$ in two different ways.
Since $\bb T_{1,\lambda} \bb F(\lambda)=\mu(\lambda)\bb
F(\lambda) $, it follows from \eqref{2decomp1} that
\begin{equation}\label{2decomp5}
\begin{aligned}
(\bb T_{1,\lambda} \bb F(\lambda), \widetilde{\bb
\Psi}_{-N,k})=\mu_{-N}(\lambda+N)\|\widetilde{\bb
\Psi}_{-N,k}\|^2 +O((\lambda+N)^2).
\end{aligned}
\end{equation}
By $\bb T_{1,-N}\widetilde{\bb \Psi}_{-N,k}=\bb 0$,  we obtain
\begin{equation}\label{produ2}
\bb T_{1,\lambda}\widetilde{\bb
\Psi}_{-N,k}=\left(-2x\tfrac{N+\lambda}{\lambda}+\tfrac{N^2-\lambda^2}{\lambda^2}\right)\widetilde{\bb
\Psi}_{-N,k}.
 \end{equation}
Since $\bb T_{1,\lambda}$ is self-adjoint, we obtain from
\eqref{2decomp1} and \eqref{produ2}
\begin{align}
\label{pro3} 
&(\bb T_{1,\lambda} \bb F(\lambda), \widetilde{\bb
\Psi}_{-N,k})=(\bb F(\lambda), \bb T_{1,\lambda}\widetilde{\bb
\Psi}_{-N,k})\\&=\left(\widetilde{\bb
\Psi}_{-N,k},\left[-2x\tfrac{N+\lambda}{\lambda}+\tfrac{N^2-\lambda^2}{\lambda^2}\right]\widetilde{\bb
\Psi}_{-N,k}\right)+ O((\lambda+N)^2).
\notag
\end{align}
Combination of \eqref{2decomp5} and \eqref{pro3} leads to 
\begin{equation}\label{2decomp10}
\mu_{-N}\|\widetilde{\bb \Psi}_{-N,k}\|^2=\left(\widetilde{\bb
\Psi}_{-N,k},\left[-\tfrac{2}{\lambda}x+\tfrac{N-\lambda}{\lambda^2}\right]\widetilde{\bb
\Psi}_{-N,k}\right)+ O(\lambda+N).
\end{equation}
Define $g(\lambda):= \left(\widetilde{\bb
\Psi}_{-N,k},\left[-\tfrac{2}{\lambda}x+\tfrac{N-\lambda}{\lambda^2}\right]\widetilde{\bb
\Psi}_{-N,k}\right)$, then
$$
g(\lambda)= \tfrac{(N-k)N}{k} \int _0^\infty
\left[-\tfrac{2}{\lambda}x+\tfrac{N-\lambda}{\lambda^2}\right](\psi'_{-N,\delta'})^2dx.
$$
By Taylor's theorem, $g(\lambda)=g(-N)+
g'(-N)(\lambda+N)+O((\lambda+N)^2)$. It is easily seen that
$$g(-N)=2e^{\omega+1}\tfrac{N-k}{k} \int
_0^\infty(x+1)^3e^{-(x+1)^2}dx>0.$$
From \eqref{2decomp10}, we get  
\begin{equation*}\label{decomp11}
\mu_{-N}=\frac{g(\lambda)}{\|\widetilde{\bb
\Psi}_{-N,k}\|^2}+O(\lambda+N)=\frac{g(-N)}{\|\widetilde{\bb
\Psi}_{-N,k}\|^2}+O(\lambda+N),
\end{equation*}
and consequently  $\mu_{-N}>0$ for $\lambda$ close to $-N$.

Let $\lambda$ be  close to $-N$ and $\lambda<-N$, 
then from item $(iii)$ and the analysis above 
($\mu(\lambda)<0$) it follows that  
$n(\bb T_{1,\lambda}|_{L^2_k(\mathcal{G})})=2$. 
Finally, by the continuation argument 
(see  item $(iii)$ of Proposition \ref{n(L_1)}), 
we extend the former property for all $\lambda<-N$.

$(iv)$
To prove the last spectral property it is sufficient 
to note that the spectrum of $\mathbf{T}_{1,\lambda}$ and 
$\mathbf{T}_{2,\lambda}$  is discrete due to the growth 
of $q(x)=(x-\tfrac{N}{\lambda})^2$ as $x\to+\infty.$
\end{proof}

\noindent\textbf{Proof of Theorem \ref{Main_log}.}
Combining Theorem \ref{well_log_graph'}, 
Proposition \ref{spec_T_lambda}, 
Theorem \ref{stabil_graph_log} 
(adapted to the case of the NLS-log-$\delta'$ equation),  
we get orbital stability of  
$e^{i\omega t}\bb\Psi_{\lambda,\delta'}$  in 
$W^1(\mathcal{G})$ for $-N<\lambda<0$.  
Spectral instability of $e^{i\omega t}\bb\Psi_{\lambda,\delta'}$ follows  for $\lambda<-N$.
\qed

\section{Applications to  other models}\label{line}
In the above sections the use of the extension theory of
symmetric operators was essential for the estimates of the Morse
index of the specific self-adjoint Schr\"odinger operators. In
this Section we show how this approach can be applied to the case
of the nonlinear Schr\"odinger equations with specific point
interactions on the line. In particular, we reprove in concise
form (avoiding the use of variational techniques) some stability
results for these equations established recently by the other
authors (see \cite{AdaNoj13a, CozFuk08, FukOht08, ghw}).

  \subsection{The NLS with point  interactions on the line }

In the scalar case the family of self-adjoint boundary conditions
for \eqref{NLS0} at $x=0$ is formally defined by
\begin{equation}\label{bc1}
\left(\begin{array}{c}\psi(0+) \\
\psi'(0+)\end{array}\right)=\tau
\left(\begin{array}{cc} a & b\ \\c &
d\end{array}\right)\left(\begin{array}{c} \psi(0-) \\
\psi'(0-)\end{array}\right),
\end{equation}
with $a,b,c,d$ and $\tau$ satisfying the conditions (see
\cite[Theorem 3.2.3]{ak} or formula (K.1.2) from \cite[Appendix
K]{AlbGes05})
\begin{equation}\label{para}
\{a,b,c,d\in \mathbb R,\ \tau\in \mathbb C: ad-bc=1,\ |\tau
|=1\}.
\end{equation}
Parameters \eqref{bc1}
label the self-adjoint extensions of the closable symmetric
operator $H_0=-\frac{d^2}{dx^2}$ defined, for instance, on the
space $C_0^{\infty}(\mathbb R\setminus\{0\})$.

We are interested in two specific choices of the parameters in
\eqref{para}, which are relevant in physical applications (see
\cite{AdaNoj13a,CMR}). The first choice $\tau=a=d=1$, $b=0$,
$c=-\gamma, \gamma\in\mathbb{R}\setminus\{0\}$ corresponds to the
$\delta$-interaction of strength $-\gamma$ which gives rise to
the following NLS-$\delta$ model
\begin{equation}\label{NLS}
i\partial_t u-H_\gamma^\delta u + |u|^{p-1}u=0.
\end{equation}
Here $H_\gamma^\delta$ is the self-adjoint operator on
$L^2(\mathbb{R})$ acting as  $(H_\gamma^\delta v)(x)=-v''(x),$ for $x\neq 0,$ on the domain
$\dom(H_\gamma^\delta)=D_{\gamma,\delta}$, where
\begin{align*}\label{H_delta} 
D_{\gamma,\delta}:=\big \{v\in H^1(\mathbb R)\cap H^2(\mathbb
R\setminus\{0\}): v'(0+)-v'(0-)=-\gamma v(0)\big \}.
\end{align*}
The operator $H_\gamma^\delta$ is formally defined by the
expression $l_\gamma^\delta=-\frac{d^2}{dx^2}-\gamma \delta(x)$,
where $\delta(x)$ is the Dirac delta distribution.

The second choice of parameters $\tau=a=d=1$, $c=0$, $b=-\beta,
\beta\in\mathbb{R}\setminus\{0\}$ corresponds to the case of
so-called $\delta'$-interaction of strength $-\beta$. It gives
rise to the following model (NLS-$\delta'$ henceforth)
\begin{equation}\label{NLS_delta'}
i\partial_t u-H_\beta^{\delta'} u +|u|^{p-1}u=0,
\end{equation}
in which $H_\beta^{\delta'}$ is the self-adjoint operator on
$L^2(\mathbb{R})$ acting as  $(H_\beta^{\delta'} v)(x)=-v''(x),$ for  $x\neq 0,$ on the domain $\dom(H_\beta^{\delta'})=D_{\beta,
\delta'}$, where
\begin{align*}\label{H_delta'} 
D_{\beta, \delta'}:=\{v\in H^2(\mathbb R\setminus\{0\}):
v(0+)-v(0-)=-\beta v'(0),\,\,v'(0+)=v'(0-)\}.
\end{align*}
Recall that $H_\beta^{\delta'}$ is formally defined by the
expression $l_\beta^{\delta'}=-\frac{d^2}{dx^2}-\beta
\langle\cdot,\delta'\rangle\delta'(x)$.

The NLS-$\delta$ model has been extensively studied in the last
decade (see \cite{Ang, AP, CM, CMR, CozFuk08, FukJea08, FukOht08,
ghw, HMZ1,oh} and reference therein). The NLS-$\delta'$ model is less
studied, in \cite{AdaNoj13a, AdaNoj13} the authors investigated
variational properties and the orbital stability of the ground
states of the NLS-$\delta'$ equation with the repulsive $\delta'$-interaction
($\beta>0$).

\subsection{The NLS-$\delta'$ equation on the line}

As above the existence of standing wave solutions
$u(t,x)=e^{i\omega t}\varphi_(x)$ of equation \eqref{NLS_delta'}
requires that the profile $\varphi\in D_{\beta, \delta'}$
satisfies the semi-linear elliptic equation
\begin{equation}\label{ellip_delta'}
H_\beta^{\delta'}\varphi+\omega\varphi-|\varphi|^{p-1}\varphi=0.
\end{equation}
It was shown in \cite{AdaNoj13a} that for $\beta>0$ equation
\eqref{ellip_delta'} has two types of solutions (odd and
asymmetric)
 \begin{equation}\label{varphi1}
\varphi_{\omega,\beta}^{odd}(x)=\sign(x)
 \Big [\frac{(p+1)\omega}{2}
\sech^2 \Big (\frac{(p-1)\sqrt{\omega}}{2}(|x|+y) \Big ) \Big ]^{\frac{1}{p-1}},
 \end{equation}
 with $x\neq 0$ and $\omega>\tfrac4{\beta^2}$,
  \begin{gather*}
\varphi_{\omega,\beta}^{as}(x)= \left\{
                    \begin{array}{ll}
\Big[\frac{(p+1)\omega}{2}
\sech^2\Big(\frac{(p-1)\sqrt{\omega}}{2}(x+y_1)\Big)\Big]^{\frac{1}{p-1}},
& \hbox{$x>0;$} \\
- \Big[\frac{(p+1)\omega}{2}
\sech^2\Big(\frac{(p-1)\sqrt{\omega}}{2}(x-y_2)\Big)\Big]^{\frac{1}{p-1}},
& \hbox{$x<0$,}
                    \end{array}
\right., \omega>\tfrac4{\beta^2}\tfrac{p+1}{p-1},\label{varphi2}                  \end{gather*}              
where $y,\ y_1$ and $y_2$ are positive constants depending on
$\beta, p, \omega$ (see \cite[Theorem 5.3]{AdaNoj13a}). Moreover,
in \cite{AdaNoj13a, AdaNoj13} were established the following
stability results.
The standing wave $e^{i\omega t}\varphi^{odd}_{\omega,\beta}$ is
stable in $H^1(\mathbb R\setminus\{0\})$ for
$p>1$, $\omega\in (\tfrac4{\beta^2},\tfrac4{\beta^2}\tfrac{p+1}{p-1} )$,
and unstable in $H^1(\mathbb R\setminus\{0\})$ for $p>1$, 
$\omega>\tfrac4{\beta^2}\tfrac{p+1}{p-1}$. The
standing wave $e^{i\omega t}\varphi^{as}_{\omega,\beta}$ is
stable in $H^1(\mathbb R\setminus\{0\})$ for $1<p\leq 5$, 
$\omega>\tfrac4{\beta^2}\tfrac{p+1}{p-1}$, and $p>5$, 
$\omega\in (\tfrac4{\beta^2}\tfrac{p+1}{p-1},
\omega_1 )$, meanwhile $e^{i\omega
t}\varphi^{as}_{\omega,\beta}$ is unstable in $H^1(\mathbb
R\setminus\{0\})$ for $p>5$,  $\omega>\omega_2>\omega_1.$

In what follows, we will use the notation
$\varphi_{\beta}=\varphi^{odd}_{\omega,\beta}$. Due to Grillakis,
Shatah and Strauss approach, we need to study the spectral
properties of the following two self-adjoint operators
\begin{align*} \label{linear_delta'} 
&L_{1,\beta}=-\frac{d^2}{dx^2}+\omega-p|\varphi_{\beta}|^{p-1},\;\;\;\;
L_{2,\beta}=-\frac{d^2}{dx^2}+\omega-|\varphi_{\beta}|^{p-1},\\
&\dom(L_{j,\beta})=D_{\beta, \delta'},\,\,\, j\in\{1,2\}.
\end{align*}
The operators $L_{1,\beta}$ and $L_{2,\beta}$ are associated with
the action functional
$$
S_\beta(\psi)=\tfrac 1{2}||\psi'||^2+\tfrac{\omega}{2}||\psi||^2
-\tfrac
1{p+1}||\psi||_{p+1}^{p+1}-\tfrac{1}{2\beta}|\psi(0+)-\psi(0-)|^2,$$
defined on  $H^1(\mathbb R\setminus\{0\})$,
in the following sense:
$$
(S_\beta)''(\varphi_{\beta})(u,v)=(L_{1,\beta}u_1,
v_1)+(L_{2,\beta}u_2, v_2),$$
where $u=u_1+iu_2$ and $v=v_1+iv_2$. 

The well-posedness for
\eqref{NLS_delta'} in $H^1(\mathbb R\setminus\{0\})$ was
established in \cite[Proposition 3.3]{AdaNoj13a}. Moreover, it
was shown that $\ker(L_{2,\beta})=\Span\{\varphi_{\beta}\}$, and
$\ker(L_{1,\beta})=\{0\}$, and the sign of
$\partial_\omega||\varphi_{\beta}||^2$ was computed.

The following result on the Morse index of $L_{1,\beta}$ was
proved in \cite{AdaNoj13a} via variational approach. We propose
an alternative proof in the framework of the extension theory.

\begin{proposition}\label{neg_delta'}
Let $\beta>0$  and $\omega>\frac{4}{\beta^2}$. 
Then 
\begin{itemize}
\item[$(i)$]
$n(L_{1,\beta})=1$ for $\omega \in (\frac{4}{\beta^2},
\tfrac4{\beta^2}\tfrac{p+1}{p-1} ]$,
\item[$(ii)$] $n(L_{1,\beta})=2$ for $\omega
\in (\tfrac4{\beta^2}\tfrac{p+1}{p-1}, \infty)$.
\end{itemize} 
\end{proposition}

\begin{proof}[\bf Proof.]

It is easily seen that $L_{1,\beta}$ is the self-adjoint
extension of the symmetric operator $L_{\min}$ defined by
\begin{equation}\label{L_min'}
L_{\min}=-\frac{d^2}{dx^2}+\omega-p|\varphi_{\beta}|^{p-1},\quad\dom(L_{\min})=\{v\in H^2(\mathbb{R}): v(0)=v'(0)=0\}.
\end{equation}
Since $\varphi_{\beta}\in L^{\infty}(\mathbb{R})$, we obtain
$\dom(L_{\min}^*)=H^2(\mathbb{R}\setminus\{0\})$.
Moreover, the operator $L_{\min}$ is non-negative for $\beta>0$.
Indeed, it is easy to verify that for $\beta>0$ and $v\in
H^2(\mathbb{R}\setminus \{0\})$ the following identity holds
\begin{equation}\label{identity'}
-v''+\omega v-p|\varphi_{\beta}|^{p-1}v=
\frac{-1}{\varphi'_{\beta}}\frac{d}{dx} \Big [(\varphi'_{\beta})^2\frac{d}{dx} \Big (\frac{v}{\varphi'_{\beta}} \Big ) \Big ],\quad
x\neq 0.
\end{equation}
 Using \eqref{identity'} and integrating by parts, we get
 \begin{align}
  \label{nonneg1'} 
(L_{\min} v,v)=
\Big(\int
_{-\infty}^{0-}+ 
\int
^{\infty}_{0+}\Big) (\varphi'_{\beta})^2 \Big [\frac{d}{dx} \Big (\frac{v}{\varphi'_{\beta}} \Big ) \Big ]^2dx +
 \Big [v'{v}-v^2\frac{\varphi''_{\beta}}{\varphi'_{\beta}} \Big ]_{0-}^{0+}.
\end{align}
The integral terms in \eqref{nonneg1'} are non-negative. Due to
the conditions $v(0)=v'(0)=0$, non-integral term vanishes, and we
get $L_{\min}\geq 0$. Note that
$$
\dom(L^*_{\min})=H^2(\mathbb{R}\setminus
\{0\})=\dom(L_{\min})\oplus\Span\{v_i^1,v_i^2\}\oplus\Span\{v_{-i}^1,v_{-i}^2\},
$$
where
\begin{equation*}
v^1_{\pm i}=\left\{
                     \begin{array}{ll}
                       e^{i\sqrt{\pm i}x} & \hbox{$x>0$;} \\
                       0 & \hbox{$x<0$.}
                     \end{array}
                   \right.,\quad v^2_{\pm i}=\left\{
                     \begin{array}{ll}
                            0 & \hbox{$x>0$;}\\
e^{-i\sqrt{\pm i}x} & \hbox{$x<0$.}
                     \end{array}
\right. , \quad \Im(\sqrt{\pm i})>0.\end{equation*}Indeed, due to the fact that $\varphi_{\beta}\in
L^\infty(\mathbb{R})$, we get $\dom(L^*_{\min})=\dom(L^*),$
where
 $$L=-\frac{d^2}{dx^2},\quad \dom(L)=\dom(L_{\min}).$$
Moreover, $n_\pm(L_{\min})=n_\pm(L)=2.$ 
Since $L_{1,\beta}$ is the self-adjoint extension of the
non-negative symmetric operator $L_{\min}$ and
$n_{\pm}(L_{\min})=2$, by Proposition \ref{semibounded},
$n(L_{1,\beta})\leq 2$.
Otherwise, we obtain from \eqref{ellip_delta'} that
$(L_{1,\beta}\varphi_{\beta},\varphi_{\beta})<0$, and therefore
$n(L_{1,\beta})\geq 1$. Thus, we get $1\leq n(L_{1,\beta})\leq
2$.

 $(i)$ 
Note that $L_{1,\beta}$ is the self-adjoint extension of the
following symmetric operator
\begin{equation*} 
 L'_{0}=-\frac{d^2}{dx^2}+\omega-p|\varphi_{\beta}|^{p-1},\quad
\dom(L'_{0})=\left\{v\in H^2(\mathbb{R}):\, v'(0)=0\right\}.
\end{equation*}

Let us show that $L'_{0}\geq 0$. Using \eqref{identity'} and
integrating by parts,
\begin{align}
\label{nonneg} 
(L'_0v,v)&=
\Big(\int
_{-\infty}^{0-}+\int
^{\infty}_{0+}\Big) (\varphi'_{\beta})^2 \Big [\frac{d}{dx} \Big (\frac{v}{\varphi'_{\beta}} \Big ) \Big ]^2dx+
 \Big [v'{v}-v^2\frac{\varphi''_{\beta}}{\varphi'_{\beta}} \Big ]_{0-}^{0+}.
\end{align}

The integral terms in \eqref{nonneg} are non-negative. Let us
focus on the non-integral term. Due to the conditions $v'(0)=0,
v(0+)=v(0-)$, and formula \eqref{varphi1}, we deduce
\begin{align*}  
 \Big [v'{v}-v^2\frac{\varphi''_{\beta}}{\varphi'_{\beta}} \Big ]_{0-}^{0+}&=
- \Big [v^2\frac{\varphi''_{\beta}}{\varphi'_{\beta}} \Big ]_{0-}^{0+}=v^2(0)\frac{\varphi''_{\beta}(0-)-\varphi''_{\beta}(0+)}{\varphi'_{\beta}(0-)}\\
&=-v^2(0)\frac{\beta\omega}{2} \Big (p-1-(p+1)\frac{4}{\beta^2\omega} \Big )\geq0. 
\end{align*}
The last inequality follows from
$\omega\leq\tfrac4{\beta^2}\tfrac{p+1}{p-1}$.

Using arguments numerously repeated above, we get
 $$\dom((L'_{0})^*)=\left\{v\in H^2(\mathbb{R}\setminus\{0\}):\,
v'(0+)=v'(0-)\right\}=\dom(L'_{0})\oplus\Span\{v_i\}\oplus\Span\{v_{-i}\},$$
where 
$$ 
v_{\pm i}=\left\{
                          \begin{array}{ll}
e^{i\sqrt{\pm i}x} & \hbox{$x>0$,} \\- e^{-i\sqrt{\pm i}x} & \hbox{$x<0$,}                          \end{array}
                        \right., \quad \Im(\sqrt{\pm i})>0.
$$                       
Then $n_\pm(L'_0)=1$, and by Proposition \ref{semibounded}, we obtain $n(L_{1,\beta})\leq 1$,
and finally $n(L_{1,\beta})=1$.

 $(ii)$  
The quadratic form of the operator $L_{1,\beta}$ is defined on
$H^1(\mathbb R\setminus\{0\})$ by $$
F_{1,\beta}(u)=||u'||^2+\omega||u||^2-p(|\varphi_{\beta}|^{p-1}u,u)-\tfrac
1{\beta}|u(0+)-u(0-)|^2.$$
Noting that $\varphi'_{\beta}(0+)=\varphi'_{\beta}(0-)$ and
integrating by parts, we get
\begin{align*} 
&F_{1,\beta}(\varphi'_{\beta})= \left(\int
_{-\infty}^{0-}+ \int _{0+}^{+\infty}\right)\varphi'_{\beta} \Big(-{\varphi}'''_{\beta}+\omega
\varphi'_{\beta}-p |\varphi_{\beta}|^{p-1}\varphi'_{\beta}\Big)dx
\\
&+\varphi'_{\beta}(0+)(\varphi''_{\beta}(0-)-\varphi''_{\beta}(0+))=\varphi'_{\beta}(0+)(\varphi''_{\beta}(0-)-\varphi''_{\beta}(0+))\\&=-\tfrac
{2}{\beta}\omega\left[\left(\tfrac{(p+1)\omega}{2}\right)\left(1-\tfrac {4}{\beta^2\omega}\right)\right]^{\tfrac
2{p-1}}\left(p-1-(p+1)\tfrac 4{\beta^2\omega}\right).
\end{align*}

The last one expression is negative due to
$\omega>\tfrac4{\beta^2}\tfrac{p+1}{p-1}$. Since
$F_{1,\beta}(\varphi_{\beta})=\left(L_{1,\beta} \varphi_{\beta},
\varphi_{\beta}\right)<0$, and the functions
$\varphi_{\beta},\varphi'_{\beta}$ have different parity, we
obtain
$ 
F_{1,\beta}(s\varphi_{\beta}+
r\varphi_{\beta}')=|s|^2F_{1,\beta}(\varphi_{\beta})+|r|^2F_{1,\beta}(\varphi_{\beta}')<0.
$ 
Therefore, we have that $F_{1,\beta}$ is negative on
two-dimensional subspace $\mathcal M=\Span\{\varphi_{\beta},
\varphi'_{\beta}\}\subset H^1(\mathbb R\setminus\{0\})$. Thus,
minimax principle induces $n(L_{1,\beta})\geq 2$, and
consequently $n(L_{1,\beta})=2$.
 \end{proof}

In \cite[Proposition 6.5]{AdaNoj13a} it was shown that
$\partial_\omega||\varphi_{\beta}||^2$ is positive for any $p>1$
and $\omega \in\left(\frac{4}{\beta^2},
\tfrac4{\beta^2}\tfrac{p+1}{p-1}\right)$. Thus, due to
Proposition \ref{neg_delta'}, we conclude that $e^{i\omega
t}\varphi_{\beta}$ is orbitally stable in this case.

Below we briefly discuss how to demonstrate the orbital instability of
$e^{i\omega t}\varphi_{\beta}$ for $p>1$ and
$\omega>\tfrac4{\beta^2}\tfrac{p+1}{p-1}$ proved in \cite[Theorem
6.11]{AdaNoj13a}.
To do that we need the following key result.
\begin{proposition}\label{morse_odd} Let
$\omega>\frac{4}{\beta^2},\,\beta>0$, and operator
$\widetilde{L}_{1,\beta}$ be defined as
 $$
\widetilde{L}_{1,\beta}=-\frac{d^2}{dx^2}+\omega-p|\varphi_{\beta}|^{p-1},\quad\dom(\widetilde{L}_{1,\beta})=D_{\beta,
\delta'}\cap X_{\odd},
 $$ 
where $X_{\odd}$ is the set of odd functions in
$L^2(\mathbb{R})$.
 Then $n(\widetilde{L}_{1,\beta})=1$.  
\end{proposition}
 
\begin{proof}[\bf Proof.]
It is obvious that  
$n(\widetilde{L}_{1,\beta})\leq 1$ in 
$X_{\odd}$. Indeed, 
$n_\pm(L_{\min})$ $=1$ in $X_{\odd}$ for 
$L_{\min}$ defined by \eqref{L_min'}. 
Since $\varphi_{\beta}\in \dom(\widetilde{L}_{1,\beta})$   
and $(\widetilde{L}_{1,\beta}\varphi_{\beta},\varphi_{\beta})$
$<0$,
then we get $n(\widetilde{L}_{1,\beta})=1$.
  \end{proof}
  
  Well-posedness of the Cauchy problem in 
$H^1(\mathbb R\setminus\{0\})\cap X_{\odd}$ associated with
equation \eqref{NLS_delta'} was shown in \cite[Theorem
6.11]{AdaNoj13a}.
Thus, we induce orbital instability of $e^{i\omega t}\varphi_{\beta}$ for
$p>1$ and $\omega>\tfrac4{\beta^2}\tfrac{p+1}{p-1}$.
Indeed, when $\partial_\omega||\varphi_{\beta}||^2>0$,
instability follows from Proposition \ref{neg_delta'}-$(ii)$ and from the results by Ohta in \cite{oh}. In
the case $\partial_\omega||\varphi_{\beta}||^2<0$ we can conclude
by Proposition \ref{morse_odd} orbital instability of $e^{i\omega
t}\varphi_{\beta}$ in $H^1(\mathbb R\setminus\{0\})\cap X_{\odd}$
which naturally induces orbital instability in $H^1(\mathbb
R\setminus\{0\})$.

\subsection{ The NLS-$\delta$   equation on the line}

The existence of standing wave solutions $u(t,x)=e^{i\omega t}
\varphi$ to equation \eqref{NLS} requires that the profile
$\varphi\in D_{\gamma,\delta}$ satisfies the semi-linear elliptic
equation
\begin{equation}\label{ellip}
H_\gamma^\delta\varphi+\omega\varphi-|\varphi|^{p-1}\varphi=0.
\end{equation}
The authors in \cite{FukJea08} (see also \cite{ghw})
showed that \eqref{ellip} for $\omega>\tfrac{\gamma^2}{4}$ has a
unique positive even solution modulo rotation
\begin{equation}\label{ellip1}
\varphi_\gamma(x)= \Big [\frac{(p+1)\omega}{2}
\sech^2 \Big (\frac{(p-1)\sqrt{\omega}}{2}|x|+\tanh^{-1}\Big(\frac{\gamma}{2\sqrt{\omega}}\Big) \Big ) \Big ]^{\frac{1}{p-1}},\quad
x\in\mathbb R.
\end{equation}
For the sake of completeness, we recall the main results on the
stability of soliton solutions to \eqref{NLS}. For $\gamma = 0$, the
orbital stability has been extensively studied in
\cite{{BerCaz81}, Caz82, Caz89, Wei83}. Namely, $e^{i\omega
t}\varphi_{0}$ is stable in $H^1(\mathbb{R})$ for any $\omega >
0$ and $1 < p < 5$ (see \cite{Caz82}), and unstable in
$H^1(\mathbb{R})$ for any $\omega > 0$ and $p \geq 5$ (see
\cite{BerCaz81} for $p > 5$ and \cite{Wei83} for $p = 5$).

The case $\gamma > 0$ was studied in \cite{FukOht08}. In
particular, the authors showed that the standing wave $e^{i\omega
t}\varphi_{\gamma}$ is stable in $H^1(\mathbb{R})$ for any
$\omega>\tfrac{\gamma^2}{4}$ and $1 < p \leq 5$, and if $p > 5$,
there exists a critical $\omega^*$ such that $e^{i\omega
t}\varphi_{\gamma}$ is stable in $H^1(\mathbb{R})$ for any
$\omega\in (\tfrac{\gamma^2}{4},\omega^* )$ and unstable
in $H^1(\mathbb{R})$ for any $\omega> \omega^*$.
In the case $\gamma<0$, the standing wave $e^{i\omega
t}\varphi_{\gamma}$ is unstable "almost for sure" in
$H^1(\mathbb{R})$ for any $p>1$ (see \cite{CozFuk08, FukJea08,
oh}).

Linearization of the NLS-$\delta$ equation on the line gives the
following two self-adjoint linear operators
\begin{equation*}\label{linear}
L_{1,\gamma}=-\frac{d^2}{dx^2}+\omega-p\varphi^{p-1}_{\gamma},
\quad
L_{2,\gamma}=-\frac{d^2}{dx^2}+\omega-\varphi^{p-1}_{\gamma},
\end{equation*}
with  $\dom(L_{j,\gamma})=D_{\gamma, \delta},\, j\in\{1,2\}$.
The operators $L_{1,\gamma}$ and $L_{2,\gamma}$ are associated
with the key action functional
\begin{equation*}
S_\gamma(\psi)=\tfrac 1{2}||\psi'||^2+\tfrac{\omega}{2}||\psi||^2
-\tfrac 1{p+1}||\psi||_{p+1}^{p+1}-\tfrac{\gamma}{2}|\psi(0)|^2,\,\,\,\psi\in H^1(\mathbb R),
\end{equation*}
by $
(S_\gamma)''(\varphi_{\gamma})(u,v)=(L_{1,\gamma}u_1,
v_1)+(L_{2,\gamma}u_2, v_2),$
where $u=u_1+iu_2$ and $v=v_1+iv_2$.  

The initial value problem associated to the NLS-$\delta$ equation
is locally well-posed in $H^1(\mathbb R)$ (see \cite[Theorem
4.6.1]{Caz89}) for any $p>1$. Making use of the explicit form
\eqref{ellip1} for $\varphi_{\gamma}$, the sign of
$\partial_\omega||\varphi_{\gamma}||^2$
was computed in \cite{FukJea08, FukOht08}. By variational
methods, it was shown in \cite{FukJea08} that $n(
L_{1,\gamma})=1$ in $H_{\rad}^1(\mathbb{R})$, for arbitrary
$\gamma$. Moreover, by using analytic perturbation theory and
continuation argument, it was shown in \cite{CozFuk08} that
$n(L_{1,\gamma})=1$ in $H^1(\mathbb{R})$ for any $\gamma>0$, as
well as $n(L_{1,\gamma})=2$ for $\gamma<0$.

Below we establish two novel proofs of the equality
$n(L_{1,\gamma})=1$ in $H^1(\mathbb{R})$ for any $\gamma>0$.
The first one is based on a generalization of the classical Sturm
oscillation theorem to the case of the $\delta$-interaction (see
\cite{Arm05, BerShu91} and Lemma \ref{Oscill} below). The second
one uses the extension theory. Note also that the equality
$\ker(L_{2,\gamma})=\Span\{\varphi_{\gamma}\}$ and Lemma
\ref{Oscill} imply $n(L_{2,\gamma})=0$.

\begin{lemma}\label{Oscill}
Let $V(x)$ be real-valued continuous function on $\mathbb{R}$ 
such that $\lim\limits_{|x|\rightarrow \infty}V(x)=c$.
Let also $\varphi_1, \varphi_2\in L^2(\mathbb{R})$ be
eigenfunctions of the operator
\begin{equation*}
L_V=-\frac{d^2}{dx^2}+V(x),\quad
\dom(L_V)=D_{\gamma, \delta},
\end{equation*}
corresponding to the eigenvalues $\lambda_1<\lambda_2<c$
respectively. Suppose that $n_1$ and $n_2$ are the number of
zeroes of $\varphi_1, \varphi_2$ respectively. Then $n_2>n_1$.
\end{lemma}

\begin{proposition}\label{num} Let $\gamma>0$ and
$\omega>\frac{\gamma^2}{4}$. Then $n(L_{1,\gamma})=1$.
\end{proposition}

\noindent\textbf{The first proof of Proposition \ref{num}. }
Initially we obtain from \eqref{ellip} that
$(L_{1,\gamma}\varphi_{\gamma},\varphi_{\gamma})<0$, and
therefore $n(L_{1,\gamma})\geq 1$. To evaluate $n(L_{1,\gamma})$
precisely consider the following self-adjoint operator
\begin{equation*}\label{L1'}
\widetilde{L}_{1,\gamma}=-\frac{d^2}{dx^2}+\omega-p\varphi^{p-1}_{0},\quad\quad
\dom(\widetilde{L}_{1,\gamma})=D_{\gamma, \delta},
\end{equation*}
where $\varphi_{0}= [\frac{(p+1)\omega}{2}
\sech^2 (\frac{(p-1)\sqrt{\omega}}{2}x ) ]^{\frac{1}{p-1}}$
is the classical soliton solution for the NLS equation.
It is easily seen that $\varphi'_{0}\in
\ker(\widetilde{L}_{1,\gamma})$. From Lemma \ref{Oscill} and the
fact that $x=0$ is the only zero of $\varphi'_{0}$ we have
$n(\widetilde{L}_{1,\gamma})\leq 1$.
Since $\varphi_{0}(x)>\varphi_{\gamma}(x)$ for all $x\in\mathbb
R$ and $\gamma>0$, we get the following inequality
\begin{equation*}\label{ineq}
(L_{1,\gamma}v,v)\geq(\widetilde{L}_{1,\gamma}v,v),\quad
\text{for all}\;\;v\in D_{\gamma, \delta}.
\end{equation*} 
Therefore,  we get $
1\leq n(L_{1,\gamma})\leq n(\widetilde{L}_{1,\gamma})\leq 1$.
Thereby, in the case $\gamma>0$ we get $n(L_{1,\gamma})=1$.
\qed

\smallskip

\noindent 
\textbf{The second proof of Proposition \ref{num}.}  
Recall that $L_{1,\gamma}$ is the self-adjoint 
extension of the following symmetric operator
\begin{equation*} 
 L_{0}=-\frac{d^2}{dx^2}+\omega-p\varphi^{p-1}_{\gamma},\quad
\dom(L_{0})=\left\{v\in H^2(\mathbb{R}):\, v(0)=0\right\}.
\end{equation*}
Moreover, it is known (see \cite[Chapter I.3]{AlbGes05}) that   
$$
\dom(L_0^*)=H^1(\mathbb{R})\cap
H^2(\mathbb{R}\setminus\{0\})=\dom(L_0)\oplus\Span\{e^{i\sqrt{i}|x|}\}\oplus\Span\{e^{i\sqrt{-i}|x|}\},
$$
with $\Im(\sqrt{\pm i})>0$. Indeed, since $\varphi_{\gamma}\in
L^{\infty}(\mathbb{R})$, we have $\dom(L_0^*)=\dom(L^*)$, where
$ L=-\frac{d^2}{dx^2},
$
$
 \dom(L)=\dom(L_0). $
In particular, $n_\pm(L_0)=n_\pm(L)=1.$
Next, it is easy to verify that for $\gamma>0$ and $v\in
H^2(\mathbb{R}\setminus \{0\})$ the following identity holds
\begin{equation}\label{identity1}
-v''+\omega v-p\varphi_{\gamma}^{p-1}v=
\frac{-1}{\varphi'_{\gamma}}\frac{d}{dx} \Big [(\varphi'_{\gamma})^2\frac{d}{dx} \Big (\frac{v}{\varphi'_{\gamma}} \Big ) \Big ],\quad\quad
x\neq 0.
\end{equation}
 Then,  using \eqref{identity1} and integrating by parts, we get
\begin{align}
\label{nonneg1} 
(L_{0} v,v)&=
\Big( \int
_{-\infty}^{0-}+ \int
^{\infty}_{0+}\Big) (\varphi'_{\gamma})^2 \Big [\frac{d}{dx} \Big (\frac{v}{\varphi'_{\gamma}} \Big ) \Big ]^2dx+
 \Big [v'v-v^2\frac{\varphi''_{\gamma}}{\varphi'_{\gamma}} \Big ]_{0-}^{0+}.
\notag
\end{align}

 Due to
the condition $v(0)=0$, non-integral term vanishes, and we get
$L_{0}\geq 0$ on $\dom(L_0)$.
Then, using Proposition~\ref{semibounded} we get
$n(L^\gamma_{1,\omega})\leq 1$. This finishes the proof due to
the inequality
$(L_{1,\gamma}\varphi_{\gamma},\varphi_{\gamma})<0$.
\qed

\smallskip

\noindent
{\bf Acknowledgements.} 
The authors are grateful to the  anonymous referee for  
valuable suggestions and  constructive comments 
which helped to improve considerably  the  
manuscript. J. Angulo Pava was supported 
in part by   CNPq/Brazil Grant and by 
FAPESP/Brazil (S\~ao Paulo Research Fundation) 
under the project 2016/07311-0. N. Goloshchapova 
was supported by FAPESP/Brazil under the 
projects 2012/50503-6 and 2016/02060-9.

\end{document}